\documentclass[11pt]{article}
\usepackage{amsmath,amssymb,amsthm,amsfonts}
\usepackage{mathptmx}
\usepackage{enumerate}
\usepackage[pdftitle={N. J. Kalton, S. Lord, D. Potapov, F. Sukochev -- Traces on compact operators and the noncommutative residue},
            pdfauthor={N. J. Kalton, S. Lord, D. Potapov, F. Sukochev},
            pdfsubject={2000MSC 47B10, 58B34, 58J42, 47G10},
            pdfkeywords={Noncommutative residue, commutator subspace, singular trace,Connes' trace theorem,
								spectral theory,Lidskii theorem},
            pdfdisplaydoctitle=true,
            pdfpagelayout=SinglePage]{hyperref}

\numberwithin{equation}{section} 

\theoremstyle{plain}
\newtheorem{theorem}{Theorem}[section]
\newtheorem{proposition}[theorem]{Proposition}
\newtheorem{lemma}[theorem]{Lemma}
\newtheorem{corollary}[theorem]{Corollary}
\newtheorem{definition}[theorem]{Definition}
\newtheorem{remark}[theorem]{Remark}
\newtheorem*{notation}{Notation}
\newtheorem{Thm*}{Theorem*}
\newtheorem*{Def*}{Definition*}
\theoremstyle{definition}
\newtheorem{example}[theorem]{Example}

\newcommand{\mmod}{\ensuremath{\mathrm{mod}}}
\newcommand{\diag}{\ensuremath{\mathrm{diag}}}
\newcommand{\Res}{\ensuremath{\mathrm{Res}}}
\newcommand{\Tr}{\ensuremath{\mathrm{Tr}}}
\newcommand{\Com}{\ensuremath{\mathrm{Com\,}}}

\newcommand{\rank}{\ensuremath{\mathrm{rank}}}
\newcommand{\supp}{\ensuremath{\mathrm{supp\,}}}

\newcommand{\lec}[1][]{\ensuremath{\; \dot{\le}_{#1} \;}}
\newcommand{\gec}[1][]{\ensuremath{\; \dot{\ge}_{#1} \;}}

\begin{document}


\title{Traces of compact operators and the noncommutative residue}

\author{Nigel Kalton\thanks{Nigel Kalton (1946-2010).  The author passed away during production of this paper.} \and 
Steven Lord\thanks{[Corresponding Author] School of Mathematical Sciences, University of Adelaide, Adelaide, 5005, Australia. \texttt{steven.lord@adelaide.edu.au}} \and
Denis Potapov\thanks{School of Mathematics and Statistics, University of New South Wales, Sydney, 2052, Australia. \texttt{d.potapov@unsw.edu.au}} \and
Fedor Sukochev\thanks{School of Mathematics and Statistics, University of New South Wales, Sydney, 2052, Australia. \texttt{f.sukochev@unsw.edu.au}}}

\maketitle

\begin{abstract} 

We extend the noncommutative residue of M.~Wodzicki on
compactly supported classical pseudo-differential operators of order $-d$
and generalise A.~Connes' trace theorem, which states that the residue can be calculated
using a singular trace on compact operators.
Contrary to the role of the noncommutative residue for the classical pseudo-differential operators, a corollary is that the pseudo-differential operators of order $-d$ do not have a `unique' trace;
pseudo-differential operators can be non-measurable in Connes' sense.
Other corollaries are given clarifying the role of Dixmier traces in noncommutative geometry \`{a} la Connes, including the definitive statement of Connes' original theorem. \\

\medskip \noindent {\it Keywords:}
noncommutative residue,
Connes' trace theorem,
Lidskii theorem,
noncommutative geometry,
spectral theory,
singular trace. \\

\medskip \noindent {\it 2000 MSC:}
47B10, 58B34, 58J42, 47G10
\end{abstract}

\section{Introduction}

A.~Connes proved, \cite[Theorem 1]{C3}, that
$$
\mathrm{Tr}_\omega(P) = \frac{1}{d(2\pi)^d} \Res_W(P)
$$
where $P$ is a classical pseudo-differential operator of order $-d$ on a $d$-dimensional closed Riemannian manifold,
$\mathrm{Tr}_\omega$ is a Dixmier trace (a trace on the compact operators with singular values $O(n^{-1})$ which is not an extension of the canonical trace), \cite{Dix},
and $\Res_W$ is the noncommutative residue of M.~Wodzicki, \cite{Wod}.

Connes' trace theorem, as it is known, has become the cornerstone of noncommutative integration \`{a} la Connes.
Applications of Dixmier traces as the substitute
noncommutative residue and integral in non-classical spaces
range from fractals, \cite{Lap1997}, \cite{GI2003}, to foliations \cite{FB}, to spaces of noncommuting co-ordinates, \cite{Connes_GSPV}, \cite{Moyal2004}, and applications in string theory and Yang-Mills, \cite{1126-6708-1998-02-003}, \cite{1126-6708-1998-02-008}, \cite{seiberg-1999-9909}, \cite{C3}, Einstein-Hilbert actions and particle physics' standard model, \cite{C4}, \cite{CC}, \cite{Kastler1995}.  

Connes' trace theorem, though, is not complete.  There are other traces, besides Dixmier traces, on the compact operators whose singular values are $O(n^{-1})$.  Wodzicki
showed the noncommutative residue is essentially the unique trace on classical pseudo-differential operators of order $-d$, so it should be expected that
every suitably normalised trace computes the noncommutative residue.
Also, all pseudo-differential operators have a notion of principal symbol
and Connes' trace theorem opens the question of whether the principal symbol of non-classical operators can be used to compute their Dixmier trace.

We generalise Connes' trace theorem.  There is an extension
of the noncommutative residue that relies only on the principal symbol of a pseudo-differential operator, and that extension calculates the Dixmier trace of the operator.
The following definition and theorem applies to
a much wider class of Hilbert-Schmidt operators, called
Laplacian modulated operators, that we develop in the text.  Here, in the introduction,
we mention only pseudo-differential operators.

A pseudo-differential operator $P : C^\infty(\mathbb R^d) \to C^\infty(\mathbb R^d)$
is compactly based if $Pu$ has compact support for all $u \in C^\infty(\mathbb R^d)$.  Equivalently the (total) symbol of $P$ has compact support in the first variable.

\begin{Def*}[Extension of the noncommutative residue]
Let $P : C_c^\infty(\mathbb R^d) \to C_c^\infty(\mathbb R^d)$ be a compactly based
pseudo-differential operator of order $-d$ with symbol $p$.
The linear map
$$
P \mapsto \Res(P):= \left[ \left\{ \frac{d}{\log(1+n)} \int_{\mathbb R^d} \int_{|\xi|\le n^{1/d}} p(x,\xi) d\xi\,dx \right\}_{n=1}^\infty \right]
$$
we call the \emph{residue} of $P$, where $[ \cdot ]$ denotes the equivalence class in $\ell_\infty / c_0$.
\end{Def*}
Alternatively, any sequence $\Res_n(P)$, $n \in \mathbb N$, such that
\begin{equation*}
\int_{\mathbb R^d} \int_{|\xi|\le n^{1/d}} p(x,\xi) d\xi\,dx = \frac{1}{d} \mathrm{Res}_n(P)  \log n + o(\log n)
\end{equation*}
defines the residue $\Res(P) = [\mathrm{Res}_n(P)] \in \ell_\infty / c_0$.
Note that a dilation invariant state $\omega \in \ell_\infty^*$ vanishes
on $c_0$.  Hence
$$
\omega([c_n]) := \omega( \{ c_n \}_{n=1}^\infty ) , \qquad \{ c_n \}_{n=1}^\infty  \in \ell_\infty
$$
is well-defined as a linear functional on $\ell_\infty / c_0$.  

\begin{Thm*}[Trace theorem]
Let $P : C_c^\infty(\mathbb R^d) \to C_c^\infty(\mathbb R^d)$ be a compactly based
pseudo-differential operator of order $-d$ with residue $\Res(P)$.
Then (the extension) $P: L_2(\mathbb R^d) \to L_2(\mathbb R^d)$ is a compact operator with singular values $O(n^{-1})$ and:
\begin{enumerate}[(i)]
\item
$$
\mathrm{Tr}_\omega(P) = \frac{1}{d(2\pi)^d} \omega(\mathrm{Res}(P));
$$
\item
$$
\mathrm{Tr}_\omega(P) = \frac{1}{d(2\pi)^d} \mathrm{Res}(P)
$$
for every Dixmier trace $\Tr_\omega$ iff $\Res(P)$ is scalar (equivalently
\begin{equation*}
\int_{\mathbb R^d} \int_{|\xi|\le n^{1/d}} p(x,\xi) d\xi\,dx = \frac1d \Res(P) \, \log n + o(\log n)
\end{equation*}
for a scalar $\Res(P)$);
\item
\begin{equation*}
\tau(P)= \frac{\tau \circ \diag \left( \{ \frac1n \}_{n=1}^\infty \right)}{d(2\pi)^d} \Res(P) 
\end{equation*}
for every trace $\tau$ on the compact operators with singular values $O(n^{-1})$
iff
\begin{equation*}
\int_{\mathbb R^d} \int_{|\xi|\le n^{1/d}} p(x,\xi) d\xi\,dx = \frac1d \Res(P) \, \log n + O(1)
\end{equation*}
for a scalar $\Res(P)$.
\end{enumerate}
\end{Thm*}

This theorem is Theorem~\ref{CTTv1} in Section~\ref{sec:modtraces} of the text. 
There is a version for closed manifolds, Theorem~\ref{connectioncorman} in Section~\ref{sec:tracethmman}.
Our proof of the trace theorem uses commutator subspaces and it is very different to
the original proof of Connes' theorem.
Let us put the rationale of the proof in its plainest form.

Wodzicki initiated the study of the noncommutative residue in \cite{Wod}.  The noncommutative residue $\Res_W(P)$ vanishes if and only if a classical
pseudo-differential operator $P$
is a finite sum of commutators.  This result paired with the study of commutator subspaces of ideals, \cite{Pearcy1971}, and resulted in an extensive work, \cite{DykemaFigielWeissWodzicki2004}, categorising
commutator spaces for arbitrary two sided ideals of compact operators, see the survey, \cite{Weiss2004}.  Our colleague Nigel Kalton, whose sudden passing was a tremendous loss to ourselves personally and to mathematics in general, contributed fundamentally to this area, through, of course, \cite{Kalton1989}, and \cite{Kalton1998}, \cite{DykemaKalton1998}, \cite{FigielKalton2002}, \cite{KS2008_Creolle}.

The commutator subspace, put simply, is the kernel of all traces
on a two-sided ideal of compact linear operators of a Hilbert space $\mathcal H$ to itself. If one could show a compact operator 
$T$ belongs to the ideal $\mathcal L_{1,\infty}$ (operators whose singular values are $O(n^{-1})$) and that it
satisfies
\begin{equation} \label{intro:commdiff}
T - c \; \diag \left\{ \frac{1}{k} \right\}_{k=1}^\infty \in \Com \mathcal L_{1,\infty}
\end{equation}
for a constant $c$ (here $\Com \mathcal L_{1,\infty}$ denotes the commutator subspace, i.e.~the linear span of elements $AB-BA$, $A \in \mathcal L_{1,\infty}$, $B$ is a bounded linear operator of $\mathcal H$ to itself, and $\diag$ is the diagonal operator in some chosen basis),
then
$$
\tau(T) = c 
$$
for a constant $c$ for every trace $\tau$ with $\tau( \diag \{ k^{-1} \}_{k=1}^\infty )=1$.  This is the type of formula Connes' original theorem suggests.
Our first result, Theorem~\ref{K}, concerns differences
in the commutator subspace, i.e.~\eqref{intro:commdiff}, and it states that
$$
T - S \in \Com \mathcal L_{1,\infty} \Leftrightarrow
\sum_{j=1}^n \lambda_j(T) - \sum_{j=1}^n \lambda_j(S) = O(1)
$$ 
by using the fundamental results of \cite{DykemaFigielWeissWodzicki2004}, \cite{Kalton1989}, and \cite{DykemaKalton1998}. Here $\{\lambda_j(T)\}_{j=1}^\infty$
are the eigenvalues of $T$, with multiplicity, in any order so that $|\lambda_j(T)|$ is decreasing, with the same for $S$.   Actually, all our initial results involve general ideals but, to stay on message, we specialise to $\mathcal L_{1,\infty}$ in the introduction.  Then our goal,~\eqref{intro:commdiff}, has the explicit spectral form
\begin{equation} \label{intro:key}
\sum_{j=1}^n \lambda_j(T) - c \, \log n = O(1) .
\end{equation}
Equation~\eqref{intro:key} indicates that the log divergent
behaviour of partial sums of eigenvalues is the key to the trace theorem.

The crucial step therefore is the following theorem
on sums of eigenvalues of pseudo-differential operators.  As far as we know the theorem is new.
Results about eigenvalues are known, of course, for positive elliptic operators on closed manifolds.  The following result is for all operators of order $-d$.
\begin{Thm*}
Let $P : C_c^\infty(\mathbb R^d) \to C_c^\infty(\mathbb R^d)$ be a compactly based pseudo-differential operator of order $-d$ and with symbol $p$.
Then
$$
\sum_{j=1}^n \lambda_j(P) - \frac{1}{(2\pi)^d} \int_{\mathbb R^d} \int_{|\xi| \leq n^{1/d}} p(x,\xi) dx d\xi = O(1)
$$
where $\{\lambda_j(P)\}_{j=1}^\infty$ are the eigenvalues of $P$, with multiplicity, in any order so that $|\lambda_j(P)|$ is decreasing.
\end{Thm*}
The theorem is Theorem~\ref{eigen} in the text, which is shown for the so-called Laplacian modulated operators, and we have stated here the special
case for compactly supported pseudo-differential operators.  Given Theorem* 2 the proof
of Theorem* 1 follows, as indicated in Section~\ref{sec:modtraces}.

The following corollaries to Theorem* 1 are also proven in the text.

\medskip The result
\begin{equation} \label{intro:key2}
\int_{\mathbb R^d} \int_{|\xi| \leq n^{1/d}} p(x,\xi) dx d\xi = 
\frac{1}{d} \Res_W(P) \log n + O(1)
\end{equation}
for a classical pseudo-differential operator $P$ demonstrates that
the residue in Definition*
is an extension of the noncommutative residue and, from Theorem* 1(iii) we obtain:
\begin{Thm*}[Connes' trace theorem]
Let $P : C_c^\infty(\mathbb R^d) \to C_c^\infty(\mathbb R^d)$ be a classical compactly based pseudo-differential operator of order $-d$
with noncommutative residue $\Res_W(P)$. Then (the extension) $P \in \mathcal L_{1,\infty}$ and
$$
\tau(P) = \frac{1}{d(2\pi)^d} \Res_W(P)
$$
for every trace $\tau$ on $\mathcal L_{1,\infty}$ with $\tau( \diag \{ k^{-1} \}_{k=1}^\infty )=1$.
\end{Thm*}
This result is Corollary~\ref{cor:CTTRd} in the text.
We show the same result for manifolds, Corollary~\ref{CTToriginalmanifold}.

\medskip In the text we construct a pseudo-differential operator $Q$ whose residue
$\Res(Q)$ is not scalar.  Using Theorem* 1(ii) we obtain:
\begin{Thm*}[Pseudo-differential operators do not have unique trace]
There exists a compactly based pseudo-differential operator $Q : C_c^\infty(\mathbb R^d) \to C_c^\infty(\mathbb R^d)$ of order $-d$ such that the value
$\mathrm{Tr}_\omega(Q)$ depends on the Dixmier trace $\mathrm{Tr}_\omega$.
\end{Thm*}

The operator $Q$ is nothing extravagant, one needs only to interrupt the homogeneity of the principal symbol, see Corollary~\ref{cor:nonmeas} in the text.
There is a similar example on closed manifolds, Corollary~\ref{cor:nonmeasX}.
In summary the pseudo-differential operators of order $-d$ form quite good examples for the theory of singular traces.  Some operators, including classical ones, have the
same value for every trace.  Others have distinct trace even for the smaller set of Dixmier traces.  Theorem* 4 shows that the qualifier classical cannot be omitted from the statement of Theorem* 3.

\medskip The Laplacian modulated operators we introduce
are a wide enough class to admit the operators
$M_f (1-\Delta)^{-d/2}$ where $f \in L_{2}(\mathbb R^d)$ (almost) has compact support, $M_f u(x) = f(x)u(x)$, $u \in C_c^\infty(\mathbb R^d)$, and $\Delta$ is the Laplacian on $\mathbb R^d$.  Using Theorem~\ref{CTTv1}(iii) (the version of Theorem* 1(iii) for Laplacian modulated operators) we prove Corollary~\ref{L2trace} in the text:
\begin{Thm*}[Integration of square integrable functions]
If $f \in L_{2}(\mathbb R^d)$ has compact support then $M_f (1-\Delta)^{-d/2} \in \mathcal L_{1,\infty}$ such that
$$
\tau(M_f (1-\Delta)^{-d/2})= \frac{\mathrm{Vol} \mathbb S^{d-1}}{d(2\pi)^d} \int_{\mathbb R^d} f(x)dx
$$
for every trace $\tau$ on $\mathcal L_{1,\infty}$ with $\tau( \diag \{ k^{-1} \}_{k=1}^\infty )=1$.
\end{Thm*}
The same statement can be made for closed manifolds, omitting of course the requirement for compact support of $f$, and with the Laplace-Beltrami operator in place of the ordinary Laplacian, Corollary~\ref{cor:LDS}.

\medskip Finally, through results on modulated operators, specifically Theorem~\ref{connection},
we obtain the following spectral formula for the noncommutative residue
on a closed manifold, Corollary~\ref{cor:wodresspectral}.  The eigenvalue part of this formula was observed by T.~Fack,~\cite{Fackcomm2004}, and proven in \cite[Corollary~2.14]{AS2005} (i.e.~the log divergence of the series of eigenvalues listed with multiplicity and ordered so that their absolute value is decreasing is equal to the noncommutative residue). 
\begin{Thm*}[Spectral formula of the noncommutative residue]
Let $P$ be a classical pseudo-differential operator of order $-d$ on a closed $d$-dimensional manifold $(X,g)$.  Then
$$
d^{-1}(2\pi)^{-d}   \; \Res_W(P) = \lim_n \frac{1}{\log n} \sum_{j=1}^n (Pe_j,e_j)
= \lim_n \frac{1}{\log n} \sum_{j=1}^n \lambda_j(P)
$$
where $\{\lambda_j(P)\}_{j=1}^\infty$
are the eigenvalues of $P$ with multiplicity in any order so that $|\lambda_j(P)|$ is decreasing,
$(\cdot, \cdot)$ is the inner product on $L^2(X,g)$, and $(e_j)_{j=1}^\infty$ is an orthonormal basis of eigenvectors of the Hodge-Laplacian
$-\Delta_g$ (the negative of the Laplace-Beltrami operator) such that $-\Delta_g e_j = \lambda_j e_j$, $\lambda_1 \leq \lambda_2 \leq \ldots$
are increasing.
\end{Thm*}

Theorems* 1-6 are the main results of the text.

\section{Preliminaries}

Let $\mathcal H$ be a separable Hilbert space with inner product complex linear in the first variable and let $\mathcal B(\mathcal H)$ (respectively, $\mathcal K(\mathcal H)$) denote the bounded (respectively, compact) linear operators on $\mathcal H$.
If $(e_n)_{n=1}^{\infty}$ is a fixed orthonormal basis of $\mathcal H$
and $\{ a_n \}_{n=1}^\infty$ is a sequence of complex numbers define the operator 
$$
\diag\{a_n\}_{n=1}^{\infty} := \sum_{n=1}^\infty a_ne_ne_n^*,
$$ 
where $e_n^*(h) := (h,e_n)$, $h \in \mathcal H$, and $( \cdot , \cdot)$ denotes the inner product.
The Calkin space $\diag(\mathcal I)$ associated to a two-sided ideal $\mathcal I$ of compact operators is the sequence space
$$
\diag ( \mathcal I) := \{ \{a_n\}_{n=1}^{\infty} | \diag\{a_n\}_{n=1}^{\infty}\in \mathcal I \}.
$$
The Calkin space is independent of the choice of orthonormal basis and an operator $T\in\mathcal I$ if and only if the sequence $\{s_n(T)\}_{n=1}^{\infty}$ of its singular values belongs to $\diag(\mathcal I)$, \cite{Calkin1941}, \cite[\S 2]{S}.

The non-zero eigenvalues of a compact operator $T$ form either a sequence converging to $0$ or a finite set.  In the former case we define an {\it eigenvalue sequence} for $T$ as the sequence of eigenvalues $\{ \lambda_n(T) \}_{n=1}^{\infty}$, each repeated according to algebraic multiplicity, and arranged in an order (not necessarily unique) such that $\{ |\lambda_n(T)| \}_{n=1}^{\infty}$ is decreasing (see, \cite[p.~7]{S}).  In the latter case we construct a similar finite sequence $\{ \lambda_n(T) \}_{n=1}^N$ of the nonzero eigenvalues and then set $\lambda_n(T)=0$ for $n>N.$   The appearance of eigenvalues will always imply they are ordered as to form an eigenvalue sequence.  For a normal compact operator $T$,
$|\lambda_n(T)| = \lambda_n(|T|) = s_n(T)$, $n \in \mathbb N$, for any eigenvalue sequence
$\{ \lambda_n(T) \}_{n=1}^\infty$.  This implies that $\{ \lambda_n(T) \}_{n=1}^\infty \in \diag( \mathcal I)$ for a normal operator $T \in \mathcal I$.
The following well-known lemma will be useful so we provide the proof for completeness.

\begin{lemma} \label{essprop}
Suppose $\diag (\mathcal I)$ is a Calkin space and $\nu \in \diag (\mathcal I)$ is a positive sequence.  If $a:=\{ a_n \}_{n=1}^\infty$ is a complex-valued sequence such that $|a_n| \leq \nu_n$ for all $n \in \mathbb N$, then $a \in \diag (\mathcal I)$.
\end{lemma}
\begin{proof}
Set, for $n \in \mathbb N$,
$b_n := \frac{a_n}{\nu_n}$ if $\nu_n \not=0$
and $b_n := 0$ if $\nu_n = 0$.
Then $b := \{ b_n \}_{n=1}^\infty \in \ell_\infty$.
Hence $\diag \, b \in \mathcal B(\mathcal H)$ and $\diag \, a = \diag (b \cdot \nu) = \diag \, b \cdot \diag \, \nu \in \mathcal I$ since $\mathcal I$ is an ideal.
 \end{proof}

\begin{corollary} \label{esscor}
Suppose $T \in \mathcal I$ is normal.  Then $\{ \lambda_n(T) \}_{n=1}^\infty \in \diag (\mathcal I)$ where $\{ \lambda_n(T) \}_{n=1}^\infty$ is an eigenvalue sequence of $T$.
\end{corollary}
\begin{proof}
Using the spectral theorem for normal operators,
$|\lambda_n(T)| = s_n(T)$, $n \in \mathbb N$.
Hence $|\lambda_n(T)| \leq \nu_n$ where $\nu_n := s_n(T) \in \diag( \mathcal I)$
is positive.  By Lemma~\ref{essprop} $\{ \lambda_n(T) \}_{n=1}^\infty \in \diag (\mathcal I)$.
 \end{proof}

The statement that $\{\lambda_n(T)\}_{n=1}^{\infty} \in \diag(\mathcal I)$
for every $T\in\mathcal I$ is false in general.
Geometrically stable ideals were introduced by Kalton, \cite{Kalton1998}.  A two-sided ideal $\mathcal I$ is called {\it geometrically stable} if given any decreasing nonnegative sequence $\{s_n\}_{n=1}^{\infty}\in\diag(\mathcal I)$ we have $\{(s_1s_2\ldots s_n)^{1/n}\}_{n=1}^{\infty}\in\diag(\mathcal I)$.   It is a theorem of Kalton and Dykema that $\mathcal I$ is geometrically stable if and only if $\{\lambda_n(T)\}_{n=1}^{\infty} \in \diag(\mathcal I)$ for all $T \in \mathcal I$, \cite[Theorem 1.3]{DykemaKalton1998}.  An ideal $\mathcal I$ is called Banach (respectively, quasi-Banach) if there is a norm (respectively, quasi-norm) $\|\cdot\|_{\mathcal I}$ on $\mathcal I$ such that $(\mathcal I,\|\cdot\|_{\mathcal I})$ is complete and we have
$ \|ATB\|_{\mathcal I}\le \|A\|_{\mathcal B(\mathcal H)}\|T\|_{\mathcal I}\|B\|_{\mathcal B(\mathcal H)}$, $A,B\in\mathcal B(\mathcal H)$, $T\in\mathcal I.$   Equivalently, $\diag (\mathcal I )$ is a Banach (respectively, quasi-Banach) symmetric sequence space, see e.g.~\cite{S}, \cite{LindenstraussTzafriri1977}, \cite{KS2008_Creolle}.
Every quasi-Banach ideal is geometrically stable, \cite{Kalton1998}.  An example of a non-geometrically stable ideal is given
in~\cite{DykemaKalton1998}.

If $\mathcal I_1$ and $\mathcal I_2$ are ideals we denote by $\mathcal I_1\mathcal I_2$ the ideal generated by all products $AB,BA$ for $A\in\mathcal I_1$ and $B\in\mathcal I_2.$
If $A,B\in\mathcal B(\mathcal H)$ we let $[A,B]=AB-BA$.  We define $[\mathcal I_1, \mathcal I_2]$ to be the linear span of all $[A,B]$ for $A\in\mathcal I_1$ and $B\in\mathcal I_2.$  
It is a theorem that $[\mathcal I_1, \mathcal I_2] = [\mathcal I_1 \mathcal I_2, \mathcal B (\mathcal H)]$, \cite[Theorem 5.10]{DykemaFigielWeissWodzicki2004}. The space $\Com \mathcal I:=[\mathcal I,\mathcal B(\mathcal H)] \subset \mathcal I$ is called the commutator subspace of an ideal $\mathcal{I}$.

\section{A theorem on the commutator subspace}

There is a fundamental description of the normal operators $T\in \Com \mathcal I$ given by Dykema, Figiel, Weiss and Wodzicki, \cite{DykemaFigielWeissWodzicki2004}, see also \cite[Theorem 3.1]{Kalton1998}.

\begin{theorem} \label{DFWW}  Suppose $\mathcal I$ is a two-sided ideal in $\mathcal K(\mathcal H)$
and $T\in\mathcal I$ is normal.  Then the following statements are equivalent:
\begin{enumerate}[(i)]
\item $T\in\Com\mathcal I$;
\item for any eigenvalue sequence $\{ \lambda_n(T) \}_{n=1}^\infty$,
\begin{equation}\label{eq:DFWW}
\left\{ \frac1n \sum_{j=1}^n \lambda_j(T) \right \}_{n=1}^{\infty}\in \diag(\mathcal I);
\end{equation}
\item for any eigenvalue sequence $\{ \lambda_n(T) \}_{n=1}^\infty$,
\begin{equation}\label{eq:DFWW2}
\frac1n \left| \sum_{j=1}^n \lambda_j(T) \right| \leq \mu_n
\end{equation}
for a positive decreasing sequence $\mu = \{ \mu_n \}_{n=1}^\infty \in \diag (\mathcal I)$.
\end{enumerate}
\end{theorem}

We would like to observe the following refinement of Theorem~\ref{DFWW}.
\begin{theorem}\label{DFWW2} Suppose $\mathcal I$ is a two-sided ideal in $\mathcal K(\mathcal H)$
and $T,S\in\mathcal I$ are normal. Then the following statements are equivalent:
\begin{enumerate}[(i)]
\item $T-S\in\Com \mathcal I$;
\item for any eigenvalue sequences $\{ \lambda_j(T) \}_{j=1}^\infty$ of $T$ and
$\{ \lambda_j(S) \}_{j=1}^\infty$ of $S$,
\begin{equation}\label{three} \left\{\frac1n \left(\sum_{j=1}^n\lambda_j(T)-\sum_{j=1}^n\lambda_j(S)\right)\right\}_{n=1}^{\infty} \in \diag(\mathcal I),\end{equation}
\item for any eigenvalue sequences $\{ \lambda_j(T) \}_{j=1}^\infty$ of $T$ and
$\{ \lambda_j(S) \}_{j=1}^\infty$ of $S$,
\begin{equation}\label{three2} 
 \frac1n \left|\sum_{j=1}^n \lambda_j(T) - \sum_{j=1}^n\lambda_j(S) \right| \leq \mu_n
\end{equation}
for a positive decreasing sequence $\mu = \{ \mu_n \}_{n=1}^\infty \in \diag (\mathcal I)$.
\end{enumerate}
\end{theorem}

\begin{proof} Observe that the normal operator
$$ V=\left(\begin{matrix} T&0\\0&-S\end{matrix}\right)=\left(\begin{matrix} T-S &0\\ 0&0\end{matrix}\right)+ \left(\begin{matrix} S &0\\ 0&-S\end{matrix}\right)$$ belongs to $\Com\mathcal I$ if and only if $T-S\in\Com \mathcal I.$
Indeed, it is straightforward to see that the eigenvector sequence of the operator
$\bigl( \begin{smallmatrix}
S&0\\ 0&-S
\end{smallmatrix} \bigr)$ satisfies \eqref{eq:DFWW}
and, since $S \in \mathcal I$, we have $\bigl( \begin{smallmatrix}
S&0\\ 0&-S
\end{smallmatrix} \bigr) \in \Com \mathcal I$ by Theorem~\ref{DFWW}.

(iii) $\Rightarrow$ (ii)  Let $a_n := n^{-1} \sum_{j=1}^n( \lambda_j(T) - \lambda_j(S))$.
By Lemma~\ref{essprop} $\{ a_n \}_{n=1}^\infty \in \diag(\mathcal I)$.

(ii) $\Rightarrow$ (i)
We have
$$\sum_{j=1}^n\lambda_j(V)=\sum_{j=1}^r\lambda_j(T)-\sum_{j=1}^s\lambda_j(S)$$ where $r+s=n$ and $|\lambda_{r+1}(T)|,|\lambda_{s+1}(S)|\le |\lambda_{n+1}(V)|.$   Hence
\begin{equation} \label{maindeal}
 \left|\sum_{j=1}^n\lambda_j(V)-\sum_{j=1}^n\lambda_j(T)+\sum_{j=1}^n\lambda_j(S)\right|\le n|\lambda_n(V)|.
\end{equation}
Since $\nu := \{ |\lambda_n(V)| \}_{n=1}^\infty \in \diag ( \mathcal I )$ is positive and decreasing
$$
\frac{1}{n} \sum_{j=1}^n \lambda_j(V) -  \frac{1}{n} \left(  \sum_{j=1}^n\lambda_j(T)-\sum_{j=1}^n\lambda_j(S) \right) \in \diag ( \mathcal I )
$$
by Lemma~\ref{essprop}.  Hence $\frac{1}{n} \sum_{j=1}^n \lambda_j(V) \in \diag ( \mathcal I )$
if  $\frac{1}{n} (\sum_{j=1}^n\lambda_j(T)-\sum_{j=1}^n\lambda_j(S)) \in \diag ( \mathcal I )$.  It follows from Theorem~\ref{DFWW} that $V \in \Com \mathcal I$.

(i) $\Rightarrow$ (iii)
We note from equation~\eqref{maindeal} that
$$
\frac{1}{n} \left|\sum_{j=1}^n\lambda_j(T)-\sum_{j=1}^n\lambda_j(S)\right|\le |\lambda_n(V)| + \frac{1}{n} \left| \sum_{j=1}^n\lambda_j(V) \right|.
$$ 
The sequence
$\nu := \{ |\lambda_n(V)| \}_{n=1}^\infty \in \diag ( \mathcal I )$ is positive and decreasing
 and, since $V \in \Com \mathcal I$, there exists a decreasing sequence $\nu'$
 such that $\frac{1}{n} \left| \sum_{j=1}^n\lambda_j(V) \right| \leq \nu'_n$.
 Now (iii) follows by setting $\mu = \nu + \nu'$.
 \end{proof}

In \cite{Kalton1998}, which used results in \cite{DykemaFigielWeissWodzicki2004} although it appeared chronologically earlier, it was shown that Theorem~\ref{DFWW} can be extended to non-normal operators under the hypothesis that $\mathcal I$ is geometrically stable.  Theorem~\ref{DFWW2} can be extended similarly.

\begin{theorem}\label{K} Suppose $\mathcal I$ is a geometrically stable ideal in $\mathcal K(\mathcal H)$ and $T,S\in\mathcal I$.  Then the following statements are equivalent: 
\begin{enumerate}[(i)]
\item $T-S\in\Com \mathcal I$;
\item for any eigenvalue sequences $\{ \lambda_j(T) \}_{j=1}^\infty$ of $T$ and
$\{ \lambda_j(S) \}_{j=1}^\infty$ of $S$,
\begin{equation}\label{three+} \left\{\frac1n \left(\sum_{j=1}^n\lambda_j(T)-\sum_{j=1}^n\lambda_j(S)\right)\right\}_{n=1}^{\infty} \in \diag(\mathcal I),\end{equation}
\item for any eigenvalue sequences $\{ \lambda_j(T) \}_{j=1}^\infty$ of $T$ and
$\{ \lambda_j(S) \}_{j=1}^\infty$ of $S$,
\begin{equation}\label{three+2} 
 \frac1n \left|\sum_{j=1}^n \lambda_j(T) - \sum_{j=1}^n\lambda_j(S) \right| \leq \mu_n
\end{equation}
for a positive decreasing sequence $\mu = \{ \mu_n \}_{n=1}^\infty \in \diag (\mathcal I)$.
\end{enumerate}
\end{theorem}
\begin{proof}
Let $T \in \mathcal I$.  From \cite[Corollary 2.5]{DykemaKalton1998} $T=N+Q$ where $Q \in \mathcal I$ is quasinilpotent
and $N \in \mathcal{I}$ is normal with eigenvalues and multiplicities the same as $T$.
From \cite[Theorem 3.3]{Kalton1998} we know $Q \in \Com \mathcal I$.
Hence $T=N_T+Q_T$ and $S=N_S+Q_S$ where $Q_S,Q_T \in \Com \mathcal{I}$ are quasinilpotent and $N_T,N_S$ are normal with eigenvalues and multiplicities the same as $T$ and $S$, respectively. Since $T-S\in\Com \mathcal I$ if and only if $N_T-N_S\in\Com \mathcal I$ the results follow from Theorem~\ref{DFWW2}.
 \end{proof}

We recall that $\diag\{\lambda_n(T)\}_{n=1}^{\infty}\in \mathcal I$ when $\mathcal I$ is geometrically stable.
Theorem~\ref{K} therefore has the following immediate corollary.

\begin{corollary}\label{K2} Let $\mathcal I$ be a geometrically stable ideal in $\mathcal K(\mathcal H)$
and $T \in \mathcal I$.  Then $T-\diag\{\lambda_n(T)\}_{n=1}^{\infty}\in \Com\mathcal I$.
\end{corollary}
\begin{proof}
Set $S:=\diag\{\lambda_n(T)\}_{n=1}^{\infty} \in \mathcal I$.  Then
$\lambda_n(S) = \lambda_n(T)$, $n \in \mathbb N$, and $T-S \in \Com \mathcal I$
by Theorem~\ref{K}. 
 \end{proof}

\section{Applications to traces} \label{sec:traces_app}

Suppose $\mathcal I$ is a two-sided ideal of compact operators.
A trace $\tau:\mathcal I\to\mathbb C$ is a linear functional that vanishes on the commutator subspace, i.e.~it satisfies the condition
$$ \tau([A,B])=0 \, , \qquad A\in\mathcal I,\ B\in\mathcal B(\mathcal H).$$
Note that we make no assumptions about continuity or positivity of the linear functional. The value
$$
\tau(\diag\{a_n\}_{n=1}^{\infty}) \ , \qquad \{ a_n \}_{n=1}^\infty \in \diag( \mathcal I)
$$
is independent of the choice of orthonormal basis.  Therefore any trace $\tau : \mathcal I \to \mathbb{C}$ induces a linear functional $\tau \circ \diag$ (defined by the above value) on the Calkin space $\diag ( \mathcal I )$.

\begin{corollary} \label{DFWWCor}
There are non-trivial traces on $\mathcal I$ if and only if $\Com\mathcal I \not= \mathcal I$,
which occurs if and only if $\{\frac1n \sum_{j=1}^n s_j \}_{n=1}^{\infty} \not\in \diag(\mathcal I)$ for some
positive sequence $\{ s_n \}_{n=1}^\infty \in \diag(\mathcal I)$.
\end{corollary}

The proof is evident by considering the quotient vector space $\mathcal I / \Com \mathcal I$
and applying Theorem~\ref{DFWW}, so we omit it. 
The condition in Corollary~\ref{DFWWCor} implies that traces on two-sided ideals other than the ideal of nuclear operators exist (e.g.~the quasi-Banach ideal $\mathcal L_{1,\infty}$
such that $\diag (\mathcal L_{1,\infty} ) = \ell_{1,\infty}$), \cite{KaftalWeiss2022},
see also \cite[\S 5]{DykemaFigielWeissWodzicki2004} for other examples of ideals that do and do not support non-trivial traces.  In~\cite{DykemaKalton1998} it was shown that every trace on
a geometrically stable ideal is determined by its associated functional applied to an eigenvalue sequence,
which is an extension of Lidskii's theorem. 

\begin{corollary}[Lidskii Theorem] \label{KCor} Let $\mathcal I$ be a geometrically stable ideal in $\mathcal K(\mathcal H).$  Suppose $T\in\mathcal I.$  Then
\begin{equation} \label{eq:KCor}
\tau(T)=\tau \circ \diag (\{\lambda_n(T)\}_{n=1}^{\infty})
\end{equation}
for every trace $\tau: \mathcal I \to \mathbb C$ and any eigenvalue sequence
$\{ \lambda_n(T) \}_{n=1}^\infty$ of $T$.
\end{corollary}

The proof, given Corollary~\ref{K2}, is trivial and therefore omitted.
For evident reasons (for some $T \in \mathcal I$, $\{ \lambda_n(T) \}_{n=1}^\infty \not\in \diag ( \mathcal I )$) the Lidskii formulation
can only apply to geometrically stable ideals.
A general characterisation of traces on non-geometrically stable ideals requires an explicit formula for products $T=AS \in \mathcal I$ where $A$
and $S$ do not commute.  Such a formula is also of interest when studying linear functionals on the bounded operators of the form
$A \mapsto \tau(AS)$, $A \in \mathcal B ( \mathcal H )$, where $S \in \mathcal I$ is Hermitian and $\tau : \mathcal I \to \mathbb{C}$ is a trace (e.g.~in A.~Connes' noncommutative geometry, \cite[\S 4]{C}). We now characterise traces of products.

We introduce some terminology for systems of eigenvectors that are ordered to correspond with eigenvalue sequences.  If $T$ is a compact operator of infinite rank, we define an orthonormal sequence $(e_n)_{n=1}^{\infty}$ to be an {\it eigenvector sequence} for $T$ if $Te_n=\lambda_n(T)e_n$ for all $n \in \mathbb{N}$ where $\{ \lambda_n \}_{n=1}^\infty$ is an eigenvalue sequence.  If $T$ is Hermitian, an eigenvector sequence exists
and there is an eigenvector sequence which forms a complete orthonormal system.

We will also need the following lemmas (see for example \cite{DykemaFigielWeissWodzicki2004}).

\begin{lemma} \label{dec} Let $\mathcal I$ be a two-sided ideal in $\mathcal K(\mathcal H)$.  Suppose $D=\diag\{\alpha_n\}_{n=1}^{\infty}$ where
$\{ \alpha_n \}_{n=1}^\infty$ is a sequence of complex numbers such that $|\alpha_n|\le \mu_n$ where $\{\mu_n \}_{n=1}^{\infty}\in\diag(\mathcal I)$ is decreasing.  Then
$$ \left|\sum_{j=1}^n\alpha_j-\sum_{j=1}^n\lambda_j(D)\right| \le 2n\mu_{n}.$$\end{lemma}
\begin{proof}
We have $\lambda_j(D)=\alpha_{m_j}$ where $m_1,m_2,\ldots$ are distinct.  Thus
$$ \sum_{j=1}^n\lambda_j(D)=\sum_{j\in\mathbb A}\alpha_j$$ where
$\mathbb A=\{m_1,\ldots,m_n\}.$ If $k\in\mathbb A\setminus
\{1,2,\ldots,n\}$ we have $|\alpha_k|\leq \mu_k \le \mu_n$.  On the
other hand, since $$ \mathbb A = \left\{m \in \mathbb N:\ \ \alpha_m =
  \lambda_{j}(D), \ \ \text{for some~$j \leq n$}\right\}, $$ if $k\in
\{1,2,\ldots,n\}\setminus \mathbb A$ we have $|\alpha_k|\le
|\lambda_n(D)| \le \mu_{n}.$ Hence
$$ \left|\sum_{j=1}^n\lambda_j(D)-\sum_{j=1}^n\alpha_j\right|\le 2n\mu_{n}.$$
 \end{proof}

\begin{lemma}\label{main0} Suppose $S\in\mathcal K(\mathcal H)$ is
  Hermitian and $(e_j)_{j=1}^{\infty}$ is an eigenvector sequence
  for $S.$ Suppose $A$ is Hermitian and $H:=\frac12(AS+SA)$.  Then we have
  \begin{equation}\label{estimate0}
    \left|\sum_{j=1}^n\lambda_j(H)-\sum_{j=1}^n
      (ASe_j,e_j)\right|\le ns_{n+1}(H)+ns_{n+1}(S)\left( \frac1n
      \sum_{j=1}^ns_j(A) \right)
  \end{equation}
  if $A \in \mathcal K (\mathcal H)$ is compact, and
  \begin{equation}\label{estimate0b}
    \left|\sum_{j=1}^n\lambda_j(H)-\sum_{j=1}^n (ASe_j,e_j)\right|\le ns_{n+1}(H)+ns_{n+1}(S)\| A \| 
  \end{equation}
  if $A \in \mathcal B(\mathcal H)$ is bounded but not compact.
\end{lemma}

\begin{proof}  Let $(f_n)_{n=1}^{\infty}$ be an eigenvector sequence for $H$.  Let $P_n$ and $Q_n$ be the orthogonal projections of $\mathcal H$ on $[e_1,\ldots,e_n]$ and $[f_1,\ldots,f_n]$ respectively and let $R_n$ be the orthogonal projection on the linear span $[e_1,\ldots,e_n,f_1,\ldots, f_n].$

If $A$ is compact define $\beta_n := \frac1n \sum_{j=1}^ns_j(A)$.  Otherwise set $\beta_n := \| A \|$.
Since $\rank(R_n-P_n)\le n$ we have $$ |\Tr (AS(R_n-P_n))| \le ns_{n+1}(S)\beta_n .$$
Similarly
$$|\Tr (SA(R_n-P_n))|=|\Tr (A(R_n-P_n)S)| \le n  s_{n+1}(S)\beta_n,$$
and hence
$$
| \Tr(H(R_n-P_n)) | \le n  s_{n+1}(S)\beta_n .
$$
Similarly
$$ |\Tr (H(R_n-Q_n))|\le n s_{n+1}(H).$$
Hence
\begin{align*}
 |\Tr (H(P_n-Q_n))| & \le |\Tr (H(P_n-R_n))|+ |\Tr (H(R_n-Q_n))| \\
 & \le ns_{n+1}(H)+n  s_{n+1}(S)\beta_n.
 \end{align*}
 \end{proof}

\begin{theorem}\label{main}  Let $\mathcal I_1$ be a two-sided ideal in $\mathcal B(\mathcal H)$
such that $\mathcal I_1=\Com\mathcal I_1$
and $\mathcal I_2$ be a two-sided ideal in $\mathcal K(\mathcal H)$.
Let $\mathcal I=\mathcal I_1\mathcal I_2$.
Suppose $S\in\mathcal I_2$ is Hermitian and
  that $(e_n)_{n=1}^{\infty}$ is an eigenvector sequence for $S$.
  Suppose $A\in\mathcal I_1$ is such that, for some decreasing
  positive sequence $\{\mu_n\}_{n=1}^{\infty}\in \diag(\mathcal I)$,
  we have $|(ASe_n,e_n)|\le \mu_n$, $n \in \mathbb N$. Then
  $AS, \diag\{(ASe_n,e_n)\}_{n=1}^{\infty}\in \mathcal I$,
$$AS-\diag\{(ASe_n,e_n)\}_{n=1}^{\infty}\in \Com \mathcal I$$ and, hence,
for every trace $\tau: \mathcal I \to \mathbb C$ we have
$$ \tau(AS)=\tau \circ \diag(\{(ASe_n,e_n)\}_{n=1}^{\infty}) = \tau\circ \diag(\{(Ae_n,e_n)\lambda_n(S)\}_{n=1}^{\infty}).$$
\end{theorem}
\begin{proof} Our assumptions imply that $AS\in\mathcal I.$ Also by the assumption that $|(ASe_n,e_n)|\le \mu_n$ for $\mu \in \diag( \mathcal I)$,
it follows from Lemma~\ref{essprop} that $\diag\{(ASe_n,e_n)\}_{n=1}^{\infty}\in \mathcal I$.

First let us assume that $A$ is Hermitian.  Set $D:=\diag\{(ASe_n,e_n)\}_{n=1}^{\infty}$ and $\alpha_n := (ASe_n,e_n)$, $n \in \mathbb N$.  By assumption
$|\alpha_n| \leq \mu_n$ where $\mu \in \diag(\mathcal I)$ is positive and decreasing and, by applying Lemma~\ref{dec} to the sequence $\alpha_n := (ASe_n,e_n)$, we have
\begin{equation} \label{eq:main1} \left|\sum_{j=1}^n\lambda_j(D)-\sum_{j=1}^n(ASe_j,e_j)\right|\le 2n\mu_{n}.
\end{equation}
From Lemma~\ref{main0} we have that
\begin{equation} \label{eq:main2}
\left|\sum_{j=1}^n\lambda_j(H)-\sum_{j=1}^n(ASe_j,e_j)\right|\le ns_{n+1}(H)+ n s_{n+1}(S) \beta_n
\end{equation}
where $H=\frac12(AS+SA)$ and $\beta_n := \frac1n\sum_{j=1}^ns_j(A)$, $n \in \mathbb N$, if $A$ is compact, or $\beta_n := \| A \|$, $n \in \mathbb N$, if $A$ is bounded but not compact.

Suppose $A \in \mathcal I_1$ where $\mathcal I_1$ is an ideal of compact operators
such that $\Com \mathcal I_1=\mathcal I_1$.   Then $|A| \in \Com \mathcal I_1$
and, from the equivalent conditions in Theorem~\ref{DFWW} there exists a decreasing sequence
$\mu' \in \diag( \mathcal I_1)$ such that $\frac1n\sum_{j=1}^ns_j(A) \leq \mu'_n$, $n \in \mathbb N$.
Set $\nu:=\{ 2\mu_{n}+ s_{n+1}(H) + \mu'_n s_{n+1}(S) \}_{n=1}^\infty$ which is
positive and decreasing.  By assumption $\{ \mu_n \}_{n=1}^\infty \in \diag(\mathcal I)$.  By the fact that $\mathcal I$ is a two-sided ideal of compact operators then $H=\frac{1}{2}(AS+SA) \in \mathcal I$
and $\{ s_{n+1}(H) \}_{n=1}^\infty \in \diag(\mathcal I)$.  Finally
$ \mu'_n s_{n+1}(S) \in \diag (\mathcal I_1) \cdot \diag(\mathcal I_2) \subset \diag (\mathcal I)$.  Hence $\nu \in \diag( \mathcal I)$.   Then, using~\eqref{eq:main1} and~\eqref{eq:main2},
\begin{equation} \label{eq:main3}
\left|\sum_{j=1}^n\lambda_j(D)-\sum_{j=1}^n\lambda_j(H)\right|\le n\nu_n .
\end{equation}
Now suppose $A \in \mathcal I_1 = \mathcal B ( \mathcal H )= \Com \mathcal B( \mathcal H )$, \cite{Pearcy1971}, \cite{Halmos1954}.  Then $\mathcal{I}=\mathcal{I}_2$ and, in this case,
we define the sequence
$\nu := \{2\mu_{n}+s_{n+1}(H)+\| A \| s_{n+1}(S)\}_{n=1}^\infty \in \diag(\mathcal I)$ which
is positive and decreasing.  Therefore, in this case,~\eqref{eq:main3} still
holds for this new choice of decreasing positive sequence $\nu \in \diag(\mathcal I)$.

With~\eqref{eq:main3} satisfied for the cases $A$ compact or $A$ bounded but not compact,
$D-H\in \Com\mathcal I$ by an application of Theorem~\ref{DFWW2}.  Since $H-AS=\frac12[S,A]\in [\mathcal I_1,\mathcal I_2]=[\mathcal B (\mathcal H), \mathcal I] = \Com\mathcal I$ (see the preliminaries), we obtain $D-AS \in \Com \mathcal I$ and the result of the theorem when $A$ and $S$ are Hermitian.

The general case follows easily by splitting $A$ into real and imaginary parts.
 \end{proof}

\begin{corollary}\label{main2}  Let $\mathcal I$ be a two-sided ideal in $\mathcal K(\mathcal H)$.  Suppose $S\in\mathcal I$ is Hermitian and $(e_n)_{n=1}^{\infty}$ is an eigenvector sequence for $S$.
Then
$$AS-\diag\{(ASe_n,e_n)\}_{n=1}^{\infty}\in \Com \mathcal I$$
for every $A \in \mathcal B( \mathcal H )$
 and hence
for every trace $\tau: \mathcal I \to \mathbb C$ we have
$$ \tau(AS)=\tau \circ \diag(\{(ASe_n,e_n)\}_{n=1}^{\infty}) = \tau \circ \diag (\{(Ae_n,e_n)\lambda_n(S)\}_{n=1}^{\infty}).
$$
\end{corollary}
\begin{proof}
Set $\mathcal I_1 := \mathcal B( \mathcal H ) = \Com \mathcal B( \mathcal H )$ and $\mathcal I_2 := \mathcal I$.  As
$$|(ASe_n,e_n)| \leq \| A \| |\lambda_n(S)| \in \diag(\mathcal I)$$
and $\| A \| |\lambda_n(S)|$ is a positive decreasing sequence,
we obtain the result from Theorem~\ref{main}.
 \end{proof}

We can now identify the form of every trace on any ideal of compact operators.  

\begin{corollary}\label{main3}  Let $\mathcal I$ be a two-sided ideal in $\mathcal K(\mathcal H)$.  Suppose $T\in\mathcal I$.
Then
$$\tau(T)= \tau \circ \diag (\{s_n(T) ( f_n,e_n ) \}_{n=1}^{\infty})$$
for every trace $\tau: \mathcal I \to \mathbb C$, where 
$$
T = \sum_{n=1}^\infty s_n(T) f_n e_n^*
$$
is a canonical decomposition of $T$ ($\{ s_n(T) \}_{n=1}^\infty$ is the sequence of singular values of $T$, $(e_n)_{n=1}^\infty$ an orthonormal basis such that $|T|e_n = s_n(T)e_n$,
$(f_n)_{n=1}^\infty$ an orthonormal system such that $Te_n = s_n(T)f_n$,
and $e_n^*(\cdot) := ( \cdot , e_n )$).
\end{corollary}
\begin{proof}
Let $T=U|T|$ be the polar decomposition of the compact operator $T$
into the positive operator $|T|$ and the partial isometry $U$.
The eigenvalue sequence $\{ \lambda_n(|T|) \}_{n=1}^\infty$ defines
the singular values $ \{ s_n(T) \}_{n=1}^\infty$ of $T$.
Let $(e_n)_{n=1}^\infty$ be any orthonormal system such that $|T|e_n = s_n(T) e_n$
(an eigenvector sequence of $|T|$).
Since $U$ is bounded and $|T| \in \mathcal I$ is positive we apply Corollary~\ref{main2} (with $A=U$ and $S=|T|$) and obtain
$$
\tau(T) = \tau(U|T|) = \tau \circ \diag (\{(Ue_n,e_n)s_n(T)\}_{n=1}^{\infty}).
$$
We set $f_n := Ue_n$, $n \in \mathbb N$.  If $(e_n)_{n=1}^\infty$ forms a complete system,
then we have the decomposition, \cite[Theorem 1.4]{S},
$
|T| = \sum_{n=1}^\infty s_n(T) e_n e_n^* .
$
It follows that $T = U|T| = \sum_{n=1}^\infty s_n(T) f_n e_n^*$.
 \end{proof}

\section{Modulated operators and the weak-$\ell_1$ space} \label{sec:modulated}

The traces of interest in Connes' trace theorem (and in Connes' noncommutative geometry in general) are traces on the ideal $\mathcal L_{1,\infty}$
associated to the weak-$\ell_1$ space $\ell_{1,\infty}$, i.e.~$\diag( \mathcal L_{1,\infty} ) = \ell_{1,\infty}$.  It is indicated below that the ideal $\mathcal L_{1,\infty}$ is geometrically stable (it is a quasi-Banach ideal),
so the Lidskii formulation applies to all its traces.

We were led in our investigations, following results like Corollary~\ref{main2},
to ask to what degree the Fredholm formulation applies to
traces on $\mathcal L_{1,\infty}$.  The Fredholm formulation of the canonical trace on trace class operators (usually taken as the definition of the canonical trace) is
$$
\Tr (T) = \sum (Te_n,e_n) , \qquad T \in \mathcal L_1
$$
where the usual sum $\sum : \ell_1 \to \mathbb C$ can be understood as the functional $\Tr \circ \diag$.
In the Fredholm formulation $(e_n)_{n=1}^\infty$ is any orthonormal basis
and the same basis can be used for all trace class operators $T \in \mathcal L_1$, which is quite distinct
to the statement of Corollary~\ref{main2}.  We know that the Fredholm formulation is false for
traces on $\mathcal L_{1,\infty}$ (but we do not offer any proof of this fact here\footnote{Private communication by D.~Zanin.}),
i.e.~if $\tau : \mathcal L_{1,\infty} \to \mathbb C$ is a non-zero trace
there does not exist any basis $(e_n)_{n=1}^\infty$ such
that $\tau(T) = \tau \circ \diag (\{(Te_n,e_n)\}_{n=1}^\infty)$ for all $T \in \mathcal L_{1,\infty}$.

We considered whether there were restricted Fredholm formulations, which may hold
for some subspace of $\mathcal L_{1,\infty}$ instead of the whole ideal.
To this end we introduce new left ideals of the Hilbert-Schmidt operators.
We have used the term modulated operators, see Definition~\ref{def:mod},
for the elements of the left ideals
and the precise statement of a `restricted Fredholm formulation' is
Theorem~\ref{connection}.  It is the aim of this section
to prove Theorem~\ref{connection}.

Our purpose for introducing modulated operators is to study operators on
manifolds modulated by the Laplacian, where this definition
is made precise in Section~\ref{sec:PsDO}.
Compactly supported pseudo-differential operators of order $-d$ on
$\mathbb R^d$
offer examples of these so-called Laplacian modulated operators, and this will be our avenue to proving extensions and variants of Connes' trace theorem.

\begin{notation}
Henceforth we use big $O$, theta $\Theta$, and little $o$ notation,
meaning $f(s) = O(g(s))$ if $|f(s)| \leq C |g(s)|$ for
a constant $C>0$ for all $s \in \mathbb N$ or $s \in \mathbb R$,
$f(s) = \Theta(g(s))$ if $c | g(s)| \leq |f(s)| \leq C |g(s)|$ for
constants $C>c>0$ for all $s \in \mathbb N$ or $s \in \mathbb R$,  and
$f(s) = o(g(s))$ if $|f(s)| |g(s)|^{-1} \to 0$ as $s \to \infty$, respectively.
\end{notation}

Let $\mathcal L_2$ denote the Hilbert-Schmidt operators on the Hilbert space $\mathcal H$.

\begin{definition} \label{def:mod}
Suppose $V : \mathcal H \to \mathcal H$ is a positive bounded operator.
An operator $T : \mathcal H \to \mathcal H$ is \emph{$V$-modulated} if
\begin{equation}\label{modulated}
\|T(1+tV)^{-1}\|_{\mathcal L_2} = O(t^{-1/2}) .\end{equation}
We denote by $\mathrm{mod} (V) $ the set of $V$-modulated operators.
\end{definition}

It follows from the definition that $\mmod (V)$
is a subset of the Hilbert-Schmidt operators $\mathcal L_2$ and
that it forms a left ideal of $\mathcal B(\mathcal H)$ (see Proposition~\ref{prop:basic_prop} below).

The main result of this section is the following theorem.

\begin{theorem}\label{connection}  
Suppose $T$ is $V$-modulated where $0 < V \in \mathcal L_{1,\infty}$,
and $( e_n )_{n=1}^\infty$ is an orthonormal basis such that $Ve_n = s_n(V)e_n$, $n \in \mathbb N$.
Then
\begin{enumerate}[(i)]
\item $T, \diag\{(Te_n,e_n)\}_{n=1}^{\infty} \in \mathcal L_{1,\infty}$;
\item $T-\diag\{(Te_n,e_n)\}_{n=1}^{\infty} \in\Com\mathcal L_{1,\infty}$;
\item every eigenvalue sequence $\{ \lambda_n(T) \}_{n=1}^\infty$ of $T$ satisfies
$$ \sum_{j=1}^n\lambda_j(T)-\sum_{j=1}^n(Te_j,e_j) = O(1) $$
where $O(1)$ denotes a bounded sequence.
\end{enumerate}
\end{theorem}

\begin{remark}[Fredholm formula]
Evidently from Theorem~\ref{connection}, if $T$ is $V$-modulated then
$$
\tau(T) = \tau \circ \diag ( \{ (Te_n,e_n) \}_{n=1}^\infty )
$$
for every trace $\tau : \mathcal L_{1,\infty} \to \mathbb C$.
\end{remark}

The proof of Theorem~\ref{connection} is provided in a section below.

\subsection{Properties of modulated operators}

We establish the basic properties of $V$-modulated operators.  This will simplify the proof of Theorem~\ref{connection} and results in later sections.

\begin{notation}
The symbols $\doteq$, $\dot{\le}$, $\dot{\ge}$, may also be used to denote equality or inequality up to a constant.
Where it is necessary to indicate that the constant depends on parameters $\theta_1, \theta_2, \ldots$
we write $\doteq_{\theta_1, \theta_2, \ldots}$, $\dot{\le}_{\theta_1, \theta_2, \ldots}$, $\dot{\ge}_{\theta_1, \theta_2, \ldots}$.   The constants may not be the same in successive uses of the symbol.  We introduce this notation to improve text where the value of constants has no relevance to statements or proofs.
\end{notation}

\begin{proposition} \label{prop:basic_prop} \label{rem:leftideal}
Suppose $V : \mathcal H \to \mathcal H$ is positive and bounded.
The set of $V$-modulated operators, $\mathrm{mod} (V)$,
is a subset of $\mathcal L_2$ that forms a left ideal of $\mathcal B (\mathcal H)$.
\end{proposition}
\begin{proof}
Suppose $T\in \mmod (V)$.  Then
$$
\| T \|_{\mathcal L_2} = \| T(1+V)^{-1}(1+V) \|_{\mathcal L_2}
\le \| T(1+V)^{-1} \|_{\mathcal L_2} \| 1+V \|
$$
and $T$ is Hilbert-Schmidt.
Suppose $A_1,A_2 \in \mathcal B( \mathcal H)$ and $T_1,T_2 \in \mmod (V)$.  We have, for $t \ge 1$,
$$
\|(A_1 T_1+A_2 T_2)(1+tV)^{-1}\|_{\mathcal L_2} \leq \| A_1 \| \|T_1(1+tV)^{-1}\|_{\mathcal L_2} + \| A_2 \| \|T_2(1+tV)^{-1}\|_{\mathcal L_2} 
$$
so
$$
\|(A_1 T_1+A_2 T_2)(1+tV)^{-1}\|_{\mathcal L_2} = O(t^{-1/2}) .
$$
Hence $\mmod (V)$ forms a left ideal of $\mathcal B( \mathcal H)$.
 \end{proof}

Proposition~\ref{prop:basic_prop} establishes $\mathrm{mod} (V)$
as a left ideal of $\mathcal B ( \mathcal H)$.
The conditions by which bounded operators act on the right of $\mathrm{mod} (V)$ are more subtle. 

If $X$ is a set, let $\chi_E$ denote the indicator
function of a subset $E \subset X$.

\begin{lemma} \label{equiv} Suppose $V : \mathcal H \to \mathcal H$ is a positive bounded operator with $\|V\|\le 1$.  Then $T \in \mmod(V)$ iff
\begin{equation}\label{mod2}\|T \chi_{[0,2^{-n}]}(V) \|_{\mathcal L_2} =
O(2^{-n/2}) . \end{equation}
\end{lemma}

\begin{proof}
Let $f,g$ be real-valued bounded Borel functions such that $|f| \leq |g|$ and $Tg(V) \in \mathcal L_2$.
Then
$$
|f(V)T^*| = \sqrt{T|f|^2(V)T^*} \leq \sqrt{T|g|^2(V)T^*} = |g(V)T^*|.
$$
Hence $\| Tf(V) \|_{\mathcal L_2} \leq \| Tg(V) \|_{\mathcal L_2}.$
Let $f(x) = \chi_{[0,2^{-n}]}(x)$, $x \geq 0$, and $g(x) = 2(1+2^nx)^{-1}$, $x\geq0$.
Then $|f| \leq |g|$ and
$$
\|T \chi_{[0,2^{-n}]}(V)\|_{\mathcal L_2}\lec \|T(1+2^nV)^{-1}\|_{\mathcal L_2} \lec 2^{-n/2}.
$$
Hence~\eqref{modulated} implies~\eqref{mod2}.

Conversely, we note that $$(\chi_{[0,2^{-(j-1)}]}-\chi_{[0,2^{-j}]})(x)(1+tx)^{-1} \leq
(1+t2^{-j})^{-1}\chi_{[0,2^{-(j-1)}]}(x), \qquad x \geq 0$$
Then if~\eqref{mod2} holds and $2^{k-1}\le t<2^k$ where $k\in\mathbb N$, we have
\begin{align*} \|T(1+tV)^{-1}\|_{\mathcal L_2}&\le \sum_{j=1}^k\|T(\chi_{[0,2^{-(j-1)}]}-\chi_{[0,2^{-j}]})(V))(1+tV)^{-1}\|_{\mathcal L_2} \\
& \qquad + \|T\chi_{[0,2^{-k}]}(V)(1+tV)^{-1}\|_{\mathcal L_2}\\
&\lec \sum_{j=1}^k (1+t2^{-j})^{-1} \|T\chi_{[0,2^{-(j-1)}]}(V)\|_{\mathcal L_2} + 2^{-k/2}\\
&\lec t^{-1}\sum_{j=1}^k 2^{j}2^{-(j-1)/2} + 2^{-k/2}\\
&\lec t^{-1/2},\end{align*} and the condition for being modulated is satisfied.
 \end{proof}

\begin{proposition}\label{changeV}  Suppose $V_1:\mathcal H\to\mathcal H$ and $V_2:\mathcal H' \to\mathcal H'$ are bounded positive operators.  Let $B: \mathcal H \to \mathcal H'$
be a bounded operator and let $A:\mathcal H' \to\mathcal H$ be a bounded operator such that
for some $a>1/2$ we have
 $$ \|V_1^aAx\|_{\mathcal H} \lec \|V_2^ax\|_{\mathcal H'}, \qquad x\in\mathcal H'.$$
If $T \in \mmod(V_1)$ then $BTA \in \mmod(V_2)$.
 \end{proposition}
 
\begin{remark}
In particular, if $\mathcal H = \mathcal H'$, $V=V_1=V_2$, $B = 1_{\mathcal H}$,
$\|VAx\| \lec \|Vx\|$ for all $x \in \mathcal H$,
then $TA \in \mmod(V)$ if $T \in \mmod(V)$.
\end{remark}
 
 \begin{proof}  We may suppose that $\|V_1\|_{\mathcal B (\mathcal H)},\|V_2\|_{\mathcal B (\mathcal H')}\le 1.$  Let $P_n=\chi_{[0,2^{-n}]}(V_1)$ and $Q_n=\chi_{[0,2^{-n}]}(V_2)$, $n= \mathbb Z_+.$  If $j\le k$ we have
$$\|(I-P_j)AQ_kx\|_{\mathcal H } \le 2^{ja}\|V_1^aAQ_kx\|_{\mathcal H } \leq 2^{ja}\|V_2^aQ_kx\|_{\mathcal H' }\le 2^{(j-k)a}\|x\|_{\mathcal H' } , \qquad x\in\mathcal H'.$$
Thus
$$ \|(I-P_j)AQ_k\|_{\mathcal B(\mathcal H',\mathcal H)} \leq 2^{(j-k)a}.$$
If $T$ is $V_1$-modulated then we have by Lemma \ref{equiv} that
$$ \|TP_j\|_{\mathcal L_2(\mathcal H)}\lec  2^{-j/2}, \qquad j \in \mathbb N.$$
Thus
\begin{align*} \|BTAQ_k\|_{\mathcal L_2(\mathcal H')}&\le \| B \|_{\mathcal B(\mathcal H, \mathcal H')} \|TP_kAQ_k\|_{\mathcal L_2(\mathcal H, \mathcal H')} \\
& \qquad \qquad + \| B \|_{\mathcal B(\mathcal H, \mathcal H')}\sum_{j=1}^k \|T(P_{j-1}-P_j)AQ_k\|_{\mathcal L_2(\mathcal H, \mathcal H')} \\
&\lec \|TP_k\|_{\mathcal L_2(\mathcal H)} \| AQ_k\|_{\mathcal B(\mathcal H', \mathcal H)} \\
& \qquad \qquad +\sum_{j=1}^k \|T(P_{j-1}-P_j)\|_{\mathcal L_2(\mathcal H)}\|(P_{j-1}-P_j)AQ_k\|_{\mathcal B(\mathcal H', \mathcal H)}\\
&\lec \|TP_k\|_{\mathcal L_2(\mathcal H)} 
+\sum_{j=1}^k \|TP_{j-1}\|_{\mathcal L_2(\mathcal H)}\|(1-P_{j-1})AQ_k\|_{\mathcal B(\mathcal H', \mathcal H)} \\ & \qquad \qquad +\sum_{j=1}^k \|TP_{j-1}\|_{\mathcal L_2(\mathcal H)}\|(1-P_{j})AQ_k\|_{\mathcal B(\mathcal H', \mathcal H)} \\
&\lec 2^{-k/2} + (1+2^{-a})2^{1/2} 2^{-k/2}\sum_{j=1}^k 2^{(a-1/2)(j-k)}\\
&\lec 2^{-k/2} .\end{align*}
Hence $BTA$ is $V_2$-modulated by Lemma~\ref{equiv}.
 \end{proof}

\subsection{Proof of Theorem~\ref{connection}}

To prove Theorem~\ref{connection}, we will need the following lemmas.

\begin{lemma}\label{trace} Let $E$ be an $n$-dimensional Hilbert space and suppose $A:E\to E$ is a linear map.  Then there is an orthonormal basis $(f_j)_{j=1}^n$ of $E$ so that
$$(Af_j,f_j)=\frac1n\Tr(A), \qquad j\in \mathbb N.$$
\end{lemma}
\begin{proof} This follows from the Hausdorff-Toeplitz theorem on the convexity of the numerical range $W(A)$.  Suppose $(f_j)_{j=1}^k$ is an orthonormal sequence of maximal cardinality such that $(Af_j,f_j)=\frac1n\Tr(A)$ for $j=1,2\ldots,k.$  Assume $k<n$ and let $F$ be the orthogonal complement of $[f_j]_{j=1}^k$ (here $k=0$ is permitted and then $F=E$).  Let $P$ be the orthogonal projection onto $F$ and consider $PA:F\to F.$  Then $\Tr(PA)=(1-k/n)\Tr(A)$ and by the convexity of $W(PA)$ we can find $f_{k+1}\in F$ with $(Af_{k+1},f_{k+1})=(n-k)^{-1}\Tr ( PA)=\frac1n\Tr(A),$ giving a contradiction.
 \end{proof}

If $a:= \{ a_{n} \}_{n=1}^\infty$ is a sequence of complex numbers let
$a^*$ denote the sequence of absolute values $|a_n|$, $n \in \mathbb N$, arranged to be decreasing. The weak-$\ell_p$ spaces, $p \geq 1$, are defined by
$$
\ell_{p,\infty} := \{ \{ a_{n} \}_{n=1}^\infty | a^* = O(n^{-1/p}) \} .
$$
Let $\mathcal L_{p,\infty}$, $p \geq 1$, denote the two-sided ideal of compact operators $T:\mathcal H\to\mathcal H$ such that
$s_n(T)=O(n^{-1/p})$ (i.e.~$\diag( \mathcal L_{p,\infty} ) = \ell_{p,\infty}$),
with quasi-norm $$ \|T\|_{\mathcal L_{p,\infty}} :=\sup_n n^{1/p}s_n(T).$$
Here, as always, $\{ s_n(T) \}_{n=1}^\infty$ denotes the singular values of $T$.
The ideal $\mathcal L_{p,\infty}$ is a quasi-Banach (hence geometrically stable) ideal.

\begin{lemma} \label{holder}
If $p,q \geq 1$ such that $p^{-1} + q^{-1} = 1$ then
$\mathcal L_{1,\infty} = \mathcal L_{p,\infty} \mathcal L_{q,\infty}$.
\end{lemma}
\begin{proof}
Suppose $A \in \mathcal L_{p,\infty}$ and $B \in \mathcal L_{q,\infty}$.
Using an inequality of Fan, \cite{Fan1951},
$$s_{2n}(AB) \leq s_{n+1}(A)s_n(B) \leq s_{n}(A)s_n(B)
= O(n^{-1/p})O(n^{-1/q}) = O(n^{-1}).$$
Similarly
$s_{2n}(BA) = O(n^{-1})$.
Hence, $AB,BA \in \mathcal L_{1,\infty}$, and
$\mathcal L_{p,\infty} \mathcal L_{q,\infty} \subset \mathcal L_{1,\infty}$.

However $\diag \{ n^{-1} \}_{n=1}^\infty  = \diag \{ n^{-1/p} \}_{n=1}^\infty \diag \{n^{-1/q}\}_{n=1}^\infty$.  So $\diag \{ n^{-1} \}_{n=1}^\infty \in \mathcal L_{p,\infty} \mathcal L_{q,\infty}$.  Since $\mathcal L_{1,\infty}$ is the smallest two-sided ideal that contains $\diag \{ n^{-1} \}_{n=1}^\infty$
then $\mathcal L_{1,\infty} \subset \mathcal L_{p,\infty} \mathcal L_{q,\infty}$.

By the two inclusions
$\mathcal L_{1,\infty} = \mathcal L_{p,\infty} \mathcal L_{q,\infty}$. 
 \end{proof}

\begin{lemma}\label{temper0}  Suppose $0<p<2$ and that $(e_n)_{n=1}^{\infty}$ is an orthonormal basis of $\mathcal H$. Suppose $( v_n )_{n=1}^{\infty}$ is a sequence in $\mathcal H$ such that
\begin{equation}\label{tempered} \sum_{j=n+1}^{\infty}\|v_j\|^2  = O(n^{1-\frac2p}) . \end{equation}
Then $$Tx=\sum_{j=1}^{\infty}(x,e_j)v_j$$ defines an operator $T\in\mathcal L_{p,\infty}$
.
\end{lemma}
\begin{proof}
Observe that $\sum_{n=1}^{\infty}\|v_n\|^2<\infty$ so that $T:\mathcal H\to\mathcal H$ is bounded and $T\in\mathcal L_2.$  For $n=\mathbb Z_+$, let $$T_nx:=\sum_{j=n+1}^{\infty}(x,e_j)v_j.$$
Each $T_n$ is also in $\mathcal L_2$.  Recalling that (see~\cite[Theorem~7.1]{GohKre1969-MR0246142})
$$
\sum_{j = n+1}^\infty s_j^2(T) = \min \{ \| T-K \|_{\mathcal L_2}^2\ |\ \ \rank (K)
\leq n \} ,
$$
we have
\begin{align*}
n s_{2n}(T)^2 & \leq \sum_{j=n}^{2n} s_j(T)^2
\leq \sum_{j=1}^\infty s_j(T_n)^2 \\
& = \| T_n \|_{\mathcal L_2}^2 = \sum_{j=1}^\infty \| T_n e_j \|^2 = \sum_{j=n+1}^\infty \| v_j \|^2 \\
& \lec n^{1-2/p} .
\end{align*}
Hence $s_{2n}(T)\lec n^{-1/p}.$
 \end{proof}

\begin{lemma}\label{maincrux}  Suppose $T : \mathcal H \to \mathcal H$ is a bounded operator
and $(f_n)_{n=1}^{\infty}$ is an orthonormal basis of $\mathcal H$ such that
$$
\sum_{j=n+1}^\infty \| Tf_j \|^2 = O( n^{-1})
$$
and
$$
| (Tf_n,f_n) | = O(n^{-1}) .
$$
Then $T,\diag \{ (Tf_n,f_n) \}_{n=1}^\infty \in \mathcal L_{1,\infty}$,
$$
T - \diag \{ (Tf_n,f_n) \}_{n=1}^\infty \in \Com \mathcal L_{1,\infty}
$$
and
$$
\sum_{j=1}^n\lambda_j(T)-\sum_{j=1}^n(Tf_j,f_j) = O(1) .
$$
\end{lemma}

\begin{proof}
Suppose $1<p<2$ is fixed and $2 < q < \infty$ is such that $\frac1p+\frac1{q}=1$.  Set
$$Sx:=\sum_{j=1}^{\infty}j^{1/p-1}(x,f_j)f_j, \qquad x\in\mathcal H .$$
Clearly $S$ is Hermitian, $S\in\mathcal L_{q,\infty}$ ($S$ has singular values $n^{-1/q}$) and $Sf_n = n^{-1/q} f_n$ (so $(f_n)_{n=1}^\infty$ is an eigenvector sequence for $S$).
Now set
$$Ax:=\sum_{j=1}^{\infty}j^{1-1/p}(x,f_j)Tf_j, \qquad x\in\mathcal H.$$
We show that $A \in \mathcal L_{p,\infty}$.  Set $v_j = j^{1-1/p}Tf_j$.  Notice that
\begin{align*}
\sum_{j=n+1}^{\infty}j^{2-2/p}\|Tf_j\|^2 & = \sum_{k=0}^\infty \sum_{j=2^kn+1}^{2^{k+1}n} j^{2-2/p} \| Tf_j \|^2 \\
& \lec \sum_{k=0}^\infty (2^{k+1})^{2-2/p} n^{2-2/p} 2^{-k}n^{-1} \\
& \lec[p] n^{1-2/p} \sum_{k=0}^\infty 2^{k(1-2/p)} \\
& \lec[p] n^{1-2/p} .
\end{align*}
Then we have
\begin{equation} \label{eq:A}
\sum_{j=n+1}^{\infty}\|v_j\|^2 = \sum_{j=n+1}^{\infty}j^{2-2/p}\|Tf_j\|^2 \lec[p] n^{1-2/p}.
\end{equation}
Hence, by an application of Lemma~\ref{temper0}, $A \in \mathcal L_{p,\infty}$.

By construction $AS=T$ and by assumption $|(ASf_n,f_n)| = |(Tf_n,f_n)| = O(n^{-1})$.  Thus Theorem~\ref{main} can be applied to $A \in \mathcal L_{p,\infty} = \Com \mathcal L_{p,\infty}$ (this last equality follows easily from Corollary~\ref{DFWWCor}) and $S \in \mathcal L_{q,\infty}$, i.e.~we use $\mathcal I_1 = \mathcal L_{p,\infty}$ and $\mathcal I_2 = \mathcal L_{q,\infty}$, noting from Lemma~\ref{holder} that $\mathcal I = \mathcal L_{1,\infty} = \mathcal L_{p,\infty} \mathcal L_{q,\infty}$.  Hence, from Theorem~\ref{main},
$T=AS \in \mathcal L_{1,\infty}$,
$D := \diag \{ (Tf_n,f_n) \}_{n=1}^\infty = \diag \{ (ASf_n,f_n) \}_{n=1}^\infty \in \mathcal L_{1,\infty}$, and
$$T - D = AS - \diag \{ (ASf_j,f_j) \}_{j=1}^\infty  \in \Com \mathcal L_{1,\infty}.$$
By Theorem~\ref{K}
$$
\sum_{j=1}^n\lambda_j(T)-\sum_{j=1}^n \lambda_j(D) = O(1) .
$$
By Lemma~\ref{dec}
$$
\sum_{j=1}^n\lambda_j(D)-\sum_{j=1}^n (Tf_j,f_j) = O(1)
$$
and the results are shown.
 \end{proof}

\begin{proof}[Proof of Theorem~\ref{connection}]
Let $T \in \mmod(V)$ where $0 \leq V \in \mathcal L_{1,\infty}$
and $(e_n)_{n=1}^{\infty}$ be an orthonormal basis of $\mathcal H$ such that $V e_n = s_n(V) e_n$. Since $s_n(V) \lec n^{-1}$ we have that
\begin{equation} \label{proofconnection1}
\left(\sum_{j=n+1}^{\infty}\|Te_j\|^2\right)^{1/2} \leq (1+ns_n(V))\|T(1+nV)^{-1}\|_{\mathcal L_2}\lec n^{-1/2} .
\end{equation}
Then we have
$$
\sum_{j=n+1}^\infty \| Te_j \|^2 \lec n^{-1} .
$$
Let $D:=\diag\{(Te_n,e_n)\}_{n=1}^{\infty}$ be the specific diagonalisation with respect to the basis $( e_n )_{n=1}^\infty$, i.e.
$ D =  \sum_{n=1}^\infty (Te_n,e_n) e_n e_n^*$.  Note that if $D':=\diag\{(Te_n,e_n)\}_{n=1}^{\infty}$ is the diagonalisation
according to any arbitrary orthonormal basis $( h_n )_{n=1}^\infty$, i.e.
$D' =  \sum_{n=1}^\infty (Te_n,e_n) h_n h_n^*$,
then there exists a unitary $U$ with $h_n = Ue_n$, $n \in \mathbb N$,
and thus $D' = UDU^*$.
Since $\|De_j\|\le \|Te_j\|$ then we also have
$$
\sum_{j=n+1}^\infty \| De_j \|^2 \lec n^{-1} .
$$
Thus
$$
\sum_{j=n+1}^\infty \| (T-D)e_j \|^2 \lec 2 n^{-1} .
$$

(i) By Lemma~\ref{temper0}, $T,D \in \mathcal L_{1,\infty}$ (set $v_j=Te_j$ and $v_j = De_j$
respectively).
It follows that $D' \in \mathcal L_{1,\infty}$
for any diagonalisation $D'$ since $D' = UDU^*$ for a unitary $U$
and $\mathcal L_{1,\infty}$ is a two-sided ideal.

(ii) Notice by design that $(Te_j,e_j) = (De_j,e_j)$, $j \in \mathbb N$.
Thus $T-D$ satisfies Lemma~\ref{maincrux} (where $((T-D)e_j,e_j)=0$, $j \in \mathbb N$)
and $T-D \in \Com \mathcal L_{1,\infty}$.
If $D':=\diag\{(Te_n,e_n)\}_{n=1}^{\infty}$ is an arbitrary diagonalisation then
$D' = UDU^*$ for a unitary $U$ and clearly $D'-D \in \Com \mathcal L_{1,\infty}$. Hence $T-D' \in \Com \mathcal L_{1,\infty}$.

(iii) Construct a new basis $(f_n)_{n=1}^{\infty}$ of $\mathcal H$ using Lemma~\ref{trace} such that $(f_j)_{j=2^{k-1}}^{2^k-1}$ is a basis of $[e_j]_{j=2^{k-1}}^{2^k-1}$, $k = \mathbb Z_+$. Let $P_k$ denote the projection onto $[e_j]_{j=2^{k-1}}^{2^k-1}$.
If $2^{k-1}\le n\le 2^k-1$ then
\begin{equation}
\sum_{j=n+1}^{\infty} \|Tf_j\|^2 \leq \sum_{j=2^{k-1}}^{\infty} \|Te_j\|^2
\lec n^{-1} \label{eq:modulatedfact-1}
\end{equation}
Similarly
$$
\| TP_k \|_{\mathcal L_2}^2 = \sum_{j=2^{k-1}}^{2^k-1} \|Te_j\|^2
\lec 2^{-k}
$$
Thus
$$
\|P_kTP_k\|_{\mathcal L_1}
\le \|TP_k\|_{\mathcal L_2}\|P_k\|_{\mathcal L_2}
\lec 2^{-k/2} 2^{k/2} \lec 1.
$$
Hence, if $2^{k-1}\le n\le 2^k-1$, then
\begin{equation}\label{upperbound}
|(Tf_n,f_n)| \leq 2^{-k} \|P_kTP_k\|_{\mathcal L_1} \lec n^{-1},\qquad n \in \mathbb N .
\end{equation}
Using~\eqref{eq:modulatedfact-1} and~\eqref{upperbound},
from Lemma~\ref{maincrux} we obtain
\begin{equation}\label{eigenvalueua}
\left|\sum_{j=1}^n \lambda_j(T) - \sum_{j=1}^n (Tf_j,f_j) \right| \lec 1 , \qquad n \in \mathbb N .
\end{equation}
Now, if $2^{k-1}\le n\le 2^k-1$, then
\begin{align} \left|\sum_{j=1}^n(Te_j,e_j)-\sum_{j=1}^n(Tf_j,f_j)\right| &= \left|\sum_{j=2^{k-1}}^n (Te_j,e_j)-\sum_{j=2^{k-1}}^{n}(Tf_j,f_j)\right| \nonumber \\
&\le 2\|P_k TP_k\|_{\mathcal L_1} \lec 1 \label{eigenvalueuc} .
\end{align}
Hence (iii) is shown from~\eqref{eigenvalueua} and~\eqref{eigenvalueuc}.   \end{proof}

\subsection{Corollaries of interest}

Before we specialise to operators modulated by
the Laplacian we note some results of interest.

\begin{corollary}\label{connectioncor}  Let $\{ T_i \}_{i=1}^N$ be a finite collection of $V$-modulated operators where $0 < V\in\mathcal L_{1,\infty}$.  Then 
\begin{equation} \label{eq:Kalton_condition3}
\sum_{j=1}^{n} \lambda_j \left( \sum_{i=1}^N T_i \right) 
 - \sum_{i=1}^N \sum_{j=1}^{n} \lambda_j( T_i )  = O(1) .
\end{equation}
\end{corollary}
\begin{proof}  Let $T_0 = \sum_{i=1}^N T_i \in \mmod(V)$.
Choose an eigenvector sequence $(e_n)_{n=1}^\infty$ of $V$ with $Ve_n = s_n(V)e_n$, $n \in \mathbb N$.
Then, by Theorem~\ref{connection}, for $i=0,\ldots,N$, there are constants $C_i$ such that
$$
\left| \sum_{j=1}^{n} \lambda_j(T_i)
 - \sum_{j=1}^{n} (T_ie_j,e_j) \right| < C_i,\ \ 0 \leq i \leq N.
$$
Hence
\begin{align*}
\left| \sum_{j=1}^{n} \left( \lambda_j( T_0 )  - \sum_{i=1}^N  \lambda_j( T_i ) \right) \right| &< C + \left| \sum_{j=1}^{n} \left( (T_0e_j,e_j)  - \sum_{i=1}^N (T_ie_j,e_j) \right) \right| =  C
\end{align*}
where $C = C_0 + \ldots + C_N$.
 \end{proof}

In the proof of Theorem~\ref{connection} we noted that a $V$-modulated operator, $T \in \mmod(V)$, satisfied the condition
$$
\sum_{k=n+1}^{\infty}\|Te_k\|^2 = O(n^{-1})
$$
for an eigenvector sequence $(e_n)_{n=1}^\infty$ of $0 < V \in \mathcal L_{1,\infty}$.  We show a converse statement.

\begin{proposition}\label{connectioncor2}  Suppose $0 < V \in \mathcal L_{1,\infty}$
  is such that the singular values of $V$ satisfy $s_n(V) = \Theta(n^{-1})$ and $(e_n)_{n=1}^\infty$ is an
  orthonormal basis of $\mathcal H$ such that $Ve_n = s_n(V)e_n$, $n \in \mathbb N$.
  Then $T \in \mmod(V)$ iff
\begin{equation} \label{eq:Kalton_condition2}
 \sum_{k=n+1}^{\infty}\|Te_k\|^2 = O(n^{-1}) .
\end{equation}
\end{proposition}
\begin{proof}
The only if statement is contained in the proof of Theorem~\ref{connection}.  We show the if
statement.  Without loss $\| V \| \leq 1$.

Clearly $T$ is bounded and Hilbert-Schmidt.
Let $c>0$ be such that $s_n(V) \geq cn^{-1}$.  Then
$$
s_{\left\lfloor c k \right\rfloor} (V) > \frac{c}{\left\lfloor c k \right\rfloor} > \frac{c}{c k} = k^{-1}.
$$
Hence $N(k^{-1}) \geq \left\lfloor c k \right\rfloor \gec k$ where
$$N(\lambda) =  \max \{ k \in \mathbb N | s_k(V) > \lambda \}$$
and we have
$$
N(k^{-1})^{-1} \lec k^{-1} .
$$
Now note
$$
\| T \chi_{[0,k^{-1}]}(V) \|_2^2 = \sum_{j= N(k^{-1}) + 1}^\infty \| Te_j \|^2
\lec N(k^{-1})^{-1}
\lec k^{-1} .
$$
Thus
$\| T \chi_{[0,2^{-n}]}(V) \|_2 \lec 2^{-n/2}$, $n \in \mathbb N$.  By Lemma~\ref{equiv}, $T \in \mmod(V)$.
 \end{proof}

\section{Applications to pseudo-differential operators} \label{sec:PsDO}

We define in this section operators that are modulated with respect to
the operator $(1-\Delta)^{-d/2}$ where $\Delta=\sum_{i=1}^d \frac{\partial^2}{\partial x_i^2}$ is the Laplacian on $\mathbb R^d$, termed by us Laplacian modulated operators.  We show that the Laplacian modulated operators
include the class of pseudo-differential operators of order $-d$, and that
the Laplacian modulated operators admit a residue map which extends the noncommutative residue.  Finally we show, with the aid of the results established,
that singular traces applied to Laplacian modulated operators calculate
the residue.

\begin{definition} \label{def:Lapmod}
Suppose $d\in \mathbb N$ and that $T:L_2(\mathbb R^d)\to L_2(\mathbb R^d)$ is a bounded operator.  We will say that $T$ is {\it Laplacian modulated} if $T$ is $(1-\Delta)^{-d/2}$-modulated.
\end{definition}

From Proposition~\ref{rem:leftideal} a Laplacian modulated operator is Hilbert-Schmidt.  We recall every Hilbert-Schmidt operator on $L_2(\mathbb R^d)$ has a unique symbol in the following sense.

\begin{lemma} \label{symbol}
A bounded operator $T:L_2(\mathbb R^d)\to L_2(\mathbb R^d)$ is
Hilbert-Schmidt iff there exists a unique function $p_T \in L_2( \mathbb R^d \times \mathbb R^d)$ such that
\begin{equation}
(Tf)(x)=\frac{1}{(2\pi)^d}\int_{\mathbb R^d}e^{i\langle x,\xi\rangle} p_T(x,\xi)\hat f(\xi)d\xi , \qquad f \in L_2(\mathbb R^d). \label{LM} 
\end{equation}
Further, $\| T \|_{\mathcal L_2} = (2\pi)^{d/2} \| p_T \|_{L_2}$ and if $\{ \phi_n \}_{n=1}^\infty \subset C_c^\infty(\mathbb R^d)$ is such that $\phi_n \nearrow 1$ pointwise, then 
$$
p_T(x,\xi) = \lim_n e_{-\xi}(x)(T\phi_n e_{\xi})(x) , \qquad x,\xi \ \mathrm{a.e.},
$$
where $e_{\xi}(x)=e^{i\langle x,\xi\rangle}.$
\end{lemma}
\begin{proof}
It follows from Plancherel's theorem that $T$ is Hilbert Schmidt iff
it can be represented in the form~\eqref{LM} for some unique $p_T\in L_2(\mathbb R^d \times\mathbb R^d)$ and that $\| T \|_{\mathcal L_2} = (2\pi)^{d/2} \| p_T \|_{L_2}$.
We have
$$
(T\phi_n e_{\xi})(x) = \frac{1}{(2\pi)^d} \int_{\mathbb R^d} e^{i\langle x,\eta\rangle} p_T(x,\eta) \hat{\phi_n}(\eta - \xi) d \eta ,  \qquad x \ \mathrm{a.e.}
$$
If $\phi_n \nearrow 1$ pointwise then $\hat{\phi_n} \to (2\pi)^d \delta$ in the sense
of distributions, where $\delta$ is the Dirac distribution.
If we set $g_{x}(\eta):= e^{i\langle x,\eta \rangle} p_T(x,\eta)$, then $g_{x} \in L_2(\mathbb R^d)$ is a tempered distribution.  Hence $g_x \star \delta = g_x$ and
$$
\lim_n e_{-\xi}(x)(T\phi_n e_{\xi})(x) =   e_{-\xi}(x) (g_x \star \delta)(\xi) = p_T(x,\xi) \qquad x,\xi \ \mathrm{a.e.}
$$
 \end{proof}

The function $p_T$ in~\eqref{LM} is called the {\it symbol} of the Hilbert-Schmidt operator $T.$ Being Laplacian modulated places the following condition on the symbol.

\begin{proposition}\label{Laplacemodulated} Suppose $d\in \mathbb N$ and that $T:L_2(\mathbb R^d)\to L_2(\mathbb R^d)$ is a bounded operator.  Then $T$ is Laplacian modulated iff $T$ can be represented in the form
 \begin{equation}\label{kernel}
 (Tf)(x)= \frac{1}{(2\pi)^d}\int_{\mathbb R^d}e^{i\langle x,\xi\rangle} p_T(x,\xi)\hat f(\xi)d\xi
 \end{equation}
 where $p_T \in L_2(\mathbb R^d \times \mathbb R^d)$ is such that
 \begin{equation} \label{LM1}
\left( \int_{\mathbb R^d} \int_{|\xi|\ge t} |p_T(x,\xi)|^2d\xi\,dx \right)^{1/2} = O(t^{-d/2}), \qquad t\ge 1.\end{equation} \end{proposition}

 \begin{proof}
If $T$ is Laplacian modulated or if it satisfies~\eqref{kernel},
then $T$ is Hilbert-Schmidt.  So we are reduced to showing
a Hilbert-Schmidt operator $T$ is Laplacian modulated iff its symbol
$p_T$ satisfies~\eqref{LM1}.

If $Q_t$, $t > 0$, is the Fourier projection
$$
(Q_t f)(\xi) = \int_{\left| \eta \right| \geq t} \hat f(\eta) e^{i\langle \xi,\eta \rangle}\, d\eta
$$
then the Hilbert-Schmidt operator $TQ_t$ has the form
$$
(TQ_t f)(\xi) = \int_{|\xi| \geq t}e^{i\langle x,\xi\rangle} p_T(x,\xi)\hat f(\xi)d\xi .
$$
By Lemma~\ref{symbol}
$$
\left \| TQ_t \right \|_{\mathcal L_2} = \left(\int_{\mathbb R^d}  \int_{\left| \xi \right| \geq t} 
  \left| p_T(x, \xi) \right|^2\, d x d\xi \right)^{1/2} .
$$
Define also
$$
(P_t f)(\xi) = \int_{(1 + \left| \eta \right|^2)^{-d/2} \leq t} \hat f(\eta) e^{i\langle \xi,\eta \rangle}\, d\eta, 
$$
or, via Fourier transform,
$$
P_t = \chi_{[0, t]} \left(  (1 - \Delta)^{-d/2} \right).
$$
Note that, for $t \geq 1$,
$$
\left| \eta \right| \geq t \Rightarrow (1+\left| \eta \right|^2)^{-d/2}  \leq t^{-d}
$$
and that
$$
\left| \eta \right| \geq 2^{-1/2} t \Leftarrow ( 1 + \left| \eta \right|^2)^{-d/2} \leq t^{-d}.
$$
Hence
$$
Q_{t} \leq P_{t^{-d}} \leq Q_{2^{-1/2}t}.
$$
Note also that~$P_{t} \leq P_{t_1}$ and $Q_t \leq Q_{t_1}$ if~$t_1 \leq t$.  Fix~$t \geq 1$ and~$n \in \mathbb{Z}$ such that~$2^{n-1}
\leq t < 2^{n}$.

We now prove the if statement. Suppose $T$ is Laplacian modulated.  Then, by
Lemma~\ref{equiv},
$$
\left\| T Q_t \right\|_{\mathcal L_2} \leq \left\| T P_{t^{-d}}\right\|_{\mathcal L_2}
\leq \left\| T P_{2^{-dn}}\right\|_{\mathcal L_2}
\lec 2^{-dn/2}
\lec t^{-d/2} .
$$
Hence~\eqref{LM1} is satisfied.

We prove the only if statement.  Let $T$ satisfy~\eqref{LM1}.
Then
$$
\left\| T P_{2^{-n}} \right\|_{\mathcal L_2} \leq \left\| T Q_{2^{-1/2} 2^{n/d}}\right\|_{\mathcal L_2} \lec 2^{d/4} 2^{-n/2}
\lec 2^{-n/2} .
$$
By Lemma~\ref{equiv}, $T$ is Laplacian modulated.
 \end{proof}

For $p \in L_2(\mathbb R^d \times \mathbb R^d)$ define
\begin{equation} \label{eq:add2}
\| p \|_{\mmod} := \| p \|_{L_2} +  \sup_{t \geq 1} t^{d/2} \left(  \int_{\mathbb R^d} \int_{|\xi| \geq t} |p(x,\xi)|^2 d\xi dx \right)^{1/2}.
\end{equation}
Define
\begin{equation} \label{eq:add3}
S^{\mmod} := \{ p \in L_2(\mathbb R^d \times \mathbb R^d) \, | \, \| p \|_{\mmod} < \infty \}.
\end{equation}
Proposition~\ref{Laplacemodulated} says that each Laplacian modulated
operator $T$ is associated uniquely to $p_T \in
S^{\mmod}$ and vice-versa.  We can call $S^\mmod$ the symbols of
Laplacian modulated operators.

If $\phi \in C_c^\infty(\mathbb R^d)$ define the multiplication operator $(M_\phi f)(x) = \phi(x) f(x)$, $ f \in L_2(\mathbb R^d)$.

\begin{definition} \label{def:compbase}
We will say a bounded operator $T : L_2(\mathbb R^d) \to L_2(\mathbb R^d)$ is \emph{compactly based} if $M_\phi T = T$ for some $\phi \in C_c^\infty(\mathbb R^d)$, and
\emph{compactly supported} if $M_\phi T M_\phi = T$.
\end{definition}

We omit proving the easily verified statements that a Hilbert-Schmidt operator $T$
is compactly based if and only if $p_T(x,\xi)$ is (almost everywhere)
compactly supported in the $x$-variable, and is compactly supported
if and only if the kernel of $T$ is compactly supported.

\begin{example}[Pseudo-differential operators]
We recall that $\langle \xi \rangle := (1+|\xi|^2)^{1/2}$, $\xi \in \mathbb R^d$, and a multi-index of order $|\beta|$ is $\beta = (\beta_1,\ldots, \beta_d) \in (\mathbb N \cup \{ 0 \})^d$ such that $|\beta| := \sum_{i=1}^d \beta_i$.   If
$p \in C^\infty(\mathbb R^d \times \mathbb R^d)$ such that, for each multi-index $\alpha$, $\beta$,
\begin{equation}
| \partial^\alpha_x \partial^\beta_\xi p(x,\xi) | \lec[\alpha, \beta] \langle \xi \rangle^{m-|\beta|}
\end{equation}
we say that $p$ belongs to the symbol class $S^m := S^m(\mathbb R^d \times \mathbb R^d)$, $m \in \mathbb R$, (in general terminology we have just defined the
uniform symbols of H\"{o}rmander type (1,0), see e.g.~\cite{HorIII} and~\cite[Chapter 2]{Ruzhansky2010}).  
If $\mathcal S (\mathbb R^d)$ denotes the Schwartz functions (the smooth functions of rapid decrease),
an operator $P: \mathcal S (\mathbb R^d) \to \mathcal S (\mathbb R^d)$ associated to a symbol $p \in S^m$,
\begin{equation} \label{eq:psdoform}
(Pu)(x)=\frac{1}{(2\pi)^d}\int_{\mathbb R^d}e^{i\langle x,\xi\rangle} p(x,\xi)\hat u(\xi)d\xi , \qquad u \in \mathcal S(\mathbb R^d) 
\end{equation}
is called a pseudo-differential operator of order $m$.

If $H^s(\mathbb R^d)$, $s \in \mathbb R$, are the Sobolev Hilbert spaces consisting of those $f \in L_2(\mathbb R^d)$ with
$$
\| f \|_s := \| (1-\Delta)^{s/2} f \|_{L_2} < \infty
$$
and $P$ is a pseudo-differential operator of order $m$, then $P$ has an extension to a continuous linear operator
\begin{equation} \label{eq:sobolevcont}
P : H^s(\mathbb R^d) \to H^{s-m}(\mathbb R^d) , \qquad s \in \mathbb R ,
\end{equation}
see, e.g.~\cite[Theorem 2.6.11]{Ruzhansky2010}.  If $P$ is order $0$ this implies
$P$ has a bounded extension $P : L_2(\mathbb R^d) \to L_2(\mathbb R^d)$.

The compactly based Laplacian modulated operators extend the compactly
based pseudo-differential operators of order $-d$.

\begin{proposition} \label{prop:psdoincl}
If $P$ is a compactly based (respectively, compactly supported)
pseudo-differential operator of order $-d$ then the bounded extension of $P$
is a compactly based (respectively, compactly supported) Laplacian modulated operator.
Also, the symbol of $P$
is equal to (provides the $L_2$-equivalence class) of the symbol of the bounded extension of $P$ as a Laplacian modulated operator.
\end{proposition}
\begin{proof}
Let $P$ have symbol $p \in S^{-d}$ that is compactly based in the first variable.  Then
\begin{equation}
\int_{\mathbb R^d} \int_{|\xi|\ge t} |p(x,\xi)|^2 d\xi dx
\lec \int_{|\xi|\ge t} \langle \xi \rangle^{-2d} d\xi
\lec \langle t \rangle^{-d} \label{psdo:LM1} .
\end{equation}
Hence $p \in L_2(\mathbb R^d  \times \mathbb R^d)$ and, if $P_0$ is the extension of $P$ then $P_0$ is Hilbert-Schmidt.
Let $p_0$ be the symbol of $P_0$ as a Hilbert-Schmidt operator.
Let $\phi \in C_c^\infty(\mathbb R^d)$, and $e_{\xi}(x):=e^{i\langle x,\xi\rangle}$, $ \xi \in \mathbb R^d$.
Since
$$
(P-P_0) \phi e_{\xi} = 0
$$
we have, by Lemma~\ref{symbol},
$$
p(x,\xi) - p_0(x,\xi) = \lim_n e_{-\xi}(x) (P-P_0) \phi_n e_{\xi}(x) = 0 , \qquad x, \xi \ \mathrm{a.e.}
$$
where $\{ \phi_n \}_{n=1}^\infty \subset C_c^\infty(\mathbb R^d)$ is such that $\phi_n \nearrow 1$ pointwise.
Then~\eqref{psdo:LM1} implies that the symbol of $P_0$ satisfies~\eqref{LM1}, hence $P_0$ is Laplacian modulated.
 \end{proof}
\end{example}

The Laplacian modulated operators form a bimodule for sufficiently regular operators.

\begin{lemma} \label{lem:module}
Let $T$ be Laplacian modulated and $R,S: L_2(\mathbb R^d) \to L_2(\mathbb R^d)$ be bounded such
that $S : H^s(\mathbb R^d) \to H^s(\mathbb R^d)$ is bounded for some $s < -d/2$. Then $R T S$ is Laplacian modulated.
\end{lemma}
\begin{proof}
The Laplacian modulated operators form a left ideal so $RT$ is Laplacian modulated.
By Proposition~\ref{changeV}, the result is shown if
$
\| (1-\Delta)^{-da/2} S u \|_{L_2} \lec \| (1-\Delta)^{-da/2} u \|_{L_2}
$
for all $u \in C_c^\infty(\mathbb R^d)$ where $a > 1/2$.  However, this is the same statement
as
$
\| S  u \|_{s} \lec \| u \|_{s} 
$
for $s < -d/2$.
\end{proof}

\begin{remark} \label{rem:bimodule}
From~\eqref{eq:sobolevcont}, $R,S : H^s(\mathbb R^d) \to H^s(\mathbb R^d)$ for any $s \in \mathbb R$ for all zero order pseudo-differential operators $R$ and $S$.
Hence the Laplacian modulated operators form a bimodule for the pseudo-differential operators of order $0$.
\end{remark}


The next example confirms that the Laplacian modulated operators
are a wider class than the pseudo-differential operators.

\begin{example}[Square-integrable functions]

For $f,g \in L_2(\mathbb R^d)$ set
$$
M_f : L_\infty(\mathbb R^d) \to L_2(\mathbb R^d) \ , (M_fh)(x) := f(x)h(x) , \qquad x \
\mathrm{a.e.}
$$
and
$$
T_{g} : L_2(\mathbb R^d) \to L_\infty(\mathbb R^d) \ , (T_{g}h)(x) := \frac{1}{(2\pi)^d}
\int_{\mathbb R^d} e^{i\langle x, \xi \rangle } g(\xi) \hat h(\xi) d\xi
= (\hat g \star h)(x) , \qquad x \ \mathrm{a.e.}
$$
Define a subspace $L_\mmod(\mathbb R^d)$ of $L_2(\mathbb R^d)$ by
$$
L_\mmod(\mathbb R^d) := \left\{ g \in L_2(\mathbb R^d) \Big|
\left( \int_{|\xi|\ge t} |g(\xi)|^2 d\xi \right)^{1/2} = O(t^{-d/2}), \ t \geq 1 \right\} .
$$

\begin{remark}
It is clear that the function $\langle \xi \rangle^{-d} = (1+|\xi|^2)^{-d/2}$, $\xi \in \mathbb R^d$,
belongs to $L_\mmod(\mathbb R^d)$.
\end{remark}

\begin{proposition}\label{MultSqFnProp}
  If $f \in L_2(\mathbb R^d)$ and $g \in L_\mmod(\mathbb R^d)$ then
  $M_f T_g$ is Laplacian modulated with symbol $f \otimes g \in S^\mmod$.
  If $f$ has compact support than $M_f T_g$ is compactly based.
\end{proposition}
\begin{proof}
First note that $\| T_{g} h \|_{L_\infty} = \| \hat g \star h \|_{L_\infty} \leq \| g \|_{L_2} \| h \|_{L_2}$ by Young's inequality.  Hence $T_g : L_2 \to L_\infty$ is continuous and linear.
Also $\| M_f h \|_{L_2} \leq \| f \|_2 \| h \|_{L_\infty}$, so $M_f : L_\infty \to L_2$ is continuous and linear.
The composition $M_f T_g : L_2 \to L_2$ is continuous and linear (and everywhere defined).  We have that
$$
(M_fT_g h)(x) = \frac{1}{(2\pi)^d}
\int_{\mathbb R^d} e^{i\langle x,  \xi \rangle} f(x)g(\xi) \hat h(\xi) d\xi
, \qquad x \ \mathrm{a.e.}
$$
is an integral operator with symbol $f \otimes g \in S^\mmod$.
Hence $M_f T_g$ is Laplacian modulated.  If $f$ has compact support, $\phi f = f$
for some $\phi \in C_c^\infty(\mathbb R^d)$, and hence $M_\phi M_f T_g = M_f T_g$.
 \end{proof}
\end{example}

\subsection{Residues of Laplacian modulated operators} \label{sec:6.2}

We define in this section the residue of a compactly based Laplacian modulated
operator.  We show it is an extension of the noncommutative
residue of classical pseudo-differential operators of order $-d$ defined by M.~Wodzicki, \cite{Wod}.

We make some observations about the symbol of a compactly based operator.

 \begin{lemma}\label{compactbase} Let $T$ be a compactly based Laplacian modulated operator with symbol $p_T$.
Then:
\begin{align}
& \label{compactbase1} \int_{\mathbb R^d} \int_{r\le |\xi|\le 2r} | p_T(x,\xi) | d\xi\,dx = O(1), \qquad r\ge 1; & \\
& \label{compactbase2} \int_{\mathbb R^d} \int_{|\xi|\le r} | p_T(x,\xi) | d\xi\,dx =O(\log(1+r)), \qquad r\ge 1; & \\
& \label{compactbase4} \int_{\mathbb R^d} \int_{\mathbb R^d} |p_T(x,\xi)| \langle \xi \rangle^{-\theta} d\xi\,dx <\infty, \qquad \theta>0 ; &
\end{align}
and, if $A$ is a positive $d \times d$-matrix with spectrum contained in $[a,b]$, $b > a > 0$ are fixed values, then
\begin{align}
& \label{compactbase5}
\int_{\mathbb R^d} \int_{|\xi|\le r} p_T(x,\xi) d\xi\,dx -
\int_{\mathbb R^d} \int_{|A\xi|\le r} p_T(x,\xi) d\xi\,dx = O(1), \qquad r\ge 1 . &
\end{align}
\end{lemma}

 \begin{proof} We prove~\eqref{compactbase1}.  Using \eqref{LM1}, if $r \geq 1$,
\begin{multline*}
\int_{\mathbb{R}^d} \int_{r \leq | \xi | \leq 2r}  
|p_T(x,\xi)|  d\xi dx \\
\lec \left( \int_{r \leq | \xi | \leq 2r} d\xi \right)^{1/2} \left(  \int_{\mathbb{R}^d} \int_{| \xi | \geq r} |p_T(x,\xi)|^2 d\xi dx  \right)^{1/2} \lec r^{d/2} r^{-d/2} .
\end{multline*}
We prove~\eqref{compactbase2}.  Fix $n \in \mathbb N$ such that $2^{n-1} \leq r < 2^n$, then  
$$
\{ \xi \in \mathbb R^d | 0 \leq |\xi| \leq r \} \subset [0,1] \cup \cup_{k=1}^n \{ \xi \in \mathbb R^d| 2^{k-1} < |\xi| \leq 2^k \} .
$$
By \eqref{compactbase1} the integral of $|p_T(x,\xi)|$ over each individual set in the union on the right hand side of the previous display is controlled by some constant $C$.  Then the integral over the initial set on the left hand side of the display
is controlled by
$$
C(n + 2) \leq C(3 + \log_2 r) \lec \log(1+r) .
$$
We prove~\eqref{compactbase4}.  Consider
\begin{multline*}
\int_{\mathbb{R}^d} \int_{\mathbb{R}^d} 
|p_T(x,\xi)| (1+|\xi|^2)^{-\theta/2} d\xi dx  \\
\lec \sum_{n=1}^\infty  2^{-n\theta} \int_{\mathbb{R}^d}  \int_{2^{n-1} \leq |\xi | \leq 2^n}  
|p_T(x,\xi)| d\xi dx  \lec \sum_{n=1}^\infty  2^{-n\theta} < \infty
\end{multline*}
by~\eqref{compactbase1}.

We prove~\eqref{compactbase5}.  It follows from
\eqref{compactbase1} that
\begin{multline*}
\left| \int_{\mathbb R^d} \int_{|A\xi| \leq r} p_T(x,\xi)  d\xi dx
- \int_{\mathbb R^d} \int_{|\xi| \leq b^{-1}r} p_T(x,\xi) d\xi dx \right| \\
\leq \int_{\mathbb R^d} \int_{ b^{-1}r \leq |\xi| \leq a^{-1}r } |p_T(x,\xi)| d\xi dx .
\end{multline*}
From \eqref{compactbase1}
$$
\int_{\mathbb R^d} \int_{ b^{-1}r \leq |\xi| \leq a^{-1}r } |p_T(x,\xi)| d\xi dx =O(1) .
$$
Also
\begin{multline*}
\left| \int_{\mathbb R^d} \int_{|\xi| \leq r} p_T(x,\xi) d\xi dx
- \int_{\mathbb R^d} \int_{|\xi| \leq b^{-1}r} p_T(x,\xi) d\xi dx \right| \\
\leq \int_{\mathbb R^d} \int_{ c_b r \leq |\xi| \leq c_b^{-1}r } |p_T(x,\xi)| d\xi dx
\end{multline*}
where $c_b = b$ if $b\leq 1$ and $c_b = b^{-1}$ if $b > 1$.
From \eqref{compactbase1} again
$$
\int_{\mathbb R^d} \int_{ c_b r \leq |\xi| \leq c_b^{-1}r } |p_T(x,\xi)| d\xi dx =O(1).
$$
Formula~\eqref{compactbase5} follows.
 \end{proof}

We notice from~\eqref{compactbase2} that, $n \in \mathbb N$,
$$
\frac{d}{\log(1+n)} \int_{\mathbb R^d} \int_{|\xi|\le n^{1/d}} p_T(x,\xi) d\xi\,dx = O(1)
$$
(as $\log(1+n^{1/d}) \sim d^{-1} \log(1+n)$).
If $\ell_\infty$ are the bounded sequences and $c_0$ denotes the closed subspace of sequences convergent to zero, let $\ell_\infty / c_0$ denote the quotient space. 

\begin{definition} \label{def:res}
Let $T$ be a compactly based Laplacian modulated operator with symbol $p_T$.
The linear map
$$
T \mapsto \Res(T):= \left[ \frac{d}{\log(1+n)} \int_{\mathbb R^d} \int_{|\xi|\le n^{1/d}} p_T(x,\xi) d\xi\,dx \right]
$$
we call the \emph{residue} of $T$, where $[ \cdot ]$ denotes the equivalence class in $\ell_\infty / c_0$.
\end{definition}
Note that any sequence $\Res_n(T)$, $n \in \mathbb N$, such that
\begin{equation} \label{eq:res}
\int_{\mathbb R^d} \int_{|\xi|\le n^{1/d}} p_T(x,\xi) d\xi\,dx = \frac{1}{d} \mathrm{Res}_n(T)  \log n + o(\log n)
\end{equation}
defines the residue $\Res(T) = [\mathrm{Res}_n(T)] \in \ell_\infty / c_0$.

\medskip We show that $\Res$, applied to compactly based pseudo-differential operators,
depends only on the principal symbol and extends the noncommutative residue.

\begin{example}[Noncommutative residue] \label{ex:wodres}
Let $S^m_{\mathrm{base}}$ be the symbols of the compactly based pseudo-differential operators of order $m$.
An equivalence relation is defined on symbols $p,q \in S^m_{\mathrm{base}}$
by $p \sim q$ if $p-q \in S^{m-1}_{\mathrm{base}}$.  
The \emph{principal symbol} of a compactly based pseudo-differential operator $P$ of order $m$ with symbol $p \in S^m_{\mathrm{base}}$ is the equivalence class $[p] \in S^m_{\mathrm{base}} / S^{m-1}_{\mathrm{base}}$.

\begin{lemma} \label{lemma:princemap}
Let $P$ be a compactly based pseudo-differential operator of order $-d$.
Then $\Res(P)$ depends only on the principal symbol of $P$.
\end{lemma}
\begin{proof}
By Proposition~\ref{prop:psdoincl} (the extension) $P$ is Laplacian modulated
and $\Res(P)$ is well defined.
If $p(x,\xi) \in S^{m}_{\mathrm{base}}$, $m <-d$, then $q(x,\xi) := p(x,\xi)\langle \xi \rangle^{\theta} \in S^{\mmod}$, $\theta = -d-m > 0$.
Then
$$
\int_{\mathbb R^d} \int_{|\xi|\le n^{1/d}} p(x,\xi) d\xi \, dx
= \int_{\mathbb R^d} \int_{|\xi|\le n^{1/d}} q(x,\xi) \langle \xi \rangle^{-\theta} d\xi \, dx 
= O(1)
$$
by~\eqref{compactbase4}.  It follows from~\eqref{eq:res} that the residue depends only on the equivalence class of
 a symbol $p \in S^{-d}_{\mathrm{base}}$.
 \end{proof}

The asymptotic expansion of $p \in S^m_{\mathrm{base}}$ means (for our purposes)
a sequence $\{ p_{m-j} \}_{j=0}^{\infty}$ such that $p_{m-j} \in S^{m-j}_{\mathrm{base}}$
and $p - \sum_{j=0}^{n} p_{m-j} \in S^{m-n-1}_{\mathrm{base}}$, $n \geq 0$.  A pseudo-differential
$P$ of order $m$ is \emph{classical} if its symbol $p$ has an asymptotic expansion $\{ p_{m-j} \}_{j=0}^{\infty}$ where each $p_{m-j}$
is a homogeneous function of order $m-j$ in $\xi$ except in a neighbourhood of zero.
The principal symbol of $P$ is the leading term $p_m \in S^m$ in the asymptotic expansion.
When $d > 1$ let $ds$ denote the volume form of the $d-1$-sphere
$$
\mathbb S^{d-1} = \{ \xi \in \mathbb R^d | |\xi|=1 \} ,
$$
according to radial and spherical co-ordinates of $\mathbb R^d$,
i.e.~$d\xi = r^{d-1} drds$, $\xi \in \mathbb R^d \setminus \{0 \}$, $r > 0$, $s \in \mathbb S^{d-1}$.  When $d=1$ let $S^{d-1} = \{ -1 , 1 \}$
with counting measure $ds$.

We understand the scalars to be embedded in $\ell_\infty / c_0$
as the classes $[ a_n ]$ where $a_n \in c$, i.e.~if $\lambda \in \mathbb C$
then $[a_n]=\lim_{n \to \infty} a_n = \lambda$.

\begin{proposition}[Extension of the noncommutative residue] \label{prop:wodres}
Let $P$ be a compactly based classical pseudo-differential operator of order $-d$ with principal symbol $p_{-d}$.  Then
$\mathrm{Res}(P)$ is the scalar
$$
\mathrm{Res}(P) =  \mathrm{Res}_W(P) := \int_{\mathbb R^d} \int_{\mathbb S^{d-1}} p_{-d}(x,s) ds \, dx$$
where $\mathrm{Res}_W$ denotes the noncommutative residue.
\end{proposition}
\begin{proof}
By the previous lemma we need only consider the principal symbol $p_{-d}$ of $P$, which we assume
without loss to be homogeneous for $| \xi | \geq 1$.
Then
\begin{align*}
\int_{\mathbb R^d}\int_{|\xi|\le n^{1/d}} p_{-d}(x,\xi) d\xi\,dx
&= \int_{\mathbb R^d}\int_{1 \leq |\xi|\le n^{1/d}} |\xi|^{-d} p_{-d}(x,\xi/|\xi|) d\xi\,dx + O(1) \\
&= \int_{\mathbb R^d} \int_{\mathbb S^d} p_{-d}(x,s) ds \,dx \int_1^{n^{1/d}} r^{-d} r^{d-1} dr + O(1) \\
&= \int_{\mathbb R^d} \int_{\mathbb S^d} p_{-d}(x,s) ds\,dx \; \log (n^{1/d}) + O(1) \\
&= \frac{1}{d} \int_{\mathbb R^d} \int_{\mathbb S^d} p_{-d}(x,s) ds\, dx \; \log n + O(1)
\end{align*}
The result follows from~\eqref{eq:res}.
 \end{proof}
\end{example}

Thus, the residue (as in Definition~\ref{def:res}) of a classical compactly based pseudo-differential operator of order $-d$ is a scalar and coincides with the noncommutative residue.
The residue of an arbitrary pseudo-differential operator is not always a scalar.

\begin{example}[Non-measurable pseudo-differential operators]

We construct a compactly supported pseudo-differential operator $Q$ of order $-d$
whose residue is not a scalar.
The following lemma
will simplify the construction.  The lemma is a standard result on pseudo-differential operators, but we will use it
several times.

\begin{lemma} \label{easysymbol}
Suppose $P$ is a pseudo-differential operator with symbol $p \in S^m$ and $\psi, \phi \in C_c^\infty(\mathbb R^d)$.
The compactly supported operator $Q = M_\psi P M_\phi$ has symbol $q \in S^m_{\mathrm{base}}$ such that
$q \sim \psi p \phi$.
\end{lemma}
\begin{proof}
The operator $Q := M_\psi P M_\phi$ is a pseudo-differential operator of order $m$, \cite[Corollary 3.1]{Shubin2003}.
Evidently it is compactly supported.  Let $q$ be the symbol of $Q$.
From \cite[Theorem 3.1]{Shubin2003}
$$
q(x,\xi) \sim \sum_{\alpha} \frac{(-i)^\alpha}{\alpha!} ( \partial_\xi^\alpha \partial_y^\alpha \psi(x) p(x,\xi) \phi(y) ) |_{y=x} 
$$
where the asymptotic sum runs over all multi-indices $\alpha$.  Since
$$
\sum_{|\alpha| \geq 1} \frac{(-i)^\alpha}{\alpha!} ( \partial_\xi^\alpha \partial_y^\alpha \psi(x) p(x,\xi) \phi(y) ) |_{y=x}
\in S^{m-1}_{\mathrm{base}}
$$
(see \cite[p.~3]{Shubin2003}) we obtain $q(x,\xi) -  \psi(x) p(x,\xi) \phi(x) \in S^{m-1}_{\mathrm{base}}$.
 \end{proof}

\begin{proposition} \label{prop:nonmeas}
There is a compactly supported pseudo-differential operator $Q$ of order $-d$ such that
$$
\Res(Q) = [ \sin \log \log n^{1/d} ].
$$
\end{proposition}
\begin{remark}
Obviously $\Res(Q)$ is not a scalar.
\end{remark}
\begin{proof}
Set
$$
p'(\xi) = \frac{ \sin \log \log |\xi| + \cos \log \log |\xi|}{|\xi|^d} , \qquad \xi \in \mathbb R^d, |\xi| \geq e .
$$
One confirms by calculation that $p'$ satisfies 
$$
| \partial_\xi^\alpha p'(\xi) | \leq 2. 3^{|\alpha|} (d+|\alpha|)! |\xi|^{-d - |\alpha|} , \qquad |\xi | > e .
$$
Let $g \in C_c^\infty(\mathbb R^d)$ be~$g(\xi) = g_1(\left| \xi
\right|)$ where~$g_1$ is positive and increasing such that
$$
g_1(\xi) = \left\{ \begin{array}{ll} 0 &  \xi \leq 3 \\
1 & \xi \geq 4
\end{array} \right. , \qquad \xi \in \mathbb R^d .
$$
Then $p := gp' \in S^{-d}$, and denote by $P$ the pseudo-differential operator with symbol $p$.

Let $\phi \in C_c^\infty(\mathbb R^d)$ be such that
$$
\int_{\mathbb R^d} |\phi(x)|^2 dx = (\mathrm{Vol} \, \mathbb S^{d-1})^{-1} .
$$  
If $Q$ is the operator $M_{\overline{\phi}} P M_\phi$ of Lemma~\ref{easysymbol} with symbol $q \sim \overline{\phi} p \phi$ then, provided $n \geq 4^d$,
\begin{align*}
\int_{\mathbb R^d}\int_{|\xi|\le n^{1/d}} q(x,\xi) d\xi\,dx 
&= \int_{\mathbb R^d}\int_{|\xi|\le n^{1/d}} |\phi(x)|^2 p(x,\xi) d\xi\,dx + O(1) \\
&= \int_{\mathbb R^d} |\phi(x)|^2 dx \int_{4 \leq |\xi|\le n^{1/d}} p'(\xi) d\xi + O(1) \\
&=  \int_4^{n^{1/d}} (\sin \log \log r + \cos \log \log r)r^{-d}r^{d-1}dr + O(1) \\
&= \frac{1}{d} (\sin \log \log n^{1/d} ) \log n + O(1) .
\end{align*}
The result follows from~\eqref{eq:res}.
 \end{proof}

The operator $Q$ will be used in Corollary~\ref{cor:nonmeas} to
provide an example of a non-measurable pseudo-differential operator.
\end{example}

\begin{example}[Integration of square-integrable functions]

The residue can be used to calculate the integral of a compactly supported 
square integrable function.

\begin{proposition}\label{MultSqFnPropInt}
If $f \in L_2(\mathbb R^d)$ has compact support and $\Delta$ is the Laplacian on $\mathbb R^d$
then $\Res(M_f (1-\Delta)^{-d/2})$ is the scalar
$$
\Res(M_f (1-\Delta)^{-d/2}) = \mathrm{ Vol } \, \mathbb  S^{d-1} \int_{\mathbb R^d} f(x) dx .
$$
\end{proposition}
\begin{proof}
Since $(1-\Delta)^{-d/2} = T_g$
where $g(\xi) = \langle \xi \rangle^{-d} \in L_{\mmod}(\mathbb R^d)$,
$M_f(1-\Delta)^{-d/2}$ is a compactly based Laplacian modulated operator
by Proposition~\ref{MultSqFnProp}.
Then
\begin{align*}
\int_{\mathbb R^d}\int_{|\xi|\le n^{1/d}} f(x) \langle \xi \rangle^{-d} d\xi\,dx
&= \int_{\mathbb R^d} f(x) dx \int_{1 \leq |\xi|\le n^{1/d}} |\xi|^{-d} d\xi + O(1) \\
&= \frac{1}{d} \mathrm{ Vol } \, \mathbb  S^{d-1} \int_{\mathbb R^d} f(x) dx \; \log n + O(1)
\end{align*}
The result follows from~\eqref{eq:res}.
 \end{proof}
\end{example}

\subsection{Eigenvalues of Laplacian modulated operators}

We now come to our main technical theorem.
This result is at the heart of Connes' trace theorem.

 \begin{theorem}\label{eigen}  Suppose $T:L_2(\mathbb R^d)\to L_2(\mathbb R^d)$ is compactly supported and Laplacian modulated with symbol $p_T$.  Then $T\in\mathcal L_{1,\infty}(L_2(\mathbb R^d))$ and 
 \begin{equation}\label{1000} \sum_{j=1}^n\lambda_j(T)-\frac{1}{(2\pi)^d}\int_{\mathbb R^d}\int_{|\xi|\le n^{1/d}} p_T(x,\xi)d\xi dx = O(1)\end{equation}
where $\{ \lambda_j(T) \}_{j=1}^\infty$ is any eigenvalue sequence of $T$.
\end{theorem}
 
\begin{remark}
If $T$ is only compactly based, then~\eqref{1000} still holds but it may not be true that $T\in\mathcal L_{1,\infty}$.  See the proof below.
\end{remark}

The following lemmas are required for the proof.
Let $\mathcal Q_z$ denote the unit cube on $\mathbb R^d$ centred on
$z \in \mathbb R^d$.

\begin{lemma} \label{lem:basis0}
There exists $0 < \phi \in C_c^\infty(\mathbb R^d)$
such that:
\begin{enumerate}[(i)]
\item $\phi(x)=1$, $x \in \pi \mathcal  Q_0$, $\phi(x) = 0$,
$x \not\in 2\pi \mathcal Q_0$;
\item $\{ u_m \}_{m \in \mathbb Z^d}$ form an orthonormal set in $L_2(\mathbb R^d)$ where
$$u_m(x)=\frac{1}{(2\pi)^{d/2}}\phi(x)e^{i\langle x,m\rangle}, \qquad m\in\mathbb Z^d;$$
\item for each $N \in \mathbb N$
$| \hat \phi(\xi) | \lec_N \langle \xi \rangle^{-N}$.
\end{enumerate}
\end{lemma}
\begin{proof}
(i) Let $h$ be a non-negative $C^{\infty}$-function on $\mathbb R$ such that
$$ \int_{-\infty}^{\infty} h(t)\,dt=1,$$
and for some $\delta<\pi/2$ we have $\supp h =(-\delta,\delta).$  We then define
$g= h\star\chi_{[-\pi,\pi]}.$  Then
$$
\hat g(\xi) = 2\pi \hat{h}(\xi) \mathrm{sinc}(\pi \xi) .
$$
Hence
$$ \hat g(0) =2\pi,$$
$$ \hat g(n) =0, \qquad n\in\mathbb Z\setminus\{0\},$$
$$ \mathrm{supp}(g)= (-\pi-\delta,\pi+\delta)$$ and
$$ g(t)=1, \qquad -\pi+\delta<t<\pi-\delta.$$
We also have that $\sqrt{g} \in C^{\infty}_c(\mathbb R).$  Let us define
$$ \phi(x) :=\prod_{j=1}^d \sqrt{g(x_j)}, \qquad x=(x_1,\ldots,x_d) \in \mathbb R^d.$$  Then
$$
\widehat{|\phi|^2}(\xi) :=\prod_{j=1}^d \hat g(\xi_j), \qquad \xi=(\xi_1,\ldots \xi_d) \in \mathbb R^d.
$$
and
$$ \widehat{|\phi|^2}(0)=(2\pi)^d,$$
$$ \widehat{|\phi|^2}(m)=0, \qquad m\in\mathbb Z^d\setminus\{0\},$$
$$ \mathrm{supp}(\phi) \subset [-2\pi,2\pi]^d,$$
$$ \phi(x)=1, \qquad x\in [-\pi,\pi]^d.$$

(ii) Since
$$
\int_{\mathbb R^d} \overline{u_{m_1}(x)} u_{m_2}(x) dx
= \frac{\widehat{|\phi|^2}(m_1-m_2)}{(2\pi)^{d}} \qquad m_1,m_2 \in \mathbb Z^d
$$
the family $(u_m)_{m\in\mathbb Z^d}$ is orthonormal in $L_2(\mathbb R^d)$.

(iii) Since $\phi$ is smooth and compactly supported the estimate now follows from standard results, see e.g.~\cite[Problem 8.16, p.~113]{Friedlander1998}.
 \end{proof}

\begin{lemma} \label{lem:basis}
For $n \in \mathbb N$ let $\phi_n = \phi(x/n)$.
Then
\begin{enumerate}[(i)]
\item $\phi_n(x)=1$, $x \in \pi n \mathcal  Q_0$, $\phi_n(x) = 0$,
$x \not\in 2\pi n \mathcal Q_0$;
\item $\{ u_{m,n} \}_{m \in \mathbb Z^d}$ form an orthonormal set in $L_2(\mathbb R^d)$ where
$$u_{m,n}(x)=\frac{1}{(2\pi n)^{d/2}}\phi_{n}(x)e^{i\langle x,m/n\rangle}, \qquad m\in\mathbb Z^d .$$
\end{enumerate}
\end{lemma}
\begin{proof}
(i) Note
$$ \phi(x/n) :=\prod_{j=1}^d \sqrt{g(x_j/n)}, \qquad x=(x_1,\ldots,x_d) \in \mathbb R^d.$$  Then
$$
\widehat{|\phi_n|^2}(\xi) :=\prod_{j=1}^d n^d \hat g(n\xi_j)
= n^d \widehat{|\phi|^2}(n\xi)
$$
and
$$ \widehat{|\phi_n|^2}(0)=(2\pi n)^d,$$
$$ \widehat{|\phi_{n}|^2}(m/n)=0, \qquad m\in\mathbb Z^d\setminus\{0\}.$$

(ii) Since
$$
\int_{\mathbb R^d} \overline{u_{m_1,n}} u_{m_2,n}(x) dx
= \frac{\widehat{|\phi_{n}|^2}((m_1-m_2)/n)}{(2\pi n)^{d}} \qquad m_1,m_2 \in \mathbb Z^d
$$
the family $(u_{m,n})_{m\in\mathbb Z^d}$ is orthonormal in $L_2(\mathbb R^d)$.
 \end{proof}

For $n \in \mathbb N$, let $\mathcal H_n$ denote the Hilbert space generated by $( u_{m,n} )_{m \in \mathbb R^d}$.  Let $P_n$ be the projection such that $\mathcal H_n = P_n L_2(\mathbb R^d)$.  Clearly $L_2(\pi n \mathcal Q_0) \subset \mathcal H_n$.  
Let $V_n : \mathcal H_n \to \mathcal H_n$
be the positive compact operator defined by
$$
V_n u_{m,n} = (1+|m|^2)^{-d/2} u_{m,n} , \qquad m \in \mathbb Z^d .
$$
Then
$$
T_n := P_n T P_n : \mathcal H_n \to \mathcal H_n
$$
for any bounded operator $T$.
We now show that if $T$ is Laplacian modulated then $T_n$ is $V_n$-modulated.
Let $\sigma_l$ denote the bi-dilation isometry of $L_2(\mathbb R^d,\mathbb R^d)$
to itself:
$$
(\sigma_l p)(x,\xi) = p(lx,\xi/l) , \qquad l > 0 .
$$

\begin{lemma} \label{PartLmodVmod}
Let $T$ be Laplacian modulated, with symbol
$p_T \in S^\mmod$.
Then
$$
\sum_{|m| \geq n^{1/d}} \| T u_{m,l} \|_{L_2}^2 \lec[d] \| \sigma_l p_T \|_{\mmod}^2 \; n^{-1}  , \qquad n \geq 1 .
$$
\end{lemma}
\begin{proof}
Fix $l \in \mathbb N$.  Denote $p_T$ just by $p$ and $u_{m,l}$ by $u_m$.  Note that
$$
\widehat{u_{m}}(\xi) = (2\pi l)^{-d/2} \hat{\phi_l}(\xi - m/l)
= (2\pi l)^{-d/2} l^d \hat{\phi}(l\xi - m) .
$$
Then
\begin{align*}
T u_{m}(x) &= 
(2\pi l)^{-d/2} l^d \int_{\mathbb R^d} e^{i\langle x, \xi \rangle}
p(x,\xi) \hat{\phi}(l\xi - m) d\xi \\
&= (2\pi l)^{-d/2} \int_{\mathbb R^d} e^{i\langle x, \xi/l \rangle}
p(x,\xi/l) \hat{\phi}(\xi - m) d\xi \\
&= (2\pi l)^{-d/2} \int_{\mathbb R^d} e^{i\langle x/l, \xi \rangle}
p(x,\xi/l) \hat{\phi}(\xi - m) d\xi
\end{align*}
and
\begin{align*}
|T u_{m}(x)|^2 & \leq (2\pi)^{-d} l^{-d} \left( \int_{\mathbb R^{d}} |p(x,\xi/l)| | \hat \phi(\xi - m) | d\xi \right)^2 .
\end{align*}
By Lemma~\ref{lem:basis}
$|\hat \phi (\xi)| \lec_{N} \langle \xi \rangle^{-N}$ for any integer $N$.
Hence
\begin{align} 
\| T u_m \|_{L_2}^2 &\lec_{N} (2\pi)^{-d} l^{-d}  \int_{\mathbb R^d} \left( \int_{\mathbb R^{d}} |p(x,\xi/l)| \langle \xi-m \rangle^{-N} d\xi \right)^2 dx \nonumber \\
&\lec_N (2\pi)^{-d} \int_{\mathbb R^d} \left( \int_{\mathbb R^{d}} |p(lx,\xi/l)| \langle \xi-m \rangle^{-N} d\xi \right)^2 dx \label{eq:Vmodproof1}
\end{align}
Temporarily denote $p(lx,\xi/l)$ by $p_l(x,\xi)$.

Now let $s \in \mathcal Q_m$.
Then
$$
|m-s| \leq \sqrt{ \sum_{i=1}^d (1/2)^2 } = d^{1/2}/2 .
$$
Hence
$$
\langle m-s \rangle := 1 + |m-s| \leq 1 + d^{1/2}/2 \leq 2d^{1/2} .
$$
Using Peetre's inequality,
$$
\langle x \rangle^{N} \langle y \rangle^{-N} \lec[N] \langle x - y \rangle^{N} 
, \qquad x,y \in \mathbb R^d ,
$$
we have
$$
\langle \xi - s \rangle^{N} \langle \xi -m \rangle^{-N} \lec[N] \langle \xi - s - (\xi -m) \rangle^{N} 
= \langle m - s \rangle^{N} \lec[N,d] 1
, \ \xi \in \mathbb R^d , s \in \mathcal Q_m , m \in \mathbb Z^d .
$$
Multiplying through by $\langle \xi - s \rangle^{-N}$ we obtain
\begin{equation} \label{eq:Vmodproof2}
\langle \xi -m \rangle^{-N} \lec[d,N] \langle \xi - s \rangle^{-N}
, \qquad s \in \mathcal Q_m .
\end{equation}
Substituting~\eqref{eq:Vmodproof2} into~\eqref{eq:Vmodproof1} provides
\begin{equation*}
\| T u_m \|_{L_2}^2 \lec[d,N] \int_{\mathbb R^d} \left( \int_{\mathbb R^{d}} |p_l(x,\xi)| \langle \xi-s \rangle^{-N} d\xi \right)^2 dx , \qquad s \in \mathcal Q_m ,
\end{equation*}
and hence
\begin{equation} \label{eq:Vmodproof3}
\| T u_m \|_{L_2}^2 = \int_{\mathcal Q_m} \| Tu_m \|_{L_2}^2 ds
\lec[d,N] \int_{\mathcal Q_m} \int_{\mathbb R^d} \left( \int_{\mathbb R^{d}} |p_l(x,\xi)| \langle \xi-s \rangle^{-N} d\xi \right)^2 dx ds .
\end{equation}
Now let $n > d^{d/2}$.  Then $n^{1/d}/2 > d^{1/2}/2$ and $n^{1/d} - d^{1/2}/2 < n^{1/d}/2$.  Hence
$$
|m| \geq n^{1/d} \implies |s| \geq n^{1/d}/2 , \qquad s \in \mathcal Q_m .
$$
Using~\eqref{eq:Vmodproof3} we have
\begin{align}
\sum_{|m| \geq n^{1/d}} \| T u_m \|_{L_2}^2 &
\lec[d,N] \sum_{|m| \geq n^{1/d}} \int_{\mathcal Q_m} \int_{\mathbb R^d} \left( \int_{\mathbb R^{d}} |p_l(x,\xi)| \langle \xi-s \rangle^{-N} d\xi \right)^2 dx ds  \nonumber \\
&\lec[d,N] \int_{|s| \geq n^{1/d}/2} \int_{\mathbb R^d} \left( \int_{\mathbb R^{d}} |p_l(x,\xi)| \langle \xi-s \rangle^{-N} d\xi \right)^2  dx ds . 
\label{eq:Vmodproof4}
\end{align}
We split the integrand in~\eqref{eq:Vmodproof4} into two parts according to the condition $| \xi - s | \geq |s|/2$,
\begin{align}
& \left( \int_{\mathbb R^{d}}
|p_l(x,\xi)| \langle \xi - s \rangle^{-N} d\xi \right)^2 \nonumber \\
&\lec \left( \int_{| \xi - s | \geq |s|/2}
|p_l(x,\xi)| \langle \xi - s \rangle^{-N} d\xi \right)^2
+  \left( \int_{| \xi - s | < |s|/2}
|p_l(x,\xi)| \langle \xi - s \rangle^{-N} d\xi \right)^2 \label{eq:Vmodproof5}
\end{align}
where the inequality is from $(a+b)^2 \leq 2(a^2 + b^2)$, $a,b > 0$.

We consider the first term from~\eqref{eq:Vmodproof5}.
As $|\xi-s|\geq |s|/2$, then $\langle \xi - s \rangle^{-N} \leq 2^{N} \langle s \rangle^{-N}$.
Then
\begin{align*}
& \left( \int_{| \xi - s | \geq |s|/2}
|p_l(x,\xi)| \langle \xi - s \rangle^{-N} d\xi \right)^2 \\
&\leq \int_{| \xi - s | \geq |s|/2}
|p_l(x,\xi)|^2 d\xi \int_{| \xi - s | \geq |s|/2} \langle \xi - s \rangle^{-2N} d\xi \\
&\lec \langle s \rangle^{-N} \int_{\mathbb R^d}
|p_l(x,\xi)|^2 d\xi \int_{| \xi - s | \geq |s|/2} \langle \xi - s \rangle^{-N} d\xi \\
&\leq \langle s \rangle^{-N} \int_{\mathbb R^d}
|p_l(x,\xi)|^2 d\xi \int_{\mathbb R^d} \langle \xi \rangle^{-N} d\xi .
\end{align*}
where we used the Holder inequality (assuming $N > d$).
We now set $N=2d$ and then
\begin{align*}
&\int_{|s| \geq n^{1/d}/2} \int_{\mathbb R^d} \left( \int_{| \xi - s | \geq |s|/2}
|p_l(x,\xi)| \langle \xi - s \rangle^{-N} d\xi \right)^2 dx ds \\
&\lec \| p_l \|_{L_2}^2 \int_{|s| \geq n^{1/d}/2} \langle s \rangle^{-2d} ds \\
&\lec \| p \|_{L_2}^2  n^{-1} .
\end{align*}

Now we consider the second term from~\ref{eq:Vmodproof5}.  We have
\begin{align*}
& \left( \int_{| \xi - s | < |s|/2}
|p_l(x,\xi)| \langle \xi - s \rangle^{-N} d\xi \right)^2 \\
&\leq \left( \int_{| \xi | < |s|/2}
|p_l(x,\xi+s)| \langle \xi \rangle^{-N} d\xi \right)^2  \\
&\lec \int_{| \xi | < |s|/2}
|p_l(x,\xi+s)|^2 \langle \xi \rangle^{-N} d\xi  \int_{\mathbb R^d} \langle \xi \rangle^{-N} d\xi 
&\lec \int_{| \xi | < |s|/2}
|p_l(x,\xi+s)|^2 \langle \xi \rangle^{-N} d\xi
\end{align*}
where we used the Holder inequality and assumed $N > d$.
Note that
\begin{align*}
& \int_{\mathbb R^d} \int_{\mathbb R^d} \int_{\mathbb R^d}
|p_l(x,\xi+s)|^2 \langle \xi \rangle^{-N} d\xi dx ds \\
&= \int_{\mathbb R^d} \int_{\mathbb R^d} \int_{\mathbb R^d}
|p_l(x,s)|^2 \langle \xi \rangle^{-N} d\xi dx ds \\
&\lec \| p_l \|_{L_2}^2 .
\end{align*}
Hence we can interchange the order of integration.
As $|s|/2 \geq n^{1/d}/4$ and $|\xi| < |s|/2$, then $| \xi + s | > |s|/2 \geq n^{1/d}/4$.
By Fubini's Theorem
\begin{align*}
& \int_{|s| \geq n^{1/d}/2} \int_{|\xi | < |s|/2}
|p_l(x,\xi+s)|^2 \langle \xi \rangle^{-N} d\xi  ds \\
&\leq \int_{\mathbb R^{d}} \left( \int_{|\xi+s| \geq n^{1/d}/4} |p_l(x,\xi+s)|^2 ds \right) \langle \xi \rangle^{-N} d\xi \\
&= \int_{\mathbb R^{d}} \left( \int_{|s| \geq n^{1/d}/4} |p_l(x,s)|^2 ds \right) \langle \xi \rangle^{-N} d\xi \\
&= \int_{|s| \geq n^{1/d}/4} |p_l(x,s)|^2 ds \int_{\mathbb R^{d}} \langle \xi \rangle^{-N} d\xi .
\end{align*}
By choosing $N > d$,
\begin{align*}
& \int_{|s| \geq n^{1/d}/2} \int_{\mathbb R^d} \left( \int_{| \xi - s | < |s|/2}
|p_l(x,\xi)| \langle \xi - s \rangle^{-N} d\xi \right)^2 dx ds \\
&\lec \int_{\mathbb R^d} \int_{|s| \geq n^{1/d} /4 }
|p_l(x,s)|^2 ds dx \\
&\lec \left(\sup_{n \geq \min\{4,d^{d/2}\}} (n^{1/d} /4)^{-d}  \int_{\mathbb R^d} \int_{|s| \geq n^{1/d} /4 }
|p_l(x,s)|^2 ds dx \right) n^{-1} \\
&\lec \left(\sup_{t \geq 1} t^d \int_{\mathbb R^d} \int_{|s| \geq t }
|p_l(x,s)|^2 ds dx \right) n^{-1} .
\end{align*}
The inequality of the Proposition is shown for $n \geq \min\{4,d^{d/2}\}$.
It is trivial to adjust the statement by a constant so that it holds
also for $1 \leq n < \min\{4,d^{d/2}\}$.
 \end{proof}

Fix $l \in \mathbb N$.  Let $\{ u_{m} \}_{m \in \mathbb Z^d}$ be the basis of $\mathcal H_l$.
Let $m_n$, $n \in \mathbb N$, be the Cantor enumeration of $\mathbb Z^d$.
Then $V_l u_{m_n} = (1 + |m_n|^2)^{-d/2} u_{m_n}$ is ordered so that
$(1 + |m_n|^2)^{-d/2}$ are the singular values of $V_l$.

\begin{lemma} \label{compact set modulated}
Let $T$ be Laplacian modulated with symbol $p_T$
and $l \in \mathbb N$.
Then $T_l: \mathcal H_l \to \mathcal H_l$ is $V_l$-modulated.
\end{lemma}
\begin{proof}
Let $\Lambda_k$ be the hypercube centered on $0$ of dimensions
$k^d$.  Then there exists an integer $k$ such that
$m_n \not\in \Lambda_{k}$ but $m_n \in \Lambda_{k+1}$.
By the Cantor enumeration $(k-1)^d < n \leq k^d$.  So $n^{1/d} \leq k$.
Hence, when $j > n$, $m_j$ is not in the ball of radius $(k-1)/2$ centered on $0$,
which is smaller than than the ball $n^{1/d}/4$ (ruling out the trivial case $n=1$). In summary, when $j > n$, then
$$m_j \notin \{ m \in \mathbb Z^d , |m| \leq n^{1/d}/4 \}.$$
Hence
$$
\sum_{j > n} \| T_l u_{m_j} \|_{L_2}^2 \lec
\sum_{|m| > n^{1/d}/4} \| T u_{m} \|_{L_2}^2 \lec n^{-1}
$$
by Proposition~\ref{PartLmodVmod} and $T_l$ is $V_l$-modulated by Proposition~\ref{connectioncor2}
since $V_l \in \mathcal L_{1,\infty}$ and $s_n(V_l) = (1 + |m_n|^2)^{-d/2} = \Theta(n^{-1})$.
 \end{proof}

Set $v_n = u_{m_n}$, $n \in \mathbb N$, where $m_n$ is the Cantor enumeration.

\begin{remark} \label{rem:stage1}
As $T_l$ is $V_l$-modulated where $0<V_l \in \mathcal L_{1,\infty}$
we can use Theorem~\ref{connection} to conclude that, for fixed $l \in \mathbb N$:
\begin{enumerate}[(i)]
\item $T_l \in \mathcal L_{1,\infty}$; and
\item $$
\sum_{j=1}^n \lambda_j(T_l) -  \sum_{j=1}^n (T_l v_j, v_j ) = O(1)
$$
for any eigenvalue sequence $\{ \lambda_j(T_l) \}_{j=1}^\infty$ of $T_l$.
\end{enumerate}
\end{remark}

%

We need a final lemma.

\begin{lemma} \label{eigen second part}
If $T$ is compactly based in $l \pi \mathcal Q_0$ with symbol $p_T$, then
$$
\sum_{j=1}^n ( T v_j, v_j ) - \frac{1}{(2\pi)^d} \int_{|\xi| \leq n^{1/d}} \int_{\mathbb R^d} p_T(x,\xi) d\xi \, dx = O(1) .
$$
\end{lemma}
\begin{proof}
It is clear there are constants $0<2a<1<b/2<\infty$ depending on $d$ so that $(v_j)_{j=1}^n= (u_m)_{m\in \mathbb A_n}$ where $\{|m|\le 2an^{1/d}\}\subset \mathbb A_n\subset \{|m|\le bn^{1/d}/2\}.$ Let
$$G_n(x,\xi) :=\frac{1}{(2\pi)^d}\sum_{m\in\mathbb A_n}e^{i\langle x,\xi-m\rangle}\hat\phi(\xi-m).$$
Now
$$ ( T u_m, u_m ) =\frac{1}{(2\pi)^{2d}}\int_{\mathbb R^d}\phi(x)e^{-i\langle x,m\rangle}\left(\int_{\mathbb R^d}e^{i \langle x , \xi \rangle} p_T(lx,\xi/l)\hat\phi(\xi-m)d\xi\right)dx.$$
We use the notation $p_l(x,\xi)=p_T(lx,\xi/l)$ again and note
that $p_l$ has support in a compact set $K$ within $\pi \mathcal Q_0$.
Here the double integral converges absolutely and we can apply Fubini's theorem to obtain that, since $\phi(x)=1$ for $x\in K$,
\begin{equation} \label{purposeG}
\sum_{j=1}^n ( T v_j,v_j) = \sum_{m\in \mathbb A_n} (Tu_m,u_m) = \frac{1}{(2\pi)^d}\int_{\mathbb R^d}\int_{K}G_n(x,\xi)p_l(x,\xi)dx\,d\xi.
\end{equation}
To obtain a bound
$$
 \int_{\mathbb R^d}\int_{K}G_n(x,\xi)p_l(x,\xi)dx\,d\xi
- \int_{K} \int_{|\xi| \le n^{1/d}} p_l(x,\xi) d\xi \, dx =O(1)
$$
we will compare $G_n(x,\xi)$ with the function
$H_n(\xi):=\chi_{\{|\xi|\le n^{1/d}\}}.$
For fixed $\xi$, consider the smooth periodic function
\begin{equation} \label{psi1}
\psi_\xi(x)=\sum_{m\in\mathbb Z^d}e^{-i\langle x,\xi\rangle}\phi(x+2\pi m).
\end{equation} 
For every $x$, the Fourier series,
$$\psi_\xi(x)= \sum_{m\in\mathbb Z^d} e^{i\langle x,m\rangle} \hat \psi_\xi(m) $$ 
converges, where
\begin{align*}
\hat \psi_\xi(m) &= \frac{1}{(2\pi)^d} \int_{2\pi \mathcal Q} e^{-i\langle x,m\rangle} \psi_\xi(x) dx \\
&= \frac{1}{(2\pi)^d} \int_{\mathbb R^d} e^{-i\langle x,m+\xi\rangle} \phi(x) dx
= \frac{1}{(2\pi)^d} \hat \phi(m+\xi). 
\end{align*}
Hence
\begin{equation} \label{psi2}
\psi_\xi(x)= \frac{1}{(2\pi)^d} \sum_{m\in\mathbb Z^d}e^{-i\langle x,m\rangle}\hat\phi(\xi-m).
\end{equation}
Note from~\eqref{psi1} that $e^{i\langle x , \xi \rangle } \psi_\xi(x) = 1$
for all $x \in \pi \mathcal Q_0$.  Using~\eqref{psi2} gives the formula
\begin{equation} \label{1onpiQ}
 1=\frac{1}{(2\pi)^d} \sum_{m\in\mathbb Z^d}e^{i\langle x,\xi-m\rangle} \hat\phi(\xi-m)
\end{equation}
for $x\in\pi\mathcal Q_0$.
Now suppose $|\xi|<an^{1/d}.$ Then, using~\eqref{1onpiQ},
$$
|H_n(\xi)-G_n(x,\xi)| = | 1 - G_n(x,\xi)|
\leq \frac{1}{(2\pi)^d} \sum_{m \notin A_n} |\hat \phi(\xi - m)| .
$$
If $m \notin \mathbb A_n$ then $|m| > 2a n^{1/d}$.  Hence $|\xi - m| > |\xi|$
and $|\xi - m| \geq |m|/2$.
This implies
$$
|\hat \phi(\xi - m)| \lec[N] \langle \xi-m \rangle^{-N} \lec[N] \langle m \rangle^{-N} .
$$
We choose $N > d$.
We obtain, when $|\xi|< an^{1/d}$,
\begin{equation} \label{Hpart1}
|H_n(\xi)-G_n(x,\xi)|
\lec[N] \sum_{|m| \geq 2an^{1/d}} \langle m \rangle^{-N} \lec[N] n^{1-N/d} .
\end{equation}
We also have an estimate $|G_n(x,\xi)| \lec 1$ when $x \in K$.
Hence, for $an^{1/d} \leq |\xi| \leq bn^{1/d}$,
\begin{equation} \label{Hpart2}
|H_n(\xi)-G_n(x,\xi)| \leq 1 .
\end{equation}
If $|\xi|>bn^{1/d}$ then $|m-\xi| \geq |m|$ when $m \in \mathbb A_n$.
Hence $|\xi-m| \geq |\xi|/2$  and
$$
|\hat\phi(\xi-m)|\lec[N]  \langle \xi-m \rangle^{1-d-N/d} \lec[N] \langle \xi \rangle^{1-d-N/d} , \qquad m\in\mathbb A_n
$$
Hence, for $|\xi|>bn^{1/d}$,
$$
|G_n(x,\xi)|\lec[N]  \langle \xi \rangle^{1-d-N/d} \sum_{|m| \leq bn^{1/d}} 1
\lec[N] \langle \xi \rangle^{1-d-N/d} n \lec[N] \langle \xi \rangle^{1-N/d} .
$$
Then, for $|\xi|>bn^{1/d}$,
\begin{equation} \label{Hpart3}
 |H_n(\xi) - G_n(x,\xi)| \lec[N] |\xi|^{1-N/d} .
\end{equation}
Combining~\eqref{Hpart1}-\eqref{Hpart3} we have, when $x \in K$, that
$$
|H_n(\xi) - G_n(x,\xi)| \lec[N] \left\{
\begin{array}{ll}
n^{1-N/d}  & |\xi|< an^{1/d} \\
1 & an^{1/d} \leq |\xi| \leq bn^{1/d} \\
|\xi|^{1-N/d} & |\xi| > bn^{1/d} .
\end{array}
\right.
$$
With this result we can show
\begin{equation} \label{Gpurpose2}
\int_{K}\int_{\mathbb R^d}|G_n(x,\xi)-H_n(\xi)||p_l(x,\xi)|d\xi\,dx = O(1)
\end{equation}
by considering the regions $|\xi|< an^{1/d}$,
$an^{1/d} \leq |\xi| \leq bn^{1/d}$, and $|\xi| > bn^{1/d}$.
and a choice of $N \geq (1+\epsilon)d$, $\epsilon > 0$.
Consider
\begin{align*}
&\int_{K}\int_{|\xi|<an^{1/d}}|G_n(x,\xi)-H_n(\xi)||p_l(x,\xi)|d\xi\,dx\\
&\lec[N] n^{1-N/d}\int_{K}\int_{|\xi|<an^{1/d}}|p_l(x,\xi)|d\xi\,dx\\
&\lec[N] n^{1-N/d} \log(n)\\
&\lec[N] n^{1+\epsilon-N/d}\lec 1
\end{align*}
by using~\eqref{compactbase2}.
Similarly
\begin{align*} &\int_{K}\int_{|\xi|>bn^{1/d}}|G_n(x,\xi)-H_n(\xi)||p_l(x,\xi)|d\xi\,dx\\
&\lec \int_{K}\int_{|\xi|>bn^{1/d}}|\xi|^{1-N/d}|p_l(x,\xi)|d\xi\,dx \lec 1
\end{align*}
by applying~\eqref{compactbase4} since $1-N/d < 0$. Finally,
\begin{align*} &\int_{K}\int_{an^{1/d}\le |\xi|\le bn^{1/d}}|G_n(x,\xi)-H_n(\xi)||p_l(x,\xi)|d\xi\,dx\\
&\lec \int_{K}\int_{an^{1/d}\le |\xi|\le bn^{1/d}}|p_l(x,\xi)|d\xi\,dx \lec 1
\end{align*}
by applying~\eqref{compactbase1}.  Hence~\eqref{Gpurpose2} is shown.
Combining~\eqref{purposeG} and~\eqref{Gpurpose2} shows that
$$\sum_{j=1}^n (Tv_j,v_j)-\frac{1}{(2\pi)^d}\int_{\mathbb R^d}\int_{|\xi|\le n^{1/d}}p_l(x,\xi)d\xi\,dx = O(1).$$
Applying~\eqref{compactbase5} shows that
\begin{equation*}
  \int_{\mathbb R^d}\int_{|\xi|\le n^{1/d}}p_l(x,\xi)d\xi\,dx - \int_{\mathbb R^d}\int_{|\xi|\le n^{1/d}}p(x,\xi)d\xi\,dx = O(1)
\end{equation*}
since
$$
\int_{\mathbb R^d}\int_{|\xi|\le n^{1/d}}p(lx,\xi/l)d\xi\,dx = \int_{\mathbb R^d}\int_{|l\xi|\le n^{1/d}}p(x,\xi)d\xi\,dx .
$$
The result is shown.
 \end{proof}

Taking into account Remark~\ref{rem:stage1} and Lemma~\ref{eigen second part}, we are now in a position to prove Theorem~\ref{eigen}.

\begin{proof}[Proof of Theorem~\ref{eigen}]

Suppose $T$ is compactly based and Laplacian modulated.
Choose $l \in \mathbb N$ sufficiently large so that
$0<\phi \in C_c^\infty(\mathbb R^d)$, $\phi(x) =1$
for $x \in 2 \pi l \mathcal Q_0$, and $M_\phi T = T$.
Since $M_\phi T = T$ then $S = TM_\phi = M_\phi T M_\phi$ is compactly supported and Laplacian modulated (see Remark~\ref{rem:bimodule}). 

From Lemma~\ref{compact set modulated} the operator $S=S_l$ is $V_l$-modulated
(if $T$ is compactly supported then, without loss, $T=S$ and $T \in \mathcal L_{1,\infty}$, which is the first part of the theorem).  By Remark~\ref{rem:stage1}
$$
\sum_{j=1}^n \lambda_j(S) - \sum_{j=1}^n ( S v_j, v_j ) = O(1) .
$$
However, up to the irrelevant multiplicity of zero
as an eigenvalue,
$$
\lambda_j(S) = \lambda_j(TM_\phi) = \lambda_j(M_\phi T) = \lambda_j(T) .
$$
Also, since $\phi v_j = v_j$, then
$$
(S v_j, v_j) =  (Tv_j, v_j ) .
$$
So
$$
\sum_{j=1}^n \lambda_j(T) - \sum_{j=1}^n ( T v_j, v_j) = O(1) . 
$$
The result of the Theorem now follows from Lemma~\ref{eigen second part}.
 \end{proof}

%

\subsection{Traces of Laplacian modulated operators} \label{sec:modtraces}

In this section we prove several versions of Connes' trace theorem.

\medskip Let $\mathcal H$ be a separable Hilbert space.
In \cite{Dix}, J.~Dixmier constructed a trace on the Banach ideal of compact operators
$$
\mathcal{M}_{1,\infty} := \left\{ T \in \mathcal K ( \mathcal H) \, \Big| \, \sup_{n \in \mathbb N} \frac1{\log(1+n)} \sum_{j=1}^n s_j(T) < \infty \right\}
$$
by linear extension of the weight
$$
\Tr_\omega(T) := \omega \left( \left\{ \frac{1}{\log(1+n)} \sum_{j=1}^n s_j(T)\right\}_{n=1}^\infty  \right) , \qquad T>0 .
$$
Here $\omega$ is a dilation invariant state on $\ell_\infty$. 
By the inclusion $\mathcal L_{1,\infty} \subset \mathcal M_{1,\infty}$
a Dixmier trace $\mathrm{Tr}_\omega$ restricts to a trace on
$\mathcal L_{1,\infty}$.  All Dixmier traces are normalised, meaning that
$$
\mathrm{Tr}_\omega \left( \diag \left\{ \frac{1}{n} \right\}_{n=1}^\infty \right) = 1 .
$$
The commutator subspace has been used previously to study spectral forms of the Dixmier traces, e.g.~\cite{Fackcomm2004}, \cite{AS2005}.  Despite the Lidskii theorem, Corollary~\ref{KCor}, it is not evident that
$$
\mathrm{Tr}_\omega(T) := \omega \left( \left\{ \frac{1}{\log(1+n)} \sum_{j=1}^n \lambda_j(T)\right\}_{n=1}^\infty  \right) 
$$
for an eigenvalue sequence $\{ \lambda_j(T) \}_{j=1}^\infty$ of a compact operator $T$.
By a combination of results from~\cite{ssz2010} and~\cite{AS2005} the result is true for the restriction to $\mathcal L_{1,\infty}$.

\begin{lemma} \label{DixspecialLidksii}
Suppose $T \in \mathcal L_{1,\infty}(\mathcal H)$.  Then
\begin{equation} \label{eq:DixLidskii}
\mathrm{Tr}_\omega(T) = \omega \left( \left\{ \frac{1}{\log(1+n)} \sum_{j=1}^n \lambda_j(T)\right\}_{n=1}^\infty  \right)
\end{equation}
for any eigenvalue sequence $\{ \lambda_j(T)\}_{j=1}^\infty$ of $T$,
and any dilation invariant state $\omega$.
\end{lemma}
\begin{proof}
From \cite[Remark 13]{ssz2010} it follows that
$$
\mathrm{Tr}_\omega(T) = \omega \left( \left\{ \frac{1}{\log(1+n)} \sum_{ \{ j | |\lambda_j| > 1/n \}} \lambda_j(T) \right\}_{n=1}^\infty  \right) .
$$
A similar result was obtained in~\cite{Fackcomm2004} and~\cite{AS2005}, but
not for all dilation invariant states.  The argument
of~\cite[Corollary 2.12]{AS2005} that, when $T \in \mathcal L_{1,\infty}$, 
$$
\sum_{ \{ j | |\lambda_j| > 1/n \}} \lambda_j(T) - \sum_{j=1}^n
  \lambda_j(T) = O(1)
$$
provides the result.
 \end{proof}

Connes' trace theorem, \cite{C3}, states that a Dixmier trace applied to a compactly supported
classical pseudo-differential operator $P$ of order $-d$ yields the noncommutative residue up to a constant,
\begin{equation} \label{CTTstatement}
\mathrm{Tr}_\omega(P) = \frac{1}{d (2\pi)^d} \mathrm{Res}_W(P) .
\end{equation}
Connes' statement was given for closed manifolds, but it is equivalent to~\eqref{CTTstatement}.
We shall consider manifolds in our final section.

The main result of our paper is the generalisation of Connes' trace theorem below.
We note that a dilation invariant state $\omega$ is a generalised limit, i.e.~it vanishes
on $c_0$.  Hence
$$
\omega([c_n]) := \omega( \{ c_n \}_{n=1}^\infty ) , \qquad \{ c_n \}_{n=1}^\infty  \in \ell_\infty
$$
is well-defined as a linear functional on $\ell_\infty / c_0$.  

\begin{theorem}[Trace theorem] \label{CTTv1}
Suppose $T$ is compactly based and Laplacian modulated such that $T \in \mathcal L_{1,\infty}(L_2(\mathbb R^d))$.  Then:
\begin{enumerate}[(i)]
\item
$$
\mathrm{Tr}_\omega(T) = \frac{1}{d(2\pi)^d} \omega(\mathrm{Res}(T))
$$
where $\Res(T) \in \ell_\infty / c_0$ is the residue of $T$;
\item
$$
\mathrm{Tr}_\omega(T) = \frac{1}{d(2\pi)^d} \mathrm{Res}(T)
$$
for every Dixmier trace $\Tr_\omega$ iff $\Res(T)$ is scalar;
\item
\begin{equation*}
\tau(T)= \frac{\tau \circ \diag \left( \{ \frac1n \}_{n=1}^\infty \right)}{(2\pi)^d d} \Res(T) 
\end{equation*}
for every trace $\tau : \mathcal L_{1,\infty}(L_2(\mathbb R^d)) \to \mathbb C$
iff
\begin{equation} \label{classical}
\int_{\mathbb R^d} \int_{|\xi|\le n^{1/d}} p_T(x,\xi) d\xi\,dx = \frac1d \Res(T) \, \log n + O(1)
\end{equation}
for a scalar $\Res(T)$.
\end{enumerate}
\end{theorem}
\begin{proof}
(i) By Theorem~\ref{eigen} and the formula~\eqref{eq:res} we have that
\begin{equation} \label{thecrux}
\frac{1}{\log(1+n)} \sum_{j=1}^n \lambda_j(T) = \frac{1}{d(2\pi)^d} \mathrm{Res}_n(T) + o(1) 
\end{equation}
for any eigenvalue sequence $\{ \lambda_j(T) \}_{j=1}^\infty$ and any representative
$\Res_n(T)$ of the equivalence class $\Res(T)$.  By the condition that $T \in \mathcal L_{1,\infty}$
we apply Lemma~\ref{DixspecialLidksii} and obtain
\begin{align*}
\mathrm{Tr}_\omega(T) &:= \omega \left( \left\{ \frac{1}{\log(1+n)} \sum_{j=1}^n \lambda_j(T)\right\}_{n=1}^\infty  \right) \\
&= \omega \left( \frac{1}{d(2\pi)^d} \mathrm{Res}_n(T) + o(1)  \right)
= \frac{1}{d(2\pi)^d} \omega \left( \Res(T) \right)
\end{align*}
since $\omega$ vanishes on sequences convergent to zero.

(ii) As $T \in \mathcal L_{1,\infty}$, by \cite[Theorem 20]{SS2012}
$\Tr_\omega(T)$ is the same
value for all Dixmier traces if and only if $\Tr_\omega(T)
= \lim_{n \to \infty} \frac{1}{\log(1+n)} \sum_{j=1}^n \lambda_j(T)$
and the limit on the right exists.  From~\eqref{thecrux}
$\lim_{n \to \infty} \frac{1}{\log(1+n)} \sum_{j=1}^n \lambda_j(T)$
exists if and only if $\Res(T)$ is scalar and
$\lim_{n \to \infty} \frac{1}{\log(1+n)} \sum_{j=1}^n \lambda_j(T) = (2\pi)^{-d} d^{-1} \Res(T)$.

(iii)  Suppose $T$ satisfies~\eqref{classical}.  Then, by
Theorem~\ref{eigen},
$$
\sum_{k=1}^n\lambda_k(T) - \frac{1}{(2\pi)^d d} \Res(T) \sum_{k=1}^n \frac1k  = O(1) .
$$
By Theorem~\ref{K},
$$
T - \diag \left\{ \frac{1}{(2\pi)^d d} \Res(T) \frac1k \right\}_{k=1}^\infty \in \Com \mathcal L_{1,\infty}.
$$
Conversely, suppose the previous display is given. Then, by Theorem~\ref{K},
for some decreasing sequence $\nu \in \ell_{1,\infty}$, 
$$
\left| \sum_{k=1}^n\lambda_k(T) - \frac{1}{(2\pi)^d d} \Res(T) \sum_{k=1}^n \frac1k  \right| \leq n \nu_n \lec 1 ,
$$
and the symbol $p_T$ of $T$ satisfies~\eqref{classical} by Theorem~\ref{eigen}.
 \end{proof}

Theorem~\ref{CTTv1} allows the Dixmier trace of any compactly based pseudo-differential operator of order $-d$
to be computed from its symbol.

\begin{corollary} \label{cor:genpsdo}
If $P$ is a compactly based pseudo-differential operator of order $-d$, then
$P \in \mathcal L_{1,\infty}(L_2(\mathbb R^d))$ and
$$
\mathrm{Tr}_\omega(P) = \frac{1}{d(2\pi)^d} \omega(\mathrm{Res}(P))
$$
for any dilation invariant state $\omega$.  Here $\Res(P)$ is given by
Definition~\ref{def:res}.
\end{corollary}
\begin{proof}
If $P$ is compactly based, there exists $0 < \phi \in C_c^\infty(\mathbb R^d)$
with $M_\phi P =P$.  The operator $P$
and $P' = M_\phi P M_\phi$ are Laplacian modulated by Proposition~\ref{prop:psdoincl}.
Since $P'$ is compactly supported then $P' \in \mathcal L_{1,\infty}$ by Theorem~\ref{eigen}.
Note that $P - P' = [M_\phi, P] \in G^{-\infty}$ is a smoothing operator belonging
to the Shubin class, see \cite[IV]{Shubin2003}, and the extensions of $G^{-\infty}$
belong to the trace class operators $\mathcal L_1$ on $\mathbb R^d$, \cite[Section 27]{Shubin2003}, see also \cite{Arsu2008}.
Hence $P = P' + (P-P') \in \mathcal L_{1,\infty}$.  Now apply Theorem~\ref{CTTv1}.
 \end{proof}

Not every pseudo-differential operator is measurable in Connes' sense, \cite[\S 4]{C}.

\begin{corollary}[Non-measurable pseudo-differential operators] \label{cor:nonmeas}
There is a compactly supported pseudo-differential operator $Q$ of order $-d$ such that
the value $\mathrm{Tr}_\omega(Q)$ depends on the dilation invariant
state $\omega$.\footnote{Such operators are called {\it non-measurable}.}
\end{corollary}
\begin{proof}
Let $Q$ be the operator from Proposition~\ref{prop:nonmeas}.
Since $\mathrm{Res}(Q) = [ \sin \log \log n^{1/d} ]$
is not a scalar, $\mathrm{Tr}_\omega(Q)$ depends on the state $\omega$
by Theorem~\ref{CTTv1}(ii).
 \end{proof}

Contrary to the case of general pseudo-differential operators of order $-d$,
the next corollary shows that the classical pseudo-differential operators of order $-d$ have unique trace.

\begin{corollary}[Connes' trace theorem] \label{cor:CTTRd}
Suppose $P$ is a compactly based classical pseudo-differential operator of order $-d$
with noncommutative residue $\mathrm{Res}_W(P)$.
Then $P \in \mathcal L_{1,\infty}(L_2(\mathbb R^d))$ and
$$
\tau(P) = \frac{1}{d(2\pi)^d} \mathrm{Res}_W(P) ,
$$
for every trace $\tau$ on $\mathcal L_{1,\infty}$ such that $\tau(\diag \{ n^{-1} \}_{n=1}^\infty ) = 1$.
\end{corollary}
\begin{proof}
By Corollary~\ref{cor:genpsdo}, $P \in \mathcal L_{1,\infty}$.
Proposition~\ref{prop:wodres} and its proof shows both that $\mathrm{Res}(P) = \mathrm{Res}_W(P)$
and that~\eqref{classical} is satisfied. The result now follows from
Theorem~\ref{CTTv1}.
 \end{proof}

\begin{remark}
Corollary~\ref{cor:nonmeas} indicates that the qualifier classical cannot be omitted from the statement of Connes' trace theorem. 
\end{remark}

\begin{remark}
Corollary~\ref{cor:CTTRd} is stronger than Connes' original theorem.
To see that the set of Dixmier traces restricted to $\mathcal L_{1,\infty}$ is smaller than the set of arbitrary normalised traces on $\mathcal L_{1,\infty}$, consider that the positive part
of the common kernel of Dixmier traces is exactly the positive part of the separable subspace $(\mathcal L^0_{1,\infty})^+$,
\cite{LSS}, while the positive part of the common kernel
of arbitrary normalised traces is exactly $\mathcal L_1^+ = (\Com \mathcal L_{1,\infty})^+$, easily seen from Theorem~\ref{DFWW} or see \cite{KaftalWeiss2022}.
\end{remark}

\begin{corollary}[Integration of square integrable functions] \label{L2trace}
Let $f \in L_{2}(\mathbb R^d)$ be compactly supported.  Then
$M_f (1-\Delta)^{-d/2} \in \mathcal L_{1,\infty}(L_2(\mathbb
R^d)),$
and
$$
\tau(M_f (1-\Delta)^{-d/2}) = \frac{\mathrm{Vol} \, \mathbb S^{d-1} }{d(2\pi)^d} \int_{\mathbb R^d} f(x)dx ,
$$
for every trace $\tau$ on $\mathcal L_{1,\infty}$ such that $\tau(\diag \{ n^{-1} \}_{n=1}^\infty ) = 1$.
\end{corollary}
\begin{proof}
  It follows from~\cite[\S~5.7]{BirmanSolomyak1991} that~$M_f (1-\Delta)^{-d/2} \in \mathcal L_{1,\infty}$.  On the other hand, from
  Proposition~\ref{MultSqFnProp}, $M_f (1-\Delta)^{-d/2}$ is
  compactly based and Laplacian modulated.
  From the proof of Proposition~\ref{MultSqFnPropInt}
  $\Res(M_f (1-\Delta)^{-d/2})$
  is scalar, equal to $\mathrm{Vol} \,
  \mathbb S^{d-1} \int_{\mathbb R^d} f(x)dx$, and~\eqref{classical} is satisfied.
  From Theorem~\ref{CTTv1}, $$ \tau(M_f (1-\Delta)^{-d/2}) =
  \frac{\tau \circ \diag \left( \{ \frac{1}{n} \}_{n=1}^\infty \right)
    \mathrm{Vol} \, \mathbb S^{d-1} }{d(2\pi)^d} \int_{\mathbb R^d}
  f(x)dx $$ for every trace on $\mathcal L_{1,\infty}$.
 \end{proof}

We now transfer our notions and results to the setting of closed Riemannian manifolds.

\section{Closed Riemannian manifolds} \label{sec:CTT}

In this section we introduce the notion of a Hodge-Laplacian modulated operator.

\begin{notation}
Henceforth $X$ will always denote a $d$-dimensional closed Riemannian manifold $(X,g)$ with metric $g$,
and $\Delta_g$ denotes the Laplace-Beltrami operator with respect to $g$, \cite[p.~3]{Chavel1984}.
\end{notation}

\begin{definition} \label{def:HodgeL}
A bounded operator $T: L_2(X) \to L_2(X)$ is \emph{Hodge-Laplacian modulated} if it is
$(1-\Delta_g)^{-d/2}$-modulated for some metric $g$.
\end{definition}

This definition is independent of the choice of metric $g$.

\begin{lemma} \label{lem:7.1}
Suppose $T$ is a bounded operator $T: L_2(X) \to L_2(X)$.  If $T$ is $(1-\Delta_{g_1})^{-d/2}$-modulated
then it is $(1-\Delta_{g_2})^{-d/2}$-modulated for any pair of metrics $g_1$ and $g_2$.  
\end{lemma}
\begin{proof}
The operator $(1-\Delta_{g_1})^{-d/2}(1-\Delta_{g_2})^{d/2}$ is a zero-order
pseudo-differential operator on $X$, see \cite[\S 4]{Shubin2003} for pseudo-differential operators on manifolds. Hence it is has a bounded extension, \cite[\S 6.4]{Shubin2003}, and there exists a constant $C$ such that
$$
\| (1-\Delta_{g_1})^{-d/2} f \|_{L_2} \leq C \| (1-\Delta_{g_2})^{-d/2} f \|_{L_2}
$$
for $f \in L_2(X)$.  By
Proposition~\ref{changeV} any $(1-\Delta_{g_1})^{-d/2}$-modulated operator
is $(1-\Delta_{g_2})^{-d/2}$-modulated.
 \end{proof}

The positive bounded operator $(1-\Delta_g)^{-d/2} : L_2(X) \to L_2(X)$ is a compact operator (alternatively $\Delta_g$ has compact resolvent, \cite[p.~8]{Chavel1984}).

Thus
there exists an orthonormal basis $(e_n)_{n=1}^\infty$
of eigenvectors
$$-\Delta_{g} e_n = s_n e_n,  \qquad n \in \mathbb N ,$$
ordered such that the eigenvalues $s_1 \leq s_n \leq \ldots$ are increasing.

Also, by Weyl's asymptotic formula, \cite[p.~9]{Chavel1984},
$$
s_n^{-d/2} \sim l_d \, n^{-1}
$$
for a constant $l_d$.  Therefore
$$
(1-\Delta_g)^{-d/2} \in \mathcal L_{1,\infty} .
$$
Due to these spectral properties we can invoke
Theorem~\ref{connection} and arrive directly at a trace theorem by making the following definition
of the residue.

\begin{lemma} \label{cor:defresman}
If $T$ is a Hodge-Laplacian modulated operator
and $(e_n)_{n=1}^\infty$ is the above eigenvector sequence of the Laplace-Beltrami operator, then
$$
\sum_{j=1}^n (Te_j,e_j) = O(\log(1+n)) .
$$
\end{lemma}
\begin{proof}
From Theorem~\ref{connection}(i) $T \in \mathcal L_{1,\infty}$. Then $\mu_n(T) \lec n^{-1}$ and
$$
\left| \sum_{j=1}^n \lambda_j(T) \right| \leq \sum_{j=1}^n \mu_j(T) \lec \log(1+n) .
$$
By Theorem~\ref{connection}(iii)
$$
\sum_{j=1}^n (Te_j,e_j) = \sum_{j=1}^n \lambda_j(T) + O(1)
$$
and the result follows.
 \end{proof}

\begin{definition} \label{def:defresman}
If $T$ is a Hodge-Laplacian modulated operator the class
\begin{equation} \label{eq:defresman}
\Res(T) := d(2\pi)^d \left[ \left\{ \frac{1}{\log(1+n)} \sum_{j=1}^n (Te_j,e_j) \right\}_{n=1}^\infty \right]
\end{equation}
is called the \emph{residue} of $T$, where $[\cdot]$ denotes an equivalence class in $\ell_\infty / c_0$,
and $(e_n)_{n=1}^\infty$ is the above eigenvector sequence of the Laplace-Beltrami operator.
\end{definition}

\begin{remark} \label{rem:localres}
The residue is evidently linear and vanishes on trace class Hodge-Laplacian operators.  The proof of Lemma~\ref{cor:defresman} shows that the residue is independent of the metric, since the eigenvalue sequence of an operator compact operator $T \in \mathcal K (L_2(X))$
does not depend on the metric.
\end{remark}

\begin{theorem}[Trace theorem for closed manifolds] \label{connectioncorman}
Let $T$ be Hodge-Laplacian modulated.
Then $T \in \mathcal L_{1,\infty}(L_2(X))$ and:
\begin{enumerate}[(i)]
\item 
$$
\mathrm{Tr}_\omega(T) = \frac{1}{d(2\pi)^d} \omega(\mathrm{Res}(T))
$$
for any Dixmier trace $\Tr_\omega$ where $\Res(T) \in \ell_\infty / c_0$ is the residue of $T$;
\item
$$
\mathrm{Tr}_\omega(T) = \frac{1}{d(2\pi)^d} \mathrm{Res}(T)
$$
for every Dixmier trace $\Tr_\omega$ iff $\Res(T)$ is scalar;
\item
$$
\tau(T)= \frac{\tau \circ \diag \left( \{ \frac1n \}_{n=1}^\infty \right)}{(2\pi)^d d} \Res(T) 
$$
for every trace $\tau : \mathcal L_{1,\infty}(L_2(\mathbb R^d)) \to \mathbb C$
iff
\begin{equation} \label{classicalman}
\sum_{j=1}^n(Te_j,e_j) = \frac{1}{d(2\pi)^d} \Res(T) \log(1+n) + O(1)
\end{equation}
for a scalar $\Res(T)$.
\end{enumerate}
\end{theorem}
\begin{proof}
The proof is omitted since it is identical to the proof of Theorem~\ref{CTTv1}
with the use of Theorem~\ref{eigen} replaced exactly by Theorem~\ref{connection}(iii).
 \end{proof}

To obtain the same corollaries of this trace theorem for closed manifolds as we did for Theorem~\ref{CTTv1} for $\mathbb R^d$,
we identify the residue locally with the residue on $\mathbb R^d$.

\subsection{Localised Hodge-Laplacian modulated operators}

We emulate the usual treatment of pseudo-differential
operators (e.g. \cite[\S 4]{Shubin2003}), in that the symbol
of a Hodge-Laplacian modulated operator is defined locally by the restriction to a chart and then
patched together using a partition of unity.
We extend Theorem~\ref{eigen} to a statement
involving the cotangent bundle. 

Without loss we let $X$ be covered by charts $\{ (U_i,h_i) \}_{i=1}^N$
such that $h_i(U_i)$ is bounded in $\mathbb R^d$.
For a chart $(U,h)$ belonging to such an atlas, define $W_h:L_2(h(U))\to L_2(U)$ by $W_h f= f \circ h$ and $W_h^{-1}: L_2(U)\to L_2(h(U))$ by $W_h^{-1} f= f \circ h^{-1}$.
Then $W_h, W_h^{-1}$ are bounded. 

Our first step is to confirm that
Hodge-Laplacian modulated operators are locally Laplacian modulated operators.
If $\phi \in C^\infty(X)$ is a smooth function we denote by $M_\phi$ the multiplication
operator $(M_\phi f)(x) = \phi(x) f(x)$, $f \in L_2(X)$, and note
that it is a pseudo-differential operator of order $0$.

\begin{proposition} \label{local}
Suppose $(U,h)$ is a chart of $X$ with $h(U)$ bounded
and $\phi,\psi \in C^\infty(X)$ such that $\phi,\psi$ have support in $U$.
\begin{enumerate}[(i)]
\item If $T$ is a Hodge-Laplacian modulated operator then
$$W_h^{-1} M_\psi T W_h  M_{\phi \circ h^{-1}} : L_2(\mathbb R^d) \to L_2(\mathbb R^d)$$ is a compactly supported Laplacian modulated operator.
\item If $T'$ is a Laplacian modulated operator 
then $$W_h M_{\psi \circ h^{-1}} T' W_h^{-1} M_{\phi}  : L_2(X) \to L_2(X)$$ is a Hodge-Laplacian modulated operator.
\end{enumerate}
\end{proposition}
\begin{proof}
Let $H^s(X)$ denote the Sobolev spaces, $s\in \mathbb R$, on $X$.
If $V$ is the pseudo-differential
operator with local symbol $(1+|\xi|^2)^{1/2}$ the Sobolev norms are defined by
$\| f \|_s := \| V^s f \|_{L_2}$, \cite[\S 7]{Shubin2003}.
We recall that $P : H^s(X) \to  H^{s}(X)$
is continuous, \cite[Theorem 7.3]{Shubin2003}, for any zero-order pseudo-differential
operator $P$ and $s \in \mathbb R$.  Since $(1-\Delta_g)^{-s/2} V^{s}$ and $V^{-s}(1-\Delta_g)^{s/2} $, $s \in \mathbb R$, are zero-order pseudo-differential operators on $X$ they have bounded extensions.

(i) Let $\tilde{T} = W_h^{-1} M_\psi T W_h  M_{\phi \circ h^{-1}}$.
Let $\rho \in C_c^\infty(\mathbb R^d)$ be such that $\rho (\psi \circ h^{-1}) = \psi \circ h^{-1}$
and $\rho (\phi \circ h^{-1}) = \phi \circ h^{-1}$.
Then $M_\rho \tilde{T} M_\rho = \tilde{T}$ and $\tilde{T}$ is compactly supported.

Note that
\begin{align*}
\|(1-\Delta_g)^{-d/2} W_h M_{\phi \circ h^{-1}} f \|_{L_2}
&= \| (1-\Delta_g)^{-d/2} \phi f \circ h \|_{L_2} \\
&\lec \| V^{-d} \phi f \circ h \|_{L_2} \\
&\lec \| (1-\Delta)^{-d/2} \phi \circ h^{-1} f \|_{L_2} \\
&\lec \| (1-\Delta)^{-d/2} M_{\phi \circ h^{-1}} f \|_{L_2} \\
&\lec \| (1-\Delta)^{-d/2} f \|_{L_2} ,\qquad f \in L_2(\mathbb R^d)
\end{align*}
since the Sobolev norms on $H^s(U)$ and $H^s(h(U))$ are equivalent norms 
and $M_{\phi \circ h^{-1}} : H^{-d}(\mathbb R^d) \to H^{-d}(\mathbb R^d)$ is bounded.
By Proposition~\ref{changeV}, $\tilde{T} = W_h^{-1} M_\psi T (W_h M_{\phi \circ h^{-1}})$ is Laplacian modulated.

(ii) We reverse the argument in (i).
Note that
\begin{align*}
\|(1-\Delta)^{-d/2} W_h^{-1} M_{\phi} f \|_{L_2}
&= \| (1-\Delta)^{-d/2} (\phi f) \circ h^{-1} \|_{L_2} \\
&\lec \| V^{-d} \phi f \|_{L_2} \\
&\lec \| V^{-d} f \|_{L_2} \\
&\lec \| (1-\Delta_g)^{-d/2} f \|_{L_2}, \qquad f \in L_2(X)
\end{align*}
since the Sobolev norms on $H^s(h(U))$ and $H^s(U)$ are equivalent
and $M_{\phi} : H^s(X) \to H^s(X)$ is bounded.
By Proposition~\ref{changeV}, $W_h M_{\psi \circ h^{-1}} T' (W_h^{-1} M_{\phi}) $ is Hodge-Laplacian modulated.
 \end{proof}

\begin{lemma} \label{lem:bimoduleman}
If $T$ is a Hodge-Laplacian modulated operator and $P_1,P_2$ are zero-order pseudo-differential
operators on $X$, then $P_1 T P_2$ is a Hodge-Laplacian modulated operator.
\end{lemma}
\begin{proof}
Since $P_1,P_2 : L_2(X) \to L_2(X)$ are bounded, $P_2 : H^s(X) \to H^{s}(X)$, $s \in \mathbb R$, is bounded, and
$$
\| (1-\Delta_{g})^{-d/2} P_2 f \|_{L_2} \lec \| (1-\Delta_{g})^{-d/2} f \|_{L_2} , \qquad f \in L_2(X) ,
$$
it follows that $P_1 TP_2$ is Hodge-Laplacian modulated by Proposition~\ref{changeV}.
 \end{proof}

We can now define the coordinate dependent symbol of a Hodge-Laplacian modulated operator.

If $(U,h)$ is a chart of $X$, let $\sum_{j,k=1}^d g_{jk}(x)dx_j dx_k$ denote the local co-ordinates of the metric for $x \in U$ and $G(x)=[ g_{jk}(x) ]_{j,k=1}^d$.
The $d \times d$ matrix $G(x)$, $x \in U$, is positive and the determinant $|G(x)|$ is a smooth function on $X$.

Let $\{ (U_i,h_i) \}_{i=1}^N$ be an atlas of $X$ where $h_i(U_i) \subset \mathbb R^d$.   Let $\Psi :=\{ \psi_j \}_{j=1}^M$ be a smooth partition of unity so that if $K_j:=\overline{\mathrm{supp}(\psi_j)}$ then $K_{j}\cap K_{j'}\neq \emptyset$ implies that there exists an $i\in \{ 1, \dots , N\}$ with $K_j\cup K_{j'}\subset U_i$.  We will
always assume such a partition of unity.

Let $T$ be a Hodge-Laplacian modulated operator.  Set
$$
T_{jj'} := M_{\psi_j} T M_{\psi_{j'}} .
$$
When $K_{j}\cap K_{j'}\neq \emptyset$,
$$
T^h_{jj'} := W_{h_i}^{-1}  M_{\psi_j} T W_{h_i} M_{\psi_{j'} \circ h_i^{-1}},
$$
where some $i\in \{ 1, \ldots , N\}$ is chosen such that $K_j\cup K_{j'}\subset U_i$,
is compactly supported and Laplacian modulated by Proposition~\ref{local}(ii).
Let $p_{jj'}^h \in L_2(\mathbb R^d \times \mathbb R^d)$ be the symbol of $T^h_{jj'}$.  Define 
$$
p_{j j'}(x,\xi) :=  p_{j j'}^h(h_i(x), G^{-1/2}(x) \xi)
$$
for each $(x,\xi) \in T^*U_i \cong U_i \times \mathbb R^d$.
If $K_{j}\cap K_{j'} = \emptyset$ set $p_{j j'} \equiv 0$.

We thus define a chart dependent function on the cotangent bundle $T^*(X,g)$,
which we call the \emph{coordinate dependent symbol} of $T$, by
$$
p^{(\Psi,g)}_T := \sum_{j,j'=1}^M p_{j j'}  .
$$

\begin{definition}
If $T_{jj'} \in \Com \mathcal L_{1,\infty}(L_2(X))$ when $K_{j}\cap K_{j'} = \emptyset$,
then we say $T$ is \emph{$\Psi$-localised}.
\end{definition}

A codisc bundle $D^*(r)(X,g)$ of radius $r > 0$ is the subbundle of the cotangent bundle $T^*(X,g)$
with fibre over $x \in X$ given by
$$
D_x^*(r)(X,g) \cong \{ \xi \in \mathbb R^d | |G^{-1/2}(x)\xi| \leq r \}.
$$
Let $dv$ be the density on $T^*(X,g)$ which corresponds locally
to $dG^{-1/2}(x)\xi dx$.

\begin{theorem} \label{localeigen}
If $(X,g)$ is a closed $d$-dimensional Riemannian manifold, $T$ is Hodge-Laplacian modulated,
and $T$ is $\Psi$-localised with respect to an atlas and partition of unity $\Psi$ as above,
then 
\begin{equation} \label{eq:localeigen}
\sum_{k=1}^n \lambda_k(T) - \frac{1}{(2\pi)^d}\int_{D^*(n^{1/d})(X,g)} p^{(\Psi,g)}_T(v) dv = O(1).
\end{equation}
\end{theorem}
\begin{proof}
Let $\Psi = \{ \psi_j \}_{j=1}^M$ and set $T_{jj'} := M_{\psi_j} T M_{\psi_{j'}}$.
By Lemma~\ref{lem:bimoduleman} $T_{jj'}$ is Hodge-Laplacian modulated such that
$T = \sum_{j,j'=1}^M T_{jj'}$.
We recall from Corollary~\ref{connectioncor} that
\begin{equation} \label{eq:local1}
\sum_{k=1}^n \lambda_k(T) - \sum_{j,j'=1}^M \sum_{k=1}^n \lambda_k(T_{jj'}) = O(1) .
\end{equation}
By the assumption that $T$ is $\Psi$-localised then $T_{jj'} \in \Com \mathcal L_{1,\infty}$ when $K_j \cap K_{j'} = \emptyset$.  Hence
$$
\sum_{k=1}^n \lambda_k(T_{jj'}) = O(1)
$$
by Theorem~\ref{K} when $K_j \cap K_{j'} = \emptyset$.
Thus \eqref{eq:local1} is valid when removing those $T_{jj'}$ such that $K_j \cap K_{j'} = \emptyset$.

When $K_j \cap K_{j'} \not= \emptyset$,
$$T_{jj'} : L_2(U_i) \to L_2(U_i)$$
and so
$$T^h_{jj'} : L^2(h_i(U_i)) \to L^2(h_i(U_i)),$$
for the chosen $i \in \{1 \, \cdots , N \}$ such that $K_j \cap K_{j'} \subset U_i$, is compactly supported and Laplacian modulated by Proposition~\ref{local}(i).
Since the Hilbert spaces $L_2(U_i)$ and $L_2(h_i(U_i))$ are equivalent
it is an easy result that
\begin{equation} \label{eq:local2}
\lambda_k(T_{jj'}) = \lambda_k(T^h_{jj'})
\end{equation}
with the same multiplicity and the same ordering.

Note that,
\begin{multline*}
\int_{D^*(n^{1/d})(X,g)} (\chi_{K_j}p_{j j'}\chi_{K_{j'}})(v) dv  \\
= \int_{h_i(U_i)} \int_{|G^{-1/2}(x) \xi | \leq n^{1/d}} p^h_{j j'}(x,G^{-1/2}(x)\xi)  dG^{-1/2}(x)\xi dx \\
= \int_{\mathbb R^d} \int_{| \xi | \leq n^{1/d}} p^h_{j j'}(x, \xi) d\xi dx .
\end{multline*}
Therefore 
\begin{equation} \label{eq:local3}
\int_{D^*(n^{1/d})(X,g)} p^{(\Psi,g)}_T(v) dv
= \sum_{j,j} \int_{\mathbb R^d} \int_{| \xi | \leq n^{1/d}} p^h_{j j'}(x, \xi) d\xi dx 
\end{equation}
where the sum is over those $j,j'$ with $K_j \cap K_{j'} \not= \emptyset$.

Finally, using~\eqref{eq:local2} and~\eqref{eq:local3}, 
and where the sums are over those $j,j'$ with $K_j \cap K_{j'} \not= \emptyset$,
\begin{align*}
& \sum_{k=1}^n \lambda_k(T) - \frac{1}{(2\pi)^d}\int_{D^*(n^{1/d})(X,g)} p^{(\Psi,g)}_T(v) dv  \\
& \ \ \ \ \ \ = \left(  \sum_{k=1}^n \lambda_k(T) - \sum_{j,j} \sum_{k=1}^n \lambda_k(T_{jj'}) \right) \\
& \ \ \ \ \ \ \ \ \ \ \ \ + \sum_{j,j} \left( \sum_{k=1}^n \lambda_k(T^h_{jj'})   - \frac{1}{(2\pi)^d}\int_{\mathbb R^d} \int_{| \xi | \leq n^{1/d}} p^h_{j j'}(x, \xi) dx d\xi \right)  \\
& \ \ \ \ \ \ = O(1) + O(1)
\end{align*}
by~\eqref{eq:local1} and by Theorem~\ref{eigen}.
 \end{proof}

Suppose $T$ and $U$ are Hodge-Laplacian modulated operators.
If $p^{(\Psi,g_1)}_T$ represents a coordinate dependent symbol of $T$ with respect to
a metric $g_1$ and some atlas and partition of unity $\Psi$ as above, and $p^{(\Omega,g_2)}_U$ a coordinate dependent symbol of $U$
with respect to a metric $g_2$ and some atlas and partition of unity $\Omega$ as above, then
we write $p^{(\Psi,g_1)}_T \sim p^{(\Omega,g_2)}_U$ if
$$
\int_{D^*(n^{1/d})(X,g_1)} p^{(\Psi,g_1)}_T(v) dv -
\int_{D^*(n^{1/d})(X,g_2)}  p^{(\Omega,g_2)}_U(v)  dv = O(1) .
$$
The relation $\sim$ is easily checked to be an equivalence relation on coordinate dependent symbols.

\begin{definition}
A Hodge-Laplacian modulated operator $T$ is called
\emph{localised} if $\phi T \psi \in \Com \mathcal L_{1,\infty}(L_2(X))$
for every pair of functions $\phi, \psi \in C^\infty(X)$ with $\phi \psi = 0$.
\end{definition}

The next result says that every localised Hodge-Laplacian modulated operator
can be assigned a coordinate and metric independent ``principal'' symbol.

\begin{corollary} \label{cor:7.3}
Let $T$ be a localised Hodge-Laplacian modulated operator.
Then $p_T^{(\Psi,g_1)} \sim p_T^{(\Omega,g_2)}$ for every coordinate dependent symbol of $T$.
\end{corollary}
\begin{proof}
If $T$ is localised then $T$ is $\Psi$-localised and $\Omega$-localised.
The eigenvalues of $T$ are coordinate and metric independent, therefore
\begin{multline*}
\frac{1}{(2\pi)^d} \left| \int_{D^*(n^{1/d})(X,g_1)} p^{(\Psi,g_1)}_T(v) dv - \int_{D^*(n^{1/d})(X,g_2)} p^{(\Omega,g_2)}_T(v) dv \right|\\
\leq \left| \sum_{k=1}^n \lambda_k(T) - \frac{1}{(2\pi)^d}\int_{D^*(n^{1/d})(X,g_1)} p^{(\Psi,g_1)}_T(v) dv \right| \\
+ \left| \sum_{k=1}^n \lambda_k(T) - \frac{1}{(2\pi)^d}\int_{D^*(n^{1/d})(X,g_2)} p^{(\Omega,g_2)}_T(v) dv \right| . 
\end{multline*}
The result follows by Theorem~\ref{localeigen}.
 \end{proof}

Due to the result of Theorem~\ref{localeigen} we will only be concerned with coordinate dependent symbols up to the equivalence $\sim$.
As a result of Lemma~\ref{lem:7.1} and Corollary~\ref{cor:7.3} we may fix a metric $g$
and we may take any coordinate dependent symbol to act as a representative for the
symbol when discussing localised Hodge-Laplacian modulated operators.  To this end we let
$p_T$ denote a coordinate dependent symbol $p_T^{(\Psi,g)}$, dropping explicit reference to the coordinates and the metric.

We can now prove the desired result that links the residue for Hodge-Laplacian operators
to a formula involving the ``principal'' symbol.

\begin{theorem} \label{globalres}
If $(X,g)$ is a closed $d$-dimensional Riemannian manifold and $T$ is localised and Hodge-Laplacian modulated with symbol $p_T$, then 
$$
\Res(T) = \left[ \frac{d}{\log(1+n)} \int_{D^*(n^{1/d})(X,g)} p_T(v) dv \right]
$$
where $[\cdot ]$ denotes an equivalence class in $\ell_\infty/ c_0$,
or, more specifically,
\begin{equation} \label{eq:globalres}
\sum_{j=1}^n (Te_j,e_j) = \frac{1}{(2\pi)^d}\int_{D^*(n^{1/d})(X,g)} p_T(v) dv + O(1).
\end{equation}
\end{theorem}
\begin{proof}
The first display clearly follows from the second display.
By Theorem~\ref{localeigen}
$$
\sum_{k=1}^n \lambda_k(T) - \frac{1}{(2\pi)^d}\int_{D^*(n^{1/d})(X,g)} p_T(v) dv = O(1).
$$
and by Theorem~\ref{connection}(iii)
$$
\sum_{k=1}^n \lambda_k(T) - \sum_{j=1}^n (Te_j,e_j) = O(1).
$$
The result is shown.
 \end{proof}

\subsection{Residues of Hodge-Laplacian modulated operators}

We give examples of Hodge-Laplacian modulated operators and compute their residue
using Theorem~\ref{globalres}.

\begin{example}[Pseudo-differential operators]

To show a pseudo-differential operator
$P : C^\infty(X) \to C^\infty(X)$ of order $-d$
is Hodge-Laplacian modulated we use the following lemma.

\begin{lemma} \label{lem:VinV}
The operator $(1-\Delta_g)^{-d/2}$ is Hodge-Laplacian modulated.
\end{lemma}
\begin{proof}
For brevity, set $V:=(1-\Delta_g)^{-d/2}$.
Let $\{ s_n \}_{n=1}^{\infty}$ be the singular values of $V$, where
$s_n \leq C n^{-1}$ by Weyl's formula for a constant $C > 0$, as explained.  Then
$$
\|V(1+tV)^{-1}\|_{\mathcal L_2}^2 = \sum_{n=1}^\infty s_n^2 (1+t s_n)^{-2}
\leq \sum_{n=1}^\infty (C^{-1}n+t)^{-2} \lec t^{-1}, \qquad t \geq 1 .
$$
Hence $V$ is $V$-modulated.
 \end{proof}

\begin{proposition} \label{ConnesPSDO}
Let $P : C^\infty(X) \to C^\infty(X)$ be a pseudo-differential operator
of order $-d$.  Then the extension $P: L_2(X) \to L_2(X)$ is localised
and Hodge-Laplacian modulated.
\end{proposition}
\begin{proof}
For brevity set $V:=(1-\Delta_g)^{-d/2}$.
Then $V$ is Hodge Laplacian modulated.
Thus $P_0V$ is Hodge-Laplacian modulated for every
zero-order pseudo-differential operator $P_0$ by Lemma~\ref{lem:bimoduleman}.
By the pseudo-differential calculus $PV^{-1}$
is zero order, \cite[\S 3, \S 6]{Shubin2003}, hence
$P = (PV^{-1})V$ is Hodge-Laplacian modulated.

The operator $P$ is localised since, $M_\psi P M_\phi \in \mathcal L_1 \subset \Com \mathcal L_{1,\infty}$ when $\psi,\phi \in C^\infty(X)$ and $\psi \phi = 0$ by the definition of pseudo-differential operators
on a manifold, \cite[\S 4, \S 27]{Shubin2003}.
 \end{proof}

It follows from Theorem~\ref{globalres} that the residue of
a pseudo-differential operator of order $-d$ can be calculated
from its principal symbol.
\end{example}

\begin{example}[Noncommutative residue]
The cosphere bundle $S^*X$ of $(X,g)$ is the subbundle of $T^*(X,g)$ 
with fibre over $x \in X$ given by
$$
S_x^*X \cong \{ \xi \in \mathbb R^d | |G^{-1/2}(x)\xi| = 1 \}.
$$
The cosphere bundle has a density $ds$ equating locally to $dx ds_x$
where $ds_x$ is the volume element of the fibre.

\begin{proposition} \label{prop:wodres2}
Let $P: C^\infty(X) \to C^\infty(X)$ be a classical pseudo-differential
of order $-d$ and with principal symbol $p_{-d}$.  Then $P$
is localised and Hodge-Laplacian modulated and the residue of $P$ is the scalar value
$$
\mathrm{Res}(P) =  \mathrm{Res}_W(P) := \int_{S^*X} p_{-d}(s) ds.
$$
where $\mathrm{Res}_W$ denotes the noncommutative residue.
\end{proposition}
\begin{proof}
By Proposition~\ref{ConnesPSDO} $P$ is localised and Hodge-Laplacian modulated.  Both $\mathrm{Res}$ and $\mathrm{Res}_W$ evidently depend
only on the principal symbol and hence we can work locally.
The result then follows immediately from Proposition~\ref{prop:wodres}.  Note
also that this implies, from the proof of Theorem~\ref{globalres} and
Proposition~\ref{prop:wodres}, that
$$
\sum_{j=1}^{n} (Pe_j,e_j) = \frac{1}{d(2\pi)^d} \int_{S^*X} p_{-d}(s) ds \log n + O(1) .
$$
 \end{proof}

An immediate corollary (the proof is omitted) is the following
spectral formulation of the noncommutative residue of classical
pseudo-differential operators.  The first equality was observed by Fack, \cite[p.~359]{Fackcomm2004}, and proven in \cite[Corollary 2.14]{AS2005}.

\begin{corollary}[Spectral formula for the noncommutative residue] \label{cor:wodresspectral}
Let $P : C^\infty(X) \to C^\infty(X)$ be a classical pseduo-differential
operator of order $-d$, $\{ \lambda_n(P) \}_{n=1}^\infty$
denote the eigenvalues of $P$ ordered so that $|\lambda_n(P)|$ is decreasing,
and $(e_n)_{n=1}^\infty$ an orthonormal basis of eigenvectors
of the Hodge-Laplacian, $-\Delta_{g} e_n = s_n e_n$, $n \in \mathbb N$, ordered such that 
the eigenvalues $s_1 \leq s_n \leq \ldots$ are increasing.
Then, if $\mathrm{Res}_W$ is the noncommutative residue,
$$
d^{-1} (2\pi)^{-d}  \Res_W(P) = \lim_{n \to \infty} \frac{1}{\log(1+n)} \sum_{j=1}^n \lambda_j(P)
= \lim_{n \to \infty} \frac{1}{\log(1+n)} \sum_{j=1}^n (Pe_j,e_j) .
$$
\end{corollary}
\end{example}

\begin{example}[Integration of square integrable functions]

If $f \in L_2(X)$ let $M_f : L_\infty(X) \to L_2(X)$ be defined by $(M_f h)(x) = f(x) h(x)$, $h \in L_\infty(X)$.

\begin{proposition} \label{prop:7.6}
If $f \in L_2(X)$ then there is a localised Hodge-Laplacian modulated
operator $T_f$ such that $M_f (1-\Delta_g)^{-d/2} - T_f \in \mathcal L_1$ 
and
$$
\Res(T_f)  = \mathrm{Vol}\, \mathbb S^{d-1} \int_{X,g} f(x)\,dx .
$$
\end{proposition}
\begin{proof}
Let $\{ \psi_j \}_{j=1}^M$
be a partition of unity as in the previous section.
For brevity, let $V:=(1-\Delta_g)^{-d/2}$.  
Set $V_{jj'}:=M_{\psi_j} V M_{\psi_{j'}}$ and $V^h_{jj'} = W^{-1}_{h_i} M_{\psi_j} V  W_{h_i} M_{\psi_{j'} \circ h^{-1}}$.  Set $T_{jj'}:=M_f V_{jj'}$
and $T^h_{jj'} = M_{f \circ h^{-1}} V^h_{jj'}$.

If $K_j\cap K_{j'} \neq \emptyset$ we can find a chart $(U_i,h_i)$ so that $K_j\cup K_{j'} \subset U_i.$   Note that $V^h_{jj'}$ is a pseudo-differential
operator of order $-d$ on $\mathbb R^d$ compactly supported in $h_i(U_i)$.
Then
$$
T^h_{j j'} = M_{f \circ h^{-1}} (1-\Delta)^{-d/2} P_0
$$
where $P_0 = (1-\Delta)^{d/2} V^h_{jj'}$ is zero-order.
By Proposition~\ref{MultSqFnProp}
$M_{f \circ h^{-1}} (1-\Delta)^{-d/2}$ is Laplacian modulated
and by Remark~\ref{rem:bimodule} $M_{f \circ h^{-1}} (1-\Delta)^{-d/2} P_0$
is Laplacian modulated.
Hence $T^h_{j j'}$ is Laplacian modulated
and compactly supported in $h(U_i)$. It follows from Proposition~\ref{local}(ii)
that $T_{jj'}$ is Hodge-Laplacian modulated.  

If $K_j\cap K_{j'}= \emptyset$ the operator $M_{\overline{\psi_{j'}}} V M_{\overline{\psi_j}} : L_1(X)\to C^\infty(X)\subset L_2(X)$ has a smooth kernel and is hence nuclear. Since $M_{f}:L_2\to L_1$ is bounded we obtain $M_{\overline{\psi_{j'}}} V M_{\overline{\psi_j}} M_{f} \in \mathcal L_1.$   By taking the adjoint
$T_{jj'} \in \mathcal L_1$.

Let $T_f = T := \sum_{jj'} T_{jj'}$ where $K_j\cap K_{j'} \neq \emptyset$.
Then $T$ is Hodge-Laplacian modulated and obviously localised.
Set $S := \sum_{jj'} T_{jj'}$ where $K_j\cap K_{j'} = \emptyset$.  Then
$S \in \mathcal L_1.$  We have that $M_f (1-\Delta_g)^{-d/2} = T + S$
so the first statement is shown.

\medskip Since $T$ is localised we need only work locally to determine the residue.
If $K_j\cap K_{j'} \neq \emptyset$ so that $K_j\cup K_{j'} \subset U_i$
we examine the compactly supported Laplacian modulated operator $T^h_{jj'}:L_2(\mathbb R^d)\to L_2(\mathbb R^d)$.
The symbol $p^h_{jj'}$ of $T^h_{jj'}$ is given by
$p^h_{jj'}(x,\xi) = f \circ h^{-1}(x)q_{jj'}(x,\xi)$
where $q_{jj'}(x,\xi)$ is the symbol of the pseudo-differential operator
$M_{\psi_{j}} V M_{\psi_{j'}}$ in local co-ordinates.  We recall
that
$$q_{jj'}(x,\xi) - (\psi_j \psi_{j'})(h^{-1}(x))|G^{1/2}(x)\xi|^{-d} \in S^{-d-1}_{\mathrm{base}}$$
by Lemma~\ref{easysymbol}.
Since $f \in L_2(X) \subset L_1(X)$, we have
\begin{align*}
& \int_{\mathbb R^d} \int_{|\xi| \leq n^{1/d}}
\left| p^h_{jj'}(x,\xi) - (f\psi_j \psi_{j'})(h^{-1}(x))|G^{1/2}(x)\xi|^{-d} \right| dx \, d\xi \\
&\lec \int_{\mathbb R^d} \int_{|\xi| \leq n^{1/d}}
|f(x)| \langle \xi \rangle^{-d-1} dx \, d\xi \lec \| f \|_{L_1} .
\end{align*}
Thus
\begin{multline*}
\int_{\mathbb R^d} \int_{|\xi| \leq n^{1/d}}
p^h_{jj'}(x,\xi) dx \, d\xi \\
= \int_{\mathbb R^d} \int_{| \xi | \leq n^{1/d}} (f\psi_j\psi_j')(h^{-1}(x)) |G(x)^{-1/2}\xi|^{-d} d\xi dx +O(1) \\
= \int_{\mathbb R^d} \int_{| G(x)^{1/2}\xi | \leq n^{1/d}} (f\psi_j\psi_j')(h^{-1}(x)) |G(x)|^{1/2} |\xi|^{-d} dx d\xi + O(1) \\
= \int_{\mathbb R^d}  \int_{|\xi | \leq n^{1/d}} (f\psi_j\psi_j')(h^{-1}(x)) |G(x)|^{1/2} |\xi|^{-d} dx d\xi +O(1) \\
= \frac{\mathrm{Vol} \, \mathbb S^{d-1}}{d}\int_{X,g} f(x) \psi_j(x)\psi_{j'}(x) dx \log n + O(1). 
\end{multline*}
where, in the second last equality, we used~\eqref{compactbase5}.
Hence
\begin{align*}
\Res(T) &= \sum_{j,j'} \Res(T_{jj'}) \\
&= \mathrm{Vol} \, \mathbb S^{d-1} \int_{X,g} f(x) \left(\sum_{j,j} \psi_j(x)\psi_{j'}(x) \right) dx
= \mathrm{Vol} \, \mathbb S^{d-1} \int_{X,g}f(x)dx.
\end{align*}
 \end{proof}
\end{example}

\subsection{Traces of Localised Hodge-Laplacian modulated operators} \label{sec:tracethmman}

In this section we obtain Connes' trace theorem and other results for
closed Riemannian manifolds as corollaries of Theorem~\ref{connectioncorman}.

\begin{corollary}[Connes' trace theorem] \label{CTToriginalmanifold}
Let $(X,g)$ be a closed $d$-dimensional Riemannian manifold.  Suppose $P : C^\infty(X) \to C^\infty(X)$
is a classical pseudo-differential operator of order $-d$ with noncommutative residue $\mathrm{Res}_W(P)$.  Then (the extension)
$P \in \mathcal{L}_{1,\infty}(L_2(X))$ and
$$
\tau(P) = \frac{1}{d(2\pi)^d} \mathrm{Res}_W(P) ,
$$
for every trace $\tau$ on $\mathcal L_{1,\infty}$ such that $\tau(\diag \{ n^{-1} \}_{n=1}^\infty ) = 1$.
\end{corollary}
\begin{proof}
That $P$ is localised and Hodge-Laplacian, and satisfies~\eqref{classicalman} for the scalar
$\Res_W(P)$, is given by Proposition~\ref{prop:wodres2}.
The results follow from Theorem~\ref{connectioncorman}.
 \end{proof}

As before, the qualifier classical cannot be omitted from Connes' trace theorem.

\begin{corollary}[Non-measurable pseudo-differential operators] \label{cor:nonmeasX}
Let $(X,g)$ be a closed $d$-dimensional Riemannian manifold.
There exists a pseudo-differential operator $Q' : C^\infty(X) \to C^\infty(X)$ of order
$-d$ such that the value $\mathrm{Tr}_\omega(Q')$ depends on the dilation invariant state $\omega$. 
\end{corollary}
\begin{proof}
Let $Q'$ be such that $Q'$ vanishes outside a chart $(U,h)$,
and in local coordinates $Q'$ is the operator $Q$ (suitably scaled) of Corollary~\ref{cor:nonmeas}.  Then the value
$\mathrm{Tr}_\omega(Q')$ depends on the
state $\omega$.
 \end{proof}

The final result is a stronger variant of one of our results in~\cite{LDS}.  In the cited
paper we showed the following result for Dixmier traces
associated to zeta function residues.  The proof employing the methods of this paper is completely different.

\begin{corollary}[see \cite{LDS}, Theorem 2.5] \label{cor:LDS} 
Let $(X,g)$ be a closed $d$-dimensional Riemannian manifold.
Let $f \in L_2(X)$ and $M_f : L_\infty(X) \to L_2(X)$ be defined by $(M_f h)(x) = f(x) h(x)$, $h \in L_\infty(X)$.  Then $M_f (1-\Delta_g)^{-d/2} \in \mathcal L_{1,\infty}(L_2(X))$
and
$$
\tau(M_f (1-\Delta_g)^{-d/2}) = \frac{\mathrm{Vol} \, \mathbb S^{d-1} }{d(2\pi)^d} \int_{X,g} f(x)dx ,
$$
for any trace $\tau$ on $\mathcal L_{1,\infty}$ such that $\tau(\diag \{ n^{-1} \}_{n=1}^\infty ) = 1$ .
\end{corollary}
\begin{proof}
Proposition~\ref{prop:7.6} provides the result that
$M_f (1-\Delta_g)^{-d/2} = T_f + S$ where $T_f \in \mathcal L_{1,\infty}$
(by Theorem~\ref{connectioncorman} since $T_f$ is Hodge-Laplacian modulated)
and $S \in \mathcal L_1$.  Hence $M_f (1-\Delta_g)^{-d/2} \in \mathcal L_{1,\infty}$.
Also
$$
\tau(M_f (1-\Delta_g)^{-d/2}) =  \tau(T_f)
$$
for every trace $\tau$ on $\mathcal L_{1,\infty}$. 
Note that, for the operator $T_f$, the equation~\eqref{classicalman} is satisfied for the scalar
$\mathrm{Vol} \, \mathbb S^{d-1} \int_{X,g} f(x)dx$ by the proof Proposition~\ref{prop:7.6}.
By Theorem~\ref{connectioncorman}
$$\tau(T_f) = \frac{\mathrm{Vol} \, \mathbb S^{d-1} }{d(2\pi)^d} \int_{X,g} f(x)dx$$
for every $\tau$.
 \end{proof}

\begin{remark}
Our final remark is that the residue of Hodge-Laplacian modulated
operators, Definition~\ref{def:defresman}, is an extensive generalisation
of the noncommutative residue. Definition~\ref{def:HodgeL}
is a global definition requiring no reference to local behaviour.
Therefore ``non-local'' Hodge-Laplacian modulated operators
can exist and they admit a residue and, in theory, calculable trace.
Whether there are any interesting
possibilities behind this observation we do not know yet,
but it is a distinction to singular traces of pseudo-differential operators.
\end{remark}

The authors thank Dmitriy Zanin for close reading of the text
and for valuable comments and suggestions.


\end{document}